\documentclass[letter,11pt,reqno]{amsart}

\usepackage[utf8]{inputenc}

\usepackage{etex}

\usepackage{xcolor}

\definecolor{verydarkblue}{rgb}{0,0,0.5}

\usepackage[
    breaklinks,
    colorlinks,
    citecolor=verydarkblue,
    linkcolor=verydarkblue,
    urlcolor=verydarkblue,
    pagebackref=true,
    hyperindex
]{hyperref}

\backrefenglish

\usepackage{fancyhdr}

\usepackage[
    hmargin=1in,
    top=1in,
    bottom=0.8in,
    includefoot,
    headheight=13pt,
    centering,
]{geometry}

\usepackage{setspace}
\usepackage{amsmath}
\usepackage{amsthm}
\usepackage{amssymb}
\usepackage{mathtools}
\usepackage{mathdots}
\usepackage[alphabetic]{amsrefs}

\usepackage{tikz}
\usetikzlibrary{arrows,decorations.pathmorphing,decorations.markings,backgrounds,positioning,fit,matrix,patterns,calc}

\ifnum\pdfshellescape=1
    \usetikzlibrary{external}
\fi

\usepackage{manfnt}
\usepackage{marginnote}
\reversemarginpar

\newtheorem{thm}{Theorem}[section]

\newtheorem{lemma}[thm]{Lemma}

\newtheorem{coro}[thm]{Corollary}
\newtheorem{prop}[thm]{Proposition}
\newtheorem{claim}[thm]{Claim}
\newtheorem{thmA}{Theorem}

\newcommand{\thistheoremname}{}
\newtheorem*{genericthm}{\thistheoremname}
\newenvironment{namedthm}[1]
  {\renewcommand{\thistheoremname}{#1}%
   \begin{genericthm}}
  {\end{genericthm}}
\theoremstyle{definition}
\newtheorem{defn}[thm]{Definition}
\newtheorem{prob}[thm]{Problem}
\theoremstyle{remark}

\newcommand{\margindbend}{%
    \makebox[0pt][r]{\small\raisebox{-2pt}{\textdbend}\hspace{1em}}%
}

\newenvironment{ex}{\refstepcounter{thm}\begin{proof}[Example\! \emph{\thethm}]}{\end{proof}}
\newenvironment{rem}{\refstepcounter{thm}\begin{proof}[Remark \emph{\thethm}]}{\end{proof}}
\newenvironment{warn}{\refstepcounter{thm}\begin{proof}[{\margindbend}Warning \emph{\thethm}]}{\end{proof}}
\newenvironment{cons}{\refstepcounter{thm}\begin{proof}[Construction \emph{\thethm}]}{\end{proof}}
\newenvironment{nota}{\refstepcounter{thm}\begin{proof}[Notation \emph{\thethm}]}{\end{proof}}
\numberwithin{equation}{section}

\IfFileExists{./article-style.tex}{

\pagestyle{fancy}
\fancyhead{}
\fancyfoot{}
\fancyhead[LE,RO]{\small \thepage}
\fancyhead[RE]{\small \nouppercase{\rightmark}}
\fancyhead[LO]{\small \nouppercase{\leftmark}}

\setcounter{tocdepth}{1}

\makeatletter

\global\BR@BackrefAlttrue

\renewcommand*{\backrefalt}[4]{%
    \tiny%
    (%
    \ifcase #1 not cited%
          \or cit.~on~p.~#2%
          \else cit.~on~pp.~#2%
    \fi%
    )%
}

\def\print@backrefs#1{%
    \space\SentenceSpace%
    \begingroup%
        \expandafter\providecommand\csname brc@#1\endcsname{0}%
        \expandafter\providecommand\csname brcd@#1\endcsname{0}%
        \expandafter\backrefalt%
            \csname brc@#1\expandafter\endcsname%
            \csname brl@#1\expandafter\endcsname%
            \csname brcd@#1\expandafter\endcsname%
            \csname brld@#1\endcsname%
    \endgroup%
}

\def\maketitle{\par
  \@topnum\z@ 
  \@setcopyright
  \thispagestyle{empty}
  \ifx\@empty\shortauthors \let\shortauthors\shorttitle
  \else \andify\shortauthors
  \fi
  \@maketitle@hook
  \begingroup
  \@maketitle
  \toks@\@xp{\shortauthors}\@temptokena\@xp{\shorttitle}%
  \toks4{\def\\{ \ignorespaces}}
  \edef\@tempa{%
    \@nx\markboth{\the\toks4
      \@nx\MakeUppercase{\the\toks@}}{\the\@temptokena}}%
  \@tempa
  \endgroup
  \c@footnote\z@
    \renewcommand{\footnoterule}{%
      \kern -3pt
      \hrule width \textwidth height .5pt
      \kern 2pt
    }
  {
    \renewcommand\thefootnote{}
    \vspace{-2em}
    \footnote{
      \par\vspace{-1.2em}\noindent%
      \setlength{\parindent}{0pt}%
      \def\@footnotetext##1{\noindent{\footnotesize##1}\par}%
      \let\@makefnmark\relax  \let\@thefnmark\relax
      \ifx\@empty\@date\else \@footnotetext{\@setdate}\fi
      \ifx\@empty\@subjclass\else \@footnotetext{\@setsubjclass}\fi
      \ifx\@empty\@keywords\else \@footnotetext{\@setkeywords}\fi
      \ifx\@empty\thankses\else \@footnotetext{%
        \@setthanks}%
      \fi
    }
    \addtocounter{footnote}{-1}
  }
  \@cleartopmattertags
}

\def\@adminfootnotes{\@empty}

\def\@settitle{\begin{center}%
  \baselineskip14\p@\relax
    \bfseries
\Large
  \@title
  \end{center}%
}

\def\@setauthors{%
  \begingroup
  \def\thanks{\protect\thanks@warning}%
  \trivlist
  \centering\footnotesize \@topsep30\p@\relax
  \advance\@topsep by -\baselineskip
  \item\relax
  \author@andify\authors
  \def\\{\protect\linebreak}%
  \large{\authors}%
  \ifx\@empty\contribs
  \else
    ,\penalty-3 \space \@setcontribs
    \@closetoccontribs
  \fi
  \endtrivlist
  \endgroup
}

\def\@setaddresses{\par
  \nobreak \begingroup
\footnotesize
  \def\author##1{\end{minipage}\hskip 2em \begin{minipage}[t]{.5\textwidth minus
  1em}\raggedright%
    ~\\[2em]{\bf##1}\\[.5em]%
  }%
  \interlinepenalty\@M
  \def\address##1##2{\begingroup
    {\ignorespaces##2}\endgroup\\[.5em]}%
  \def\curraddr##1##2{\begingroup
    \@ifnotempty{##2}{\nobreak\indent\curraddrname
      \@ifnotempty{##1}{, \ignorespaces##1\unskip}\/:\space
      ##2\par}\endgroup}%
  \def\email##1##2{\begingroup
    \@ifnotempty{##2}{\nobreak\indent
      \@ifnotempty{##1}{, \ignorespaces##1\unskip}
      \ttfamily##2\par}\endgroup}%
  \def\urladdr##1##2{\begingroup
    \def~{\char`\~}%
    \@ifnotempty{##2}{\nobreak\indent\urladdrname
      \@ifnotempty{##1}{, \ignorespaces##1\unskip}\/:\space
      \ttfamily##2\par}\endgroup}%
  \setlength{\parindent}{0pt}%
  \vfill%
  {
  \hskip -2em%
  \begin{minipage}{0mm}
  \addresses
  \end{minipage}
  }
  \endgroup
}

\renewcommand{\author}[2][]{%
  \ifx\@empty\authors
    \gdef\authors{#2}%
    \g@addto@macro\addresses{\author{#2}}%
  \else
    \g@addto@macro\authors{\and#2}%
    \g@addto@macro\addresses{\author{#2}}%
  \fi
  \@ifnotempty{#1}{%
    \ifx\@empty\shortauthors
      \gdef\shortauthors{#1}%
    \else
      \g@addto@macro\shortauthors{\and#1}%
    \fi
  }%
}
\edef\author{\@nx\@dblarg
  \@xp\@nx\csname\string\author\endcsname}

\def\@secnumfont{\@empty}

\def\section{\@startsection{section}{1}%
  \z@{.7\linespacing\@plus\linespacing}{.5\linespacing}%
  {\large\bfseries\centering}}

\DefineAdditiveKey{bib}{secondauthor}{\name}
\DefineSimpleKey{bib}{archivePrefix}
\DefineSimpleKey{bib}{eprinttype}
\DefineSimpleKey{bib}{eprintclass}
\DefineSimpleKey{bib}{primaryClass}

\newcommand{\my@ifeq}[3]{%
    \edef\my@a{#1}%
    \edef\my@b{#2}%
    \ifx\my@a\my@b#3\fi%
}

\newcommand{\@doititle@doi}[1]{%
    \href%
        {https://doi.org/\csname bib'doi\endcsname}%
        {\textit{#1}}%
}

\newcommand{\@doititle@url}[1]{%
    \href%
        {\csname bib'url\endcsname}%
        {\textit{#1}}%
}

\newcommand{\@doititle@mr}[1]{{%
    \def\MR##1{##1}%
    \def\fld@elt##1{##1}%
    \href%
        {http://www.ams.org/mathscinet-getitem?mr=\csname bib'review\endcsname}%
        {\textit{#1}}%
}}

\newcommand{\my@eprint}[1]{%
    \IfEmptyBibField{eprinttype}{%
        \IfEmptyBibField{archivePrefix}{%
            \edef\my@etype{arXiv}%
        }{%
            \edef\my@etype{\csname bib'archivePrefix\endcsname}%
        }%
    }{%
        \edef\my@etype{\csname bib'eprinttype\endcsname}%
    }%
    \def\my@arxiv{%
        \href%
            {https://arxiv.org/abs/#1}%
            {\tt arXiv:#1}%
    }%
    \my@ifeq{\my@etype}{arxiv}{\my@arxiv}%
    \my@ifeq{\my@etype}{arXiv}{\my@arxiv}%
    \my@ifeq{\my@etype}{ArXiv}{\my@arxiv}%
    \def\my@hal{%
        \href%
            {https://hal.archives-ouvertes.fr/#1}%
            {\tt #1}%
    }%
    \my@ifeq{\my@etype}{hal}{\my@hal}%
    \my@ifeq{\my@etype}{HAL}{\my@hal}%
    \my@ifeq{\my@etype}{Hal}{\my@hal}%
}

\newcommand{\@doititle@eprint}[1]{%
    \IfEmptyBibField{eprinttype}{%
        \IfEmptyBibField{archivePrefix}{%
            \edef\my@etype{arXiv}%
        }{%
            \edef\my@etype{\csname bib'archivePrefix\endcsname}%
        }%
    }{%
        \edef\my@etype{\csname bib'eprinttype\endcsname}%
    }%
    \def\my@arxiv{%
        \href%
            {https://arxiv.org/abs/\csname bib'eprint\endcsname}%
            {\textit{#1}}%
    }%
    \my@ifeq{\my@etype}{arxiv}{\my@arxiv}%
    \my@ifeq{\my@etype}{arXiv}{\my@arxiv}%
    \my@ifeq{\my@etype}{ArXiv}{\my@arxiv}%
    \def\my@hal{%
        \href%
            {https://hal.archives-ouvertes.fr/\csname bib'eprint\endcsname}%
            {\textit{#1}}%
    }%
    \my@ifeq{\my@etype}{hal}{\my@hal}%
    \my@ifeq{\my@etype}{HAL}{\my@hal}%
    \my@ifeq{\my@etype}{Hal}{\my@hal}%
}

\newcommand{\@doititle}[1]{%
    \IfEmptyBibField{doi}{%
        \IfEmptyBibField{url}{%
            \IfEmptyBibField{review}{%
                \IfEmptyBibField{eprint}{%
                    \let\@tempa\textit
                }{%
                    \let\@tempa\@doititle@eprint
                }
            }{%
                \let\@tempa\@doititle@mr
            }%
        }{%
            \let\@tempa\@doititle@url
        }%
    }{%
        \let\@tempa\@doititle@doi
    }%
    \@tempa{#1}%
}

\newcommand{\PrintSecondAuthors}[1]{%
    \ifx\previous@primary\current@primary
        \@empty
    \else
        \PrintNames{, with }{}{#1}%
    \fi
}

\BibSpec{article}{%
    +{}  {\PrintAuthors}                {author}
    +{}  {\PrintSecondAuthors}          {secondauthor}
    +{,} { \@doititle}                  {title}
    +{.} { }                            {part}
    +{.} { \@doititle}                  {subtitle}
    +{,} { \PrintContributions}         {contribution}
    +{.} { \PrintPartials}              {partial}
    +{,} { }                            {journal}
    +{}  { \textbf}                     {volume}
    +{}  { \PrintDatePV}                {date}
    +{,} { \issuetext}                  {number}
    +{,} { \eprintpages}                {pages}
    +{,} { }                            {status}
    +{,} { available at \my@eprint}     {eprint}
    +{}  { \PrintTranslation}           {translation}
    +{;} { \PrintReprint}               {reprint}
    +{.} { }                            {note}
    +{.} {}                             {transition}
}

\BibSpec{book}{%
    +{}  {\PrintPrimary}                {transition}
    +{,} { \@doititle}                  {title}
    +{.} { }                            {part}
    +{.} { \@doititle}                  {subtitle}
    +{,} { \PrintEdition}               {edition}
    +{}  { \PrintEditorsB}              {editor}
    +{,} { \PrintTranslatorsC}          {translator}
    +{,} { \PrintContributions}         {contribution}
    +{,} { }                            {series}
    +{,} { \voltext}                    {volume}
    +{,} { }                            {publisher}
    +{,} { }                            {organization}
    +{,} { }                            {address}
    +{,} { \PrintDateB}                 {date}
    +{,} { }                            {status}
    +{}  { \parenthesize}               {language}
    +{}  { \PrintTranslation}           {translation}
    +{;} { \PrintReprint}               {reprint}
    +{.} { }                            {note}
    +{.} {}                             {transition}
}

\BibSpec{collection.article}{%
    +{}  {\PrintAuthors}                {author}
    +{,} { \@doititle}                  {title}
    +{.} { }                            {part}
    +{.} { \@doititle}                  {subtitle}
    +{,} { \PrintContributions}         {contribution}
    +{,} { \PrintConference}            {conference}
    +{}  {\PrintBook}                   {book}
    +{,} { }                            {booktitle}
    +{,} { \PrintDateB}                 {date}
    +{,} { pp.~}                        {pages}
    +{,} { }                            {status}
    +{,} { available at \eprint}        {eprint}
    +{}  { \parenthesize}               {language}
    +{;} { \PrintReprint}               {reprint}
    +{.} { }                            {note}
    +{.} {}                             {transition}
}

\BibSpec{webpage}{%
    +{}  {\PrintAuthors}                {author}
    +{,} { \@doititle}                  {title}
    +{.} { \@doititle}                  {subtitle}
    +{}  { \PrintDate}                  {date}
    +{,} { \url}                        {url}
    +{.} { Accessed \PrintDateField}    {accessdate}
    +{.} { }                            {note}
    +{.} {}                             {transition}
}

\makeatother
}{}

\tikzstyle{mutable}=[inner sep=0.5mm,circle,draw,minimum size=2mm]
\tikzstyle{frozen}=[inner sep=.9mm,rectangle,draw]
\tikzstyle{dot} = [fill=black!25,inner sep=0.5mm,circle,draw,minimum size=1mm]
\tikzstyle{blue dot} = [draw=blue,fill=blue!25,inner sep=0.5mm,circle,draw,minimum size=1mm]
\tikzstyle{marked}=[inner sep=0.5mm,circle,draw,blue!75!black,fill=blue!50]
\tikzstyle{outline}=[thick,line width=1.5mm,draw=black!10]
\tikzstyle{oriented}=[draw=red,thick,decoration={markings,mark=at position 0.52 with {\arrow{>}}},postaction={decorate}]
\tikzstyle{antioriented}=[draw=red,thick,decoration={markings,mark=at position 0.52 with {\arrow{<}}},postaction={decorate}]
\tikzstyle{faded oriented}=[draw=black!25,thick,decoration={markings,mark=at position 0.52 with {\arrow{>}}},postaction={decorate}]\tikzstyle{invisible}=[inner sep=-.3, minimum size=-.3]
\tikzstyle{matching}=[line width=1.5pt,blue]

\def\dddots{\rotatebox{-45}{$\cdots$}}

\newcommand{\high}[1]{\textcolor{dark blue}{ #1}}

\colorlet{gpurple}{red!35!blue}
\colorlet{ggreen}{green!50!black}
\colorlet{dark purple}{red!35!blue}
\colorlet{dark green}{green!70!black}
\colorlet{dark red}{red!80!black}
\colorlet{dark blue}{blue!80!black!80!cyan}

\tikzstyle{mutable}=[inner sep=0.5mm,circle,draw,minimum size=2mm]
\tikzstyle{frozen}=[inner sep=.9mm,rectangle,draw]
\tikzstyle{dot} = [fill=black!25,inner sep=0.5mm,circle,draw,minimum size=1mm]
\tikzstyle{marked}=[inner sep=0.5mm,circle,draw,blue!75!black,fill=blue!50]
\tikzstyle{traj}=[black,opacity=.5]

\def\Fr{\operatorname{F}}
\def\pperp{{\ddagger}}
\def\uni{{\mathrm{uni}}}

\def\C{\mathsf{C}}

\def\k{\Bbbk}

\begin{document}


\title{Juggler's friezes}  

\author{Roi Docampo}

\address[R.\ Docampo]{%
Department of Mathematics\\
University of Oklahoma\\
601 Elm Avenue, Room 423\\
Norman, OK 73019 (USA)%
}

\email{roi@ou.edu}

\author{Greg Muller}

\address[G.\ Muller]{%
Department of Mathematics\\
University of Oklahoma\\
601 Elm Avenue, Room 423\\
Norman, OK 73019 (USA)%
}

\email{gmuller@ou.edu}



\thanks{%
The research of the first author was partially supported by a grant from the
Simons Foundation (638459,~RD)%
}

\begin{abstract}
This note generalizes $\mathrm{SL}(k)$-friezes to configurations of numbers in
which one of the boundary rows has been replaced by a ragged edge (described by
a \emph{juggling function}). We provide several equivalent
definitions/characterizations of these \textbf{juggler's friezes}, in terms of
determinants, linear recurrences, and a dual juggler's frieze. We generalize
classic results, such as periodicity, duality, and a parametrization by part
of a Grassmannian. We also provide a method of constructing such friezes from
certain $k\times n$ matrices using the \emph{twist} of a matrix defined in
\cite{MS16}.
\end{abstract}

\vspace*{0.1em}
\maketitle


$\mathrm{SL}(k)$-friezes are configurations of numbers in which certain diamonds have determinant $1$ or $0$ (when regarded as matrices). The theory began in \cite{Cox71}, when Coxeter introduced what would now be called $\mathrm{SL}(2)$-friezes. This was followed by the (relatively) classical theory of $\mathrm{SL}(k)$-friezes, including periodicity and many enumerative results \cite{Cox71, CR72, CC73a,Sha84}.

The study of friezes recently resurged after a connection to cluster algebras was found in \cite{CC06}. This fueled a number of generalizations and theorems; most notably, that $\mathrm{SL}(k)$-friezes of height $h$ are parametrized by a certain subvariety of the Grassmannian $\mathrm{Gr}(k,k+h)$ \cite{MGOST14}, and each cluster on this Grassmannian determines a unique positive integral frieze \cite{BFGST21}. Many results on clusters can then be translated into friezes, such as enumerative results and mutation rules; see \cite{MG15,BFGST18a}.
However, the Grassmannian has a much richer structure to exploit: the Grassmannian has a stratification into \emph{positroid varieties}, each of which has a cluster structure. The subvariety and cluster structure corresponding to $\mathrm{SL}(k)$-friezes is the `big cell' in this stratification. 

In this paper, we extend this connection to other positroid varieties by introducing a generalization of $\mathrm{SL}(k)$-friezes, dubbed \textbf{juggler's friezes} for the \emph{juggling functions} that give them their shape.
We prove that juggler's friezes enjoy analogs of many of the vital properties of $\mathrm{SL}(k)$-friezes.
\begin{itemize}
    \item Juggler's friezes are periodic (Theorem \ref{thm: periodicity}).
    \item Juggler's friezes are parametrized by the unimodular positroid variety (Theorem \ref{thm: Frbijection}).
    \item Juggler's friezes are {linear recurrences with superperiodic solutions} (Theorem \ref{thm: juggleqp}).
    \item Juggler's friezes are characterized by a dual transform (Theorem \ref{thm: juggleduality}).
\end{itemize}

Sections 1 and 2 review the theory of $\mathrm{SL}(k)$-friezes, highlighting the results and techniques we later generalize. One original result here is a reformulation of the parametrization of $\mathrm{SL}(k)$-friezes by the Grassmannian in terms of the \emph{twist} of a matrix (Theorem \ref{thm: cons}).

Sections 3 and 4 introduce juggler's friezes and review their aforementioned properties. The exposition intentionally mirrors the preceding sections, to highlight the connection between $\mathrm{SL}(k)$-friezes and the broader theory of juggler's friezes. Proofs are deferred until Section 5.
While we describe the connection between juggler's friezes and cluster algebras in Section \ref{section: clusters}, proofs of these results exceed the scope of this paper, and will appear in a subsequent paper \cite{MulRes2}.


\newpage

\vspace*{\fill}

\tableofcontents

\vfill

\newpage

\section{$\mathrm{SL}(k)$-friezes}

\subsection{A breezy review of $\mathrm{SL}(k)$-friezes}\label{section: breezy}

For a positive integers $h$ and $k$, a \textbf{$\mathrm{SL}(k)$-frieze of height $h$} consists of $(h+1)$-many rows of numbers,\footnote{By `numbers', we mean elements in a fixed field $\k$, although all our examples will have values in the integers.} offset in a diamond grid, such that 
\begin{itemize}
    \item the top and bottom rows consist entirely of 1s, 
    \item each solid $k\times k$-minor is 1 (the \textbf{frieze} condition), and
    \item each solid $(k+1)\times (k+1)$-minor is 0 (the \textbf{tameness} condition).
\end{itemize}
By a \emph{solid minor}, we mean the determinant of a diamond of entries, regarded as a matrix by rotating $45^\circ$ clockwise.\footnote{\textsc{Fine print}: The empty entries above and below the frieze are treated as 0s, and we only impose the minor conditions on diamonds whose center is strictly between the top and bottom rows. A more precise definition is in Definition \ref{defn: frieze}.} 

\begin{rem}
Note that, by our definition, \textbf{every $\mathrm{SL}(k)$-frieze is tame}.
Most authors define `friezes' without the tameness condition, and use the tameness condition to define \emph{tame friezes}.
\end{rem}

\begin{warn}
The \textbf{height} $h$ counts the gaps between rows, not the number of rows themselves!\footnote{This varies among references. \cite{CC73a} call $h+1$ the \emph{order} of a frieze, and \cite{MG15} calls $h-1$ the \emph{width} of a frieze.}
\end{warn}

\begin{ex}\label{ex: intro1}
An $\mathrm{SL}(2)$-frieze of height $6$.
\[
\begin{tikzpicture}[baseline=(M-7-1.base),
    ampersand replacement=\&,
    ]
    \matrix[matrix of math nodes,
        matrix anchor = M-1-8.center,
        origin/.style={},
        throw/.style={},
        pivot/.style={draw,circle,inner sep=0.25mm,minimum size=2mm},       
        nodes in empty cells,
        inner sep=0pt,
        nodes={anchor=center},
        column sep={.4cm,between origins},
        row sep={.4cm,between origins},
    ] (M) at (0,0) {
     \& 1 \& \& 1 \& \& 1 \& \& 1 \& \& 1 \& \& 1 \& \& 1 \& \& 1 \&  \& 1 \& \& 1 \& \& 1 \& \& 1 \& \& 1 \& \& 1 \& \& 1 \& \& 1 \& \& 1 \& \& 1 \& \\
    \cdots \& \& 3 \& \& 2 \& \& 2 \& \& 1 \& \& 4 \& \& 3 \& \& 1 \& \& 2 \& \& 3 \& \& 2 \& \& 2 \& \& 1 \& \& 4 \& \& 3 \& \& 1 \& \& 2 \& \& 3 \& \& \cdots  \\
     \& 5 \& \& 5 \& \& 3 \& \& 1 \& \& 3 \& \& 11 \& \& 2 \& \& 1 \& \& 5 \& \& 5 \& \& 3 \& \& 1 \& \& 3 \& \& 11 \& \& 2 \& \& 1 \& \& 5 \& \& 5 \& \\
    \cdots \& \& 8 \& \& 7 \& \& 1 \& \& 2 \& \& 8 \& \& 7 \& \& 1 \& \& 2 \& \& 8 \& \& 7 \& \& 1 \& \& 2 \& \& 8 \& \& 7 \& \& 1 \& \& 2 \& \& 8 \& \& \cdots  \\
     \& 3 \& \& 11 \& \& 2 \& \& 1 \& \& 5 \& \& 5 \& \& 3\& \& 1 \& \& 3 \& \& 11 \& \& 2 \& \& 1 \& \& 5 \& \& 5 \& \& 3\& \& 1 \& \& 3 \& \& 11 \&  \\
    \cdots \& \& 4 \& \& 3 \& \& 1 \& \& 2 \& \& 3 \& \& 2 \& \& 2 \& \& 1 \& \& 4 \& \& 3 \& \& 1 \& \& 2 \& \& 3 \& \& 2 \& \& 2 \& \& 1 \& \& 4 \& \& \cdots   \\
     \& 1 \& \& 1 \& \& 1 \& \& 1 \& \& 1 \& \& 1 \& \& 1 \& \& 1 \& \& 1 \& \& 1 \& \& 1 \& \& 1 \& \& 1 \& \& 1 \& \& 1 \& \& 1 \& \& 1 \& \& 1 \& \\
    };
    
    \draw[dark blue, fill=dark blue!50,opacity=.25,rounded corners] (M-3-11.center) -- (M-1-13.center) -- (M-3-15.center) -- (M-5-13.center) -- cycle;
    \draw[dark red, fill=dark red!50,opacity=.25,rounded corners] (M-4-20.center) -- (M-1-23.center) -- (M-4-26.center) -- (M-7-23.center) -- cycle;

\end{tikzpicture} \]
Note that the \textcolor{dark blue}{left submatrix} has determinant 1, and the \textcolor{dark red}{right submatrix} has determinant 0.
\end{ex}

\begin{ex}\label{ex: intro2}
An $\mathrm{SL}(3)$-frieze of height $5$.
    \[
    \begin{tikzpicture}[baseline=(M-6-1.base),
    ampersand replacement=\&,
    ]
    \clip[use as bounding box] (-7.5,.3) rectangle (7.5,-2.3);
    \matrix[matrix of math nodes,
        matrix anchor = M-1-29.center,
        origin/.style={},
        throw/.style={},
        pivot/.style={draw,circle,inner sep=0.25mm,minimum size=2mm},       
        nodes in empty cells,
        inner sep=0pt,
        nodes={anchor=center},
        column sep={.4cm,between origins},
        row sep={.4cm,between origins},
    ] (M) at (0,0) {
    \& \& \& \& \& 1 \& \& 1 \& \& 1 \& \& 1 \& \& 1 \& \& 1 \& \& 1 \& \& 1 \& \& 1 \& \& 1 \& \& 1 \& \& 1 \& \& 1 \& \& 1 \& \& 1 \& \& 1 \& \& 1 \& \& 1 \& \& 1 \& \& 1 \& \& 1 \& \& 1 \& \& 1 \& \& 1 \& \& \\
    \& \& \& \& 18 \& \& 1 \& \& 5 \& \& \cdots \& \& 3 \& \& 3 \& \& 3 \& \& 1 \& \& 18 \& \& 1 \& \& 5 \& \& 2 \& \& 3 \& \& 3 \& \& 3 \& \& 1 \& \& 18 \& \& 1 \& \& 5 \& \& 2 \& \& 3 \& \& \cdots \& \& 3 \& \& 1 \& \& \\
    \& \& \& 7 \& \& 16 \& \& 1 \& \& 8 \& \& 3 \& \& 6 \& \& 5 \& \& 2 \& \& 7 \& \& 16 \& \& 1 \& \& 8 \& \& 3 \& \& 6 \& \& 5 \& \& 2 \& \& 7 \& \& 16 \& \& 1 \& \& 8 \& \& 3 \& \& 6 \& \& 5 \& \& 2 \& \& \\
    \& \& 4 \& \& 6 \& \& 9 \& \& 1 \& \& \cdots \& \& 4 \& \& 7 \& \& 3 \& \& 4 \& \& 6 \& \& 9 \& \& 1 \& \& 10 \& \& 4 \& \& 7 \& \& 3 \& \& 4 \& \& 6 \& \& 9 \& \& 1 \& \& 10 \& \& 4 \& \& \cdots \& \& 3 \& \& \\
    \& 2 \& \& 3 \& \& 3 \& \& 4 \& \& 1 \& \& 11 \& \& 2 \& \& 4 \& \& 2 \& \& 3 \& \& 3 \& \& 4 \& \& 1 \& \& 11 \& \& 2 \& \& 4 \& \& 2 \& \& 3 \& \& 3 \& \& 4 \& \& 1 \& \& 11 \& \& 2 \& \& 4 \& \& \\
    1 \& \& 1 \& \& 1 \& \& 1 \& \& 1 \& \& \cdots \& \& 1 \& \& 1 \& \& 1 \& \& 1 \& \& 1 \& \& 1 \& \& 1 \& \& 1 \& \& 1 \& \& 1 \& \& 1 \& \& 1 \& \& 1 \& \& 1 \& \& 1 \& \& 1 \& \& 1 \& \& \cdots \& \& \\
    };
    \end{tikzpicture}\qedhere
    \]
\end{ex}

Coxeter \cite{Cox71} introduced $\mathrm{SL}(2)$-friezes (called \emph{frieze patterns} at the time) to understand and generalize the \emph{pentagramma mirificum}, a recurrence observed by Gauss among the side lengths of a right-angled spherical pentagram, and showed they were always $(h+2)$-periodic (observe that Example \ref{ex: intro1} is $8$-periodic). 
More generally, an $\mathrm{SL}(k)$-frieze of height $h$ are $(h+k)$-periodic \cite{BR10,MGOST14,KV15} (observe that Example \ref{ex: intro2} is $8$-periodic).

The first enumerative result was Coxeter-Conway \cite{CC73a}, who showed that the positive integral $\mathrm{SL}(2)$-friezes of height $h$ are in bijection with triangulations of an $(h+2)$-gon, and are therefore counted by a Catalan number. Despite this, there are infinitely many positive integral $\mathrm{SL}(k)$-friezes of height $h$ except when $\min(k,h)\leq2$, or $\min(k,h)=3$ and $\max(k,h)\leq 6$ (see e.g.~\cite{BFGST21}).

\subsection{(Pre)friezes as matrices}
\label{section: matrices}

Our next goal is to give several characterizations of $\mathrm{SL}(k)$-friezes which replace the frieze and tameness conditions. It will be convenient to refer to infinite strips of numbers without any extra conditions; to this end, a \textbf{prefrieze} will consist of finitely-many rows of numbers offset in a diamond grid, in which the top and bottom rows consist of $1$s; i.e. a frieze minus any conditions on the minors. 

\begin{ex}
A prefrieze of height $3$, satisfying no obvious determinantal identities or periodicity.
    \[
    \begin{tikzpicture}[
        baseline=(M-4-1.base),
    ampersand replacement=\&,
    ]
    \clip[use as bounding box] (-7.5,.3) rectangle (7.5,-1.5);
    \matrix[matrix of math nodes,
        matrix anchor = M-1-29.center,
        origin/.style={},
        throw/.style={},
        pivot/.style={draw,circle,inner sep=0.25mm,minimum size=2mm},       
        nodes in empty cells,
        inner sep=0pt,
        nodes={anchor=center},
        column sep={.4cm,between origins},
        row sep={.4cm,between origins},
    ] (M) at (0,0) {
    \& \& \& \& \& 1 \& \& 1 \& \& 1 \& \& 1 \& \& 1 \& \& 1 \& \& 1 \& \& 1 \& \& 1 \& \& 1 \& \& 1 \& \& 1 \& \& 1 \& \& 1 \& \& 1 \& \& 1 \& \& 1 \& \& 1 \& \& 1 \& \& 1 \& \& 1 \& \& 1 \& \& 1 \& \& 1 \& \& \\
    \& \& \& \&  \& \&  \& \&  \& \& \cdots \& \& 3 \& \& 7 \& \& 9 \& \& 8 \& \& 7 \& \& 7 \& \& 3 \& \& 8 \& \& 1 \& \& 6 \& \& 0 \& \& 6 \& \& 7 \& \& 4 \& \& 5 \& \& 4 \& \& 5 \& \& \cdots \& \&  \& \&  \& \& \\
    \& \& \&  \& \&  \& \&  \& \&  \& \& 8 \& \& 8 \& \& 9 \& \& 7 \& \& 4 \& \& 9 \& \& 1 \& \& 4 \& \& 7 \& \& 7 \& \& 3 \& \& 7 \& \& 4 \& \& 7 \& \& 5 \& \& 7 \& \& 6 \& \& 4 \& \&  \& \&  \& \& \\
    1 \& \& 1 \& \& 1 \& \& 1 \& \& 1 \& \& \cdots \& \& 1 \& \& 1 \& \& 1 \& \& 1 \& \& 1 \& \& 1 \& \& 1 \& \& 1 \& \& 1 \& \& 1 \& \& 1 \& \& 1 \& \& 1 \& \& 1 \& \& 1 \& \& 1 \& \& 1 \& \& \cdots \& \& \\
    };
    \end{tikzpicture}
    \qedhere
    \]
\end{ex}


It will be convenient to identify each prefrieze with a $\mathbb{Z}\times \mathbb{Z}$-matrix by rotating the prefrieze $45^\circ$ clockwise and setting the top row of the prefrieze to be the main diagonal of the matrix. 

\begin{ex}
The $\mathrm{SL}(3)$-frieze in Example \ref{ex: intro2} corresponds to the following matrix.
\[
\begin{tikzpicture}[baseline=(current bounding box.center),
    ampersand replacement=\&,
    ]
    \matrix[matrix of math nodes,
        nodes in empty cells,
        inner sep=0pt,
        nodes={anchor=center,node font=\scriptsize},
        column sep={.35cm,between origins},
        row sep={.35cm,between origins},
        left delimiter={[},
        right delimiter={]},
    ] (M) at (0,0) {
        \dddots \& \& \&  \& \& \& \& \&  \\
         \& 1 \& |[gray]| 0 \& |[gray]| 0 \& |[gray]| 0 \& |[gray]| 0 \& |[gray]| 0 \& |[gray]| 0 \& |[gray]| 0 \& |[gray]| 0 \& |[gray]| 0 \& |[gray]| 0 \& |[gray]| 0 \& |[gray]| 0 \& |[gray]| 0 \& \\
         \& 3 \& 1 \& |[gray]| 0 \& |[gray]| 0 \& |[gray]| 0 \& |[gray]| 0 \& |[gray]| 0 \& |[gray]| 0 \& |[gray]| 0 \& |[gray]| 0 \& |[gray]| 0 \& |[gray]| 0 \& |[gray]| 0 \& |[gray]| 0 \& \\
         \& 6 \& 3 \& 1 \& |[gray]| 0 \& |[gray]| 0 \& |[gray]| 0 \& |[gray]| 0 \& |[gray]| 0 \& |[gray]| 0 \& |[gray]| 0 \& |[gray]| 0 \& |[gray]| 0 \& |[gray]| 0 \& |[gray]| 0 \& \\
         \& 7 \& 5 \& 3 \& 1 \& |[gray]| 0 \& |[gray]| 0 \& |[gray]| 0 \& |[gray]| 0 \& |[gray]| 0 \& |[gray]| 0 \& |[gray]| 0 \& |[gray]| 0 \& |[gray]| 0 \& |[gray]| 0 \& \\
         \& 4 \& 3 \& 2 \& 1 \& 1 \& |[gray]| 0 \& |[gray]| 0 \& |[gray]| 0 \& |[gray]| 0 \& |[gray]| 0 \& |[gray]| 0 \& |[gray]| 0 \& |[gray]| 0 \& |[gray]| 0 \& \\
         \& 1 \& 2 \& 4 \& 7 \& 18 \& 1 \& |[gray]| 0 \& |[gray]| 0 \& |[gray]| 0 \& |[gray]| 0 \& |[gray]| 0 \& |[gray]| 0 \& |[gray]| 0 \& |[gray]| 0 \& \\
         \& |[gray]| 0 \& 1 \& 3 \& 6 \& 16 \& 1 \& 1 \& |[gray]| 0 \& |[gray]| 0 \& |[gray]| 0 \& |[gray]| 0 \& |[gray]| 0 \& |[gray]| 0 \& |[gray]| 0 \& \\
         \& |[gray]| 0 \& |[gray]| 0 \& 1 \& 3 \& 9 \& 1 \& 5 \& 1 \& |[gray]| 0 \& |[gray]| 0 \& |[gray]| 0 \& |[gray]| 0 \& |[gray]| 0 \& |[gray]| 0 \& \\
         \& |[gray]| 0 \& |[gray]| 0 \& |[gray]| 0 \& 1 \& 4 \& 1 \& 8 \& 2 \& 1 \& |[gray]| 0 \& |[gray]| 0 \& |[gray]| 0 \& |[gray]| 0 \& |[gray]| 0  \\
         \& |[gray]| 0 \& |[gray]| 0 \& |[gray]| 0\& |[gray]| 0 \& 1 \& 1 \& 10 \& 3 \& 3 \& 1 \& |[gray]| 0 \& |[gray]| 0 \& |[gray]| 0 \& |[gray]| 0 \\
         \& |[gray]| 0 \& |[gray]| 0 \& |[gray]| 0 \& |[gray]| 0 \& |[gray]| 0 \& 1 \& 11 \& 4 \& 6 \& 3 \& 1 \& |[gray]| 0 \& |[gray]| 0 \& |[gray]| 0 \\
         \& |[gray]| 0 \& |[gray]| 0 \& |[gray]| 0 \& |[gray]| 0 \& |[gray]| 0 \& |[gray]| 0 \& 1 \& 2 \& 7 \& 5 \& 3 \& 1 \& |[gray]| 0 \& |[gray]| 0 \\
         \& |[gray]| 0 \& |[gray]| 0 \& |[gray]| 0 \& |[gray]| 0 \& |[gray]| 0 \& |[gray]| 0 \& |[gray]| 0 \& 1 \& 4 \& 3 \& 2 \& 1 \& 1 \& |[gray]| 0 \\
         \& |[gray]| 0 \& |[gray]| 0 \& |[gray]| 0 \& |[gray]| 0 \& |[gray]| 0 \& |[gray]| 0 \& |[gray]| 0 \& |[gray]| 0 \& 1 \& 2 \& 4 \& 7 \& 18 \& 1 \\
        \& \& \& \& \& \& \& \& \& \& \& \& \& \& \& \dddots \\
    };
\end{tikzpicture}
\]
The implicit zeroes are in gray.
\end{ex}

This identification allows us to index the entries of a prefrieze via their indices as entries of a matrix. Given a prefrieze $\C$, we let $\C_{a,b}$ be the entry in the $a$th row\footnote{This creates a conflict between the `rows' of a frieze and the `rows' of the corresponding matrix. We generally use `rows' to refer to the first concept, and `rows of the corresponding matrix' otherwise, except in Section \ref{section: proofs}.} and $b$th column of the corresponding matrix; e.g.~the following indeterminant prefrieze of height 4.
\[ 
\begin{tikzpicture}[baseline=(current bounding box.center),
    ampersand replacement=\&,
    ]
    \begin{scope}
    \matrix[matrix of math nodes,
        matrix anchor = M-4-8.center,
        throw/.style={},
        origin/.style={dark green,draw,circle,inner sep=0.25mm,minimum size=2mm},
        pivot/.style={draw,circle,inner sep=0.25mm,minimum size=2mm},       
        nodes in empty cells,
        inner sep=0pt,
        nodes={anchor=center},
        column sep={.65cm,between origins},
        row sep={.65cm,between origins},
    ] (M) at (0,0) {
     \& |[throw]| 1 \& \& |[throw]| 1 \& \& |[throw]| 1 \& \& |[throw]| 1 \& \&|[throw]| 1 \& \& |[throw]| 1 \& \& |[throw]| 1 \& \& \cdots \\
    \cdots \& \& \C_{1,0} \& \& \C_{2,1} \& \& \C_{3,2} \& \&\C_{4,3} \& \&\C_{5,4} \& \&\C_{6,5} \& \&\C_{7,6} \&  \\
     \&\C_{1,-1} \& \&\C_{2,0} \& \&\C_{3,1} \& \&\C_{4,2} \& \&\C_{5,3} \& \&\C_{6,4} \& \&\C_{7,5} \& \& \cdots \\
    \cdots \& \&\C_{2,-1} \& \&\C_{3,0} \& \&\C_{4,1} \& \&\C_{5,2} \& \&\C_{6,3} \& \&\C_{7,4} \& \&\C_{8,5} \&  \\
     \& 1 \& \& 1 \& \& 1 \& \& 1 \& \& 1 \& \& 1 \& \& 1 \& \& \cdots \\
    };
    
    \end{scope}

\end{tikzpicture}
\]
If $\C$ has height $h$, then $\C_{a,a}=\C_{a+h,a}=1$ for all $a$, and $\C_{a,b}=0$ for all $a<b$ or $a>b+h$.
If $I,J\subset \mathbb{Z}$, we let $\C_{I,J}$ denote the submatrix of the corresponding matrix with rows in $I$ and columns in $J$. 
This notation allows us to state the definition of an $\mathrm{SL}(k)$-frieze more precisely.

\newpage

\begin{defn}\label{defn: frieze}
An \textbf{$\mathrm{SL}(k)$-frieze} is a prefrieze $\C$ of height $h$ such that
\begin{enumerate}
    \item $\det(\C_{[a,a+k-1],[b,b+k-1]}) =1$ whenever $b\leq a \leq b+h$, and 
    \item $\det(\C_{[a,a+k],[b,b+k]}) =0$ whenever $b< a < b+h$.
\end{enumerate}
\end{defn}

\subsection{Friezes and duality}\label{section: friezedual}

We begin our alternate characterizations of friezes by observing that, instead of considering solid minors of size $k$ and $k+1$, one may instead consider solid minors centered on the second row with sizes between $k$ and $k+h$.

\begin{lemma}\label{lemma: predual}
A prefrieze of height $h$ is a $\mathrm{SL}(k)$-frieze iff 
\begin{enumerate}
    \item $\det(\C_{[a+1,a+k],[a,a+k-1]}) =1$ for all $a\in \mathbb{Z}$, and 
    \item $\det(\C_{[a+1,a+\ell],[a,a+\ell-1]}) =0$ for all $a\in \mathbb{Z}$ and all $\ell$ with $k<\ell\leq k+h$.
\end{enumerate}
\end{lemma}

\begin{proof}[Proof sketch]
A more general proof is given by Lemma \ref{lemma: cofrieze}, but we sketch a simpler argument here.

One direction follows from the observation that, if every solid $k\times k$-minor of an $\ell\times \ell$ matrix is $1$, then the determinant of that matrix is equal to the determinant of the matrix of its solid $(k+1)\times (k+1)$-minors. This is a corollary of the Dodgson condensation algorithm, which `resets' at the $(k+1)$th step if all the $k\times k$-minors are $1$. This observation implies that any $\ell\times \ell$-minor of a $\mathrm{SL}(k)$ frieze is $0$, provided its $k\times k$-minors stay in the region where they are forced to be $1$.

The other direction follows by considering the adjugate matrix $\mathrm{Adj}(\C)$ of the matrix corresponding to $\C$.\footnote{While an infinite matrix may not have an adjugate, a unitriangular matrix does; see Lemma \ref{lemma: minordual}.} The frieze and tameness conditions translate to $\mathrm{Adj}(\C)$ via the identities
\[ \det(\C_{[a,a+k-1],[b,b+k-1]}) = \pm\det(\mathrm{Adj}(\C)_{[b+k,a+k-1],[b,a-1]})\]
\[ \det(\C_{[a,a+k],[b,b+k]}) = \pm\det(\mathrm{Adj}(\C)_{[b+k+1,a+k],[b,a-1]})\]
Conditions (1) and (2) in the statement of the lemma imply the determinants on the right-hand-side are respectively upper unitriangular and strictly upper triangular (up to sign). After checking the signs work out, this shows that $\C$ satisfies the frieze and tameness conditions.
\end{proof}

This lemma may seem arbitrary, but we can rephrase it in a more pleasing form by collecting all sufficiently small solid minors of $\C$ centered on the second row into a single object.
Given a prefrieze $\C$ and a positive integer $n$, define the
\textbf{$n$-truncated dual} $\C^\dagger$ to be the diamond grid of numbers (or
equivalently, the $\mathbb{Z}\times \mathbb{Z}$-matrix) defined by
\[
    \C^\dagger_{a,b}
    :=
    \begin{cases}
        \det(\C_{[b+1,a],[b,a-1]}) &\text{if }b\leq a <b+n, \\
        0 & \text{otherwise.}
    \end{cases}
\]
When $a=b$, the corresponding determinant is empty and is defined to be $1$.

\begin{ex}
The $8$-truncated dual $\C^\dagger$ of Example \ref{ex: intro2} is below. 
\[
\begin{tikzpicture}[
    ampersand replacement=\&,
    ]
    \matrix[matrix of math nodes,
        matrix anchor = M-1-8.center,
        origin/.style={},
        throw/.style={},
        pivot/.style={draw,circle,inner sep=0.25mm,minimum size=2mm},       
        nodes in empty cells,
        inner sep=0pt,
        nodes={anchor=center},
        column sep={.4cm,between origins},
        row sep={.4cm,between origins},
    ] (M) at (0,0) {
     \& 1 \& \& 1 \& \& 1 \& \& 1 \& \& 1 \& \& 1 \& \& 1 \& \& 1 \&  \& 1 \& \& 1 \& \& 1 \& \& 1 \& \& 1 \& \& 1 \& \& 1 \& \& 1 \& \& 1 \& \& 1 \& \\
    \cdots \& \& 3 \& \& 3 \& \& 3 \& \& 1 \& \& 18 \& \& 1 \& \& 5 \& \& 2 \& \& 3 \& \& 3 \& \& 3 \& \& 1 \& \& 18 \& \& 1 \& \& 5 \& \& 2 \& \& 3 \& \& \cdots \\
     \& 3 \& \& 3 \& \& 4 \& \& 1 \& \& 11 \& \& 2 \& \& 4 \& \& 2 \& \& 3 \& \& 3 \& \& 4 \& \& 1 \& \& 11 \& \& 2 \& \& 4 \& \& 2 \& \& 3 \& \& 3 \& \\
    \cdots \& \& 1 \& \& 1 \& \& 1 \& \& 1 \& \& 1 \& \& 1 \& \& 1 \& \& 1 \& \& 1 \& \& 1 \& \& 1 \& \& 1 \& \& 1 \& \& 1 \& \& 1 \& \& 1 \& \& 1 \& \& \cdots \\
    };
\end{tikzpicture} \]
The entries depicted above are only the non-zero ones; there are infinitely many implicit rows of $0$ above and below the strip above. Four of these rows correspond to solid minors of $\C$ of sizes 4 to 7 that vanish; the rest are zero by the definition of $\C^\dagger$.
\end{ex}

The minors in Lemma \ref{lemma: predual} are entries of $\C^\dagger$, and the conditions on these minors translate to the requirement that the $k$th row of $\C^\dagger$ consists of $1$s, and all lower rows consist of $0$s; that is, that $\C^\dagger$ is a prefrieze of height $k$.

\begin{lemma}
A prefrieze $\C$ of height $h$ is an $\mathrm{SL}(k)$-frieze iff the $(h+k)$-truncated dual $\C^\dagger$ is a prefrieze of height $k$.
\end{lemma}

To reap the benefits of this result, note that  $(\C^\dagger)^\dagger=\C$ whenever $\C$ is a prefrieze of height less than $n$.\footnote{This follows from the corresponding property of the adjugate of a matrix, or as a special case of Lemma \ref{lemma: doubledual}.} Therefore, if $\C$ is an $\mathrm{SL}(k)$-frieze of height $h$, then the $(h+k)$-truncated dual of $\C^\dagger$ is a prefrieze of height $k$, and so $\C^\dagger$ is an $\mathrm{SL}(h)$-frieze of height $k$. We can restate this as follows.

\begin{thm}
Taking $(h+k)$-truncated duals defines mutually inverse bijections between the sets of $\mathrm{SL}(k)$-friezes of height $h$ and $\mathrm{SL}(h)$-friezes of height $k$.
\end{thm}

\begin{rem}
The frieze $\C^\dagger$ differs from the \emph{combinatorial Gale dual} of $\C$ defined in \cite{MGOST14} by a vertical reflection (called \emph{projective duality} in \emph{loc.~cit.}) and a horizontal shift. 
\end{rem}

\subsection{Friezes as linear recurrences}

Given a prefrieze $\C$ of height $h$, we may consider the matrix equation $\C\mathsf{x=0}$, where $\mathsf{x}=(...,x_{-1},x_0,x_1,x_2,...)$ denotes a vector of variables of length $\mathbb{Z}$. This is equivalent to a system of linear equations of the form
\[ x_{a-h}+\C_{a,a-h+1} x_{a-h+1} + \C_{a,a-h+2} x_{a-h+2} + \cdots + \C_{a,a-1} x_{a-1} + x_a =0 \]
running over all $a\in \mathbb{Z}$. Each such equation can be rewritten as a formula for $x_a$ as a linear combination of earlier variables:
\[ x_a = - \C_{a,a-1} x_{a-1} - \C_{a,a-2}x_{a-2} - \cdots - \C_{a,a-h+1}x_{a-h+1} - x_{a-h}\]
For this reason, we call the matrix equation $\C\mathsf{x=0}$ a \textbf{linear recurrence}.\footnote{Also called a \emph{linear recurrence relation} or a \emph{system of linear difference equations}.}

\begin{thm}\cite{MGOST14}\label{thm: friezeperiodic}
A prefrieze $\C$ of height $h$ is a $\mathrm{SL}(k)$-frieze iff every solution to the linear recurrence $\C \mathsf{x=0}$ 
satisfies $x_{a+h+k} = (-1)^{k-1}x_a$ for all $a\in \mathbb{Z}$.
\end{thm}

\begin{rem}
\cite{MGOST14} uses an alternating sign convention when associating a linear recurrence to a frieze, which we do not. As a result, the sign $(-1)^{k-1}$ appears different than in \emph{loc.~cit.}
\end{rem}

The solutions to the linear recurrence $\C \mathsf{x} =0$ are closely related to the dual $\C^\dagger$, as follows.

\begin{thm}
Let $\C$ be a $\mathrm{SL}(k)$-frieze of height $h$.
For each $b$, the sequence $\mathbf{x}$ defined by 
\begin{itemize}
    \item $x_a := (-1)^{a+b}\C^\dagger_{a,b} $ when $b\leq a< b+k+h$, and 
    \item $x_{a+h+k} = (-1)^{k-1}x_a$ for all $a$
\end{itemize}
is a solution to $\C \mathsf{x=0}$. Running over all $b$, these solutions span the space of solutions to $\C \mathsf{x=0}$.
\end{thm}

Conceptually, each northwest-southeast diagonal of $\C^\dagger$ determines a solution to $\C \mathsf{x=0}$ by first flipping half the signs and then extending superperiodically. This theorem is essentially \cite[Prop.~5.1.1]{MGOST14}, and a special case of Theorem \ref{thm: qpsols}.

\begin{ex}
If $\C$ is the frieze in Example \ref{ex: intro2}, then two solutions to $\C\mathsf{x=0}$ are
\[ \cdots , 0 , 1 , -3 , 3 , -1 , 0 , 0 , 0 , 0 , 1 , -3 , 3 , -1 , 0 \cdots \]
\[ \cdots , 0 , 0 , 1 , -4 , 3 , -1 , 0 , 0 , 0 , 0 , 1 , -4 , 3 , -1 , \cdots \]
Observe that they each satisfy the condition that $x_{a+8}=x_a$ (since $k-1=2$).
\end{ex}

\section{Constructing $\mathrm{SL}(k)$-friezes}

\label{section: constructfrieze}

\def\lt{\tau^{-1}}

While we now know several methods to recognize $\mathrm{SL}(k)$-friezes, our story lacks a method to construct them, which we remedy now. We give two equivalent constructions, and observe that this parametrizes $\mathrm{SL}(k)$-friezes of height $h$ by a subvariety of the Grassmannian $\mathrm{Gr}(h,k+h)$.

\begin{nota}
Given a $k\times n$ matrix $\mathsf{A}$ and $a\in \mathbb{Z}$, let $\mathsf{A}_a$ denote the $\overline{a}$th column of $\mathsf{A}$, where $\overline{a}\equiv a$ mod $n$. If $I\subset \mathbb{Z}$, let $\mathsf{A}_I$ denote the submatrix of $\mathsf{A}$ consisting of columns indexed by $I$ mod $n$.\footnote{The columns of $\mathsf{A}_I$ remain in the same order as in $\mathsf{A}$; e.g.~if the width $n=5$, then $\mathsf{A}_{\{4,5,6\}}$ would consist of the $1$st, $4$th, and $5$th columns of $\mathsf{A}$, in that order. }
\end{nota}

\subsection{Constructing friezes}

A $k\times n$-matrix $\mathsf{A}$ is \textbf{consecutively unimodular} if $\det(\mathsf{A}_{[a,a+k-1]})=1$ for all $a\in \mathbb{Z}$; i.e.~every submatrix on $k$-many cyclically consecutive\footnote{A \emph{cyclically consecutive} set is the image of a consecutive set under the quotient $\mathbb{Z}\rightarrow \mathbb{Z}/n\mathbb{Z}$.}
 columns has determinant 1. 
\begin{ex}\label{ex: conuni}
The following $3\times 8$ matrix is consecutively unimodular, which can be verified by checking the six solid $3\times 3$-minors and the minors on columns $178$ and $128$.
\[ \mathsf{A} := 
\begin{bmatrix}
1 & 11 & 4 & 6 & 3 & 1 & 0 & 0 \\
0 & 1 & 2 & 7 & 5 & 3 & 1 & 0 \\
0 & 0 & 1 & 4 & 3 & 2 & 1 & 1 \\
\end{bmatrix}
\qedhere
\]
\end{ex}

\begin{thm}[\cite{BFGST21}]\label{thm: detfrieze}
Each consecutively unimodular $k\times n$-matrix $\mathsf{A}$ determines an
$\mathrm{SL}(k)$-frieze $\Fr(\mathsf{A})$ of height $h=n-k$, defined by
\[
    \Fr(\mathsf{A})_{a,b}
    =
    \begin{cases}
        \det(\mathsf{A}_{[{a}+1,{a}+k-1]\cup \{{b}\}}) & \text{if $b\leq a \leq b+n-k$}, \\
        0 & \text{otherwise}.
    \end{cases}
\]
\end{thm}

\def\rt{\tau}

The construction in Theorem \ref{thm: detfrieze} \emph{a priori} requires computing $k(n-k)$-many $k\times k$-determinants, which is a daunting task even for small examples like Example \ref{ex: conuni}. We can simplify this construction considerably by means of the \emph{twist} of a matrix. 
%
%
The \textbf{(right) twist} of a consecutively unimodular $k\times n$-matrix $\mathsf{A}$ is the $k\times n$-matrix $\rt(\mathsf{A})$ whose columns are defined by the dot product equations
\[
    \mathsf{A}_a \cdot \rt(\mathsf{A})_b
    =
    \begin{cases}
        1 & \text{if }a=b, \\
        0 & \text{if }b<a<b+k. \\
    \end{cases}
\]
Since the matrix is consecutively unimodular, the columns of $\mathsf{A}$ indexed by $[b,b+k-1]$ are linearly independent, and so the above system uniquely determines each column of $\rt(\mathsf{A})$. Many interesting properties of the twist can be found in \cite[Section 6]{MS17}; one relevant result is that the twist of a consecutively unimodular matrix is again consecutively unimodular \cite[Prop.~6.6]{MS17}.

\begin{warn}
The twist may be defined more generally for $k\times n$-matrices of rank $k$, but the definition must be generalized to ensure a linearly independent set of columns; see Section \ref{section: jugtwist}.
\end{warn}

\begin{rem}
The twist of a matrix in this context was introduced in \cite{MS16} and a normalized version was given in \cite{MS17}. Since the two definitions differ by a factor of the determinant of cyclically consecutive columns; they coincide for a consecutively unimodular matrix.
\end{rem}

Using the twist, we can describe an algorithm (Construction \ref{cons: unwrap}, on the following page) for constructing an infinite strip of numbers from a consecutively unimodular matrix. As the following theorem asserts, this construction produces the same $SL(k)$-frieze as Theorem \ref{thm: detfrieze}.

\begin{thm}\label{thm: cons}
If $\mathsf{A}$ is a consecutively unimodular $k\times n$-matrix,
Construction \ref{cons: unwrap} produces $\Fr(\mathsf{A})$, the same $\mathrm{SL}(k)$-frieze as Theorem \ref{thm: detfrieze}.
\end{thm}

\begin{proof}
Let $\C$ be the output of the construction, let $a\leq b<a+n$, and let $\overline{a},\overline{b}$ denote the residues of $a,b$ mod $n$. Then steps (3) and (4) of the construction are equivalent to the following conditional.
\begin{align*}
    \C_{a,b}
    &=
    \begin{cases}
    (\rt(\mathsf{A})^\top \mathsf{A})_{\overline{a},\overline{b}} & \text{if }\overline{a} \geq \overline{b} \\
    (-1)^{k-1}(\rt(\mathsf{A})^\top \mathsf{A})_{\overline{a},\overline{b}} & \text{if }\overline{a} < \overline{b} \\
    \end{cases}
\\
    &=
    \begin{cases}
    \rt(\mathsf{A})_{\overline{a}} \cdot \mathsf{A}_{\overline{b}} & \text{if } \overline{a} \geq \overline{b} \\
    (-1)^{k-1}\rt(\mathsf{A})_{\overline{a}}\cdot \mathsf{A}_{\overline{b}} \hspace{1em} & \text{if }\overline{a} < \overline{b} \\
    \end{cases}
\end{align*}
The $\overline{a}$th column of $\rt(\mathsf{A})$ is defined to be the unique solution to ${\mathsf{A}}_{[\overline{a},\overline{a}+k-1]}^\top\mathsf{x}=\mathsf{e}_{\overline{a}}$.
Therefore,  $\rt(\mathsf{A})_{\overline{a}} = ({\mathsf{A}}_{[\overline{a},\overline{a}+k-1]}^{\top})^{-1} \, \mathsf{e}_{\overline{a}}$, and so
\[ \rt(\mathsf{A})_{\overline{a}} \cdot \mathsf{A}_{\overline{b}}  = (({\mathsf{A}}_{[\overline{a},\overline{a}+k-1]}^{\top})^{-1} \, \mathsf{e}_{\overline{a}}) \cdot \mathsf{A}_{\overline{b}} = 
\mathsf{e}_{\overline{a}} \cdot ({\mathsf{A}}_{[\overline{a},\overline{a}+k-1]})^{-1} \, \mathsf{A}_{\overline{b}}
\]
Equivalently, $\rt(\mathsf{A})_{\overline{a}} \cdot \mathsf{A}_{\overline{b}}$ is the ${\overline{a}}$th entry of the solution to ${\mathsf{A}}_{[\overline{a},\overline{a}+k-1]}\mathsf{x} = \mathsf{A}_{\overline{b}}$. By Cramer's Rule, this is the determinant of ${\mathsf{A}}_{[\overline{a},\overline{a}+k-1]}$ after replacing the $\overline{a}$th column of $\mathsf{A}$ by the $\overline{b}$th column.\footnote{There is also a denominator of $\det(\mathsf{A}_{[\overline{a},\overline{a}+k-1]})$ in Cramer's Rule, but this determinant is 1 by assumption.} This is \emph{almost} $\det(\mathsf{A}_{[\overline{a}+1,\overline{a}+k-1]\cup \{b\}})$; however, our convention for ordering the columns in $\det(\mathsf{A}_{[\overline{a}+1,\overline{a}+k-1]\cup \{b\}})$ means there is an additional factor of $(-1)^{k-1}$ if $\overline{a}<\overline{b}$. That is,
\[ 
\rt(\mathsf{A})_{\overline{a}} \cdot \mathsf{A}_{\overline{b}} =
\begin{cases}
\det(\mathsf{A}_{[\overline{a}+1,\overline{a}+k-1]\cup \{\overline{b}\}}) & \text{if }\overline{a} \geq \overline{b}, \\
(-1)^{k-1}\det(\mathsf{A}_{[\overline{a}+1,\overline{a}+k-1]\cup \{\overline{b}\}}) & \text{if }\overline{a} < \overline{b}. \\
\end{cases}
\]
Therefore, $\C_{a,b} = \det(\mathsf{A}_{[\overline{a}+1,\overline{a}+k-1]\cup \{\overline{b}\}})$. Since this is $0$ whenever $\overline{b}\in [\overline{a}+1,\overline{a}+k-1]$, we obtain the formula in Theorem \ref{thm: detfrieze}.
\end{proof}

\begin{rem}
Theorem \ref{thm: cons} provides an alternate proof of Theorem \ref{thm: detfrieze}. By Cauchy-Binet,
\[ \det(\Fr(\mathsf{A})_{[a,a+k-1],[b,b+k-1]}) =  \det( (\rt(\mathsf{A})^\top \mathsf{A})_{[a,a+k-1],[b,b+k-1]}) =  \det(\rt(\mathsf{A})_{[a,a+k-1]}) \det(\mathsf{A}_{[b,b+k-1]})  \]
Both of the final determinants are $1$, as $\mathsf{A}$ and $\rt(\mathsf{A})$ are consecutively unimodular. By a similar argument, any $(k+1)\times (k+1)$ minor of $\C$ vanishes, since the product $\rt(\mathsf{A})^\top\mathsf{A}$ has rank $k$.
\end{rem}

\newpage

\begin{cons}\label{cons: unwrap}
Let $\mathsf{A}$ be a consecutively unimodular $k\times n$-matrix. As a running example:
\[ \mathsf{A} := 
\begin{bmatrix}
1 & 11 & 4 & 6 & 3 & 1 & 0 & 0 \\
0 & 1 & 2 & 7 & 5 & 3 & 1 & 0 \\
0 & 0 & 1 & 4 & 3 & 2 & 1 & 1 \\
\end{bmatrix}\]
\begin{enumerate}
    \item Compute the twist $\rt(\mathsf{A})$ of $\mathsf{A}$.
    \[ \rt(\mathsf{A}) = 
    \begin{bmatrix}
    1 & 1 & 1 & 1 & 1 & 1 & 0 & 0 \\
    -11 & -10 & -6 & -3 & -1 & 0 & 1 & 0 \\
    18 & 16 & 9 & 4 & 1 & 0 & 0 & 1 \\
    \end{bmatrix}\]
    \item Compute the matrix product $\rt(\mathsf{A})^\top \mathsf{A}$.
    \[ \rt(\mathsf{A})^\top \mathsf{A} = 
    \begin{bmatrix}
    1 & 0 & 0 & 1 & 2 & 4 & 7 & 18 \\
    1 & 1 & 0 & 0 & 1 & 3 & 6 & 16 \\
    1 & 5 & 1 & 0 & 0 & 1 & 3 & 9 \\
    1 & 8 & 2 & 1 & 0 & 0 & 1 & 4 \\
    1 & 10 & 3 & 3 & 1 & 0 & 0 & 1\\
    1 & 11 & 4 & 6 & 3 & 1 & 0 & 0 \\
    0 & 1 & 2 & 7 & 5 & 3 & 1 & 0 \\
    0 & 0 & 1 & 4 & 3 & 2 & 1 & 1 
    \end{bmatrix}\]
    \begin{rem}
    The entries of $\rt(\mathsf{A})^\top \mathsf{A}$ are dot products between columns of $\rt(\mathsf{A})$ and columns of $\mathsf{A}$, so $\rt(\mathsf{A})^\top \mathsf{A}$ has a strip of 0s above the diagonal bounded by two diagonals of $\pm1$s.
    \end{rem}

    \item Take the entries above the diagonal of $\rt(\mathsf{A})^\top \mathsf{A}$, multiply them by $(-1)^{k-1}$, and slide them to the left of the rest of the matrix to form a parallelogram.
    
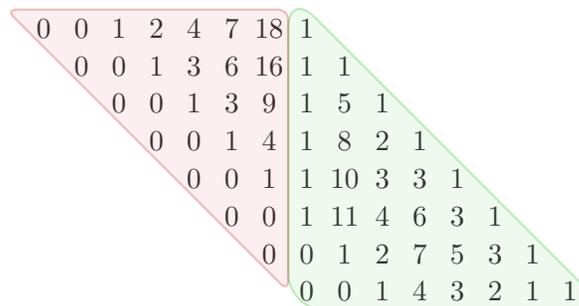
\[
    \begin{tikzpicture}[baseline=(M-6-1.base),
    ampersand replacement=\&,
    ]
    \matrix[matrix of math nodes,
        matrix anchor = M-1-8.center,
        origin/.style={},
        throw/.style={},
        pivot/.style={draw,circle,inner sep=0.25mm,minimum size=2mm},       
        nodes in empty cells,
        inner sep=0pt,
        nodes={anchor=center},
        column sep={.5cm,between origins},
        row sep={.5cm,between origins},
    ] (M) at (0,0) {
    0 \& 0 \& 1 \& 2 \& 4 \& 7 \& 18 \& 1 \&  \&  \&  \&  \&  \&  \&  \\
    \& 0 \& 0 \& 1 \& 3 \& 6 \& 16 \& 1 \& 1 \&  \&  \&  \&  \&  \&   \\
    \& \& 0 \& 0 \& 1 \& 3 \& 9 \& 1 \& 5 \& 1 \&  \&  \&  \&  \&  \\
    \& \& \& 0 \& 0 \& 1 \& 4 \& 1 \& 8 \& 2 \& 1 \&  \&  \&  \&  \\
    \& \& \& \& 0 \& 0 \& 1 \& 1 \& 10 \& 3 \& 3 \& 1 \&  \&  \& \\
    \& \& \& \& \& 0 \& 0 \& 1 \& 11 \& 4 \& 6 \& 3 \& 1 \&  \& \\
    \& \& \& \& \& \& 0 \& 0 \& 1 \& 2 \& 7 \& 5 \& 3 \& 1 \& \\
    \& \& \& \& \& \& \& 0 \& 0 \& 1 \& 4 \& 3 \& 2 \& 1 \& 1 \\
    };
    \draw[thick,dark green,fill=dark green!25,opacity=.25,rounded corners=5mm] ($(M-1-8.center)+(-.25,.5)$) -- ($(M-8-15.center)+(.5,-.25)$) -- ($(M-8-8.center)+(-.25,-.25)$) -- cycle;
    \draw[thick,dark red,fill=dark red!25,opacity=.25,rounded corners] ($(M-1-7.center)+(.25,.25)$) -- ($(M-7-7.center)+(.25,-.5)$) -- ($(M-1-1.center)+(-.5,.25)$) -- cycle;
    \end{tikzpicture}
    \]
    \item Rotate the resulting parallelogram $45^\circ$ counter-clockwise, delete any $0$s below the bottom row of $1$s, and duplicate the rest $n$-periodically in an infinite horizontal strip.
    
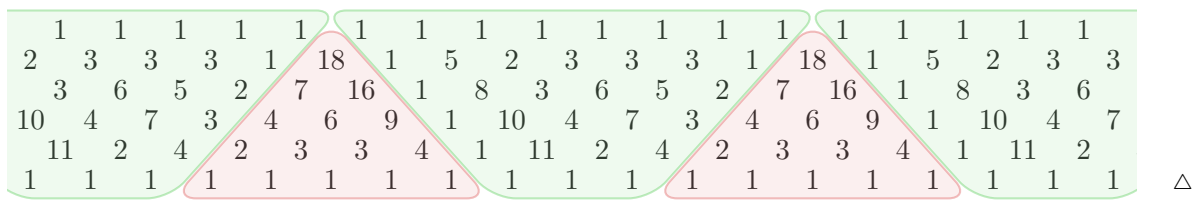
\[
    \begin{tikzpicture}[baseline=(M-6-1.base),
    ampersand replacement=\&,
    ]
    \clip[use as bounding box] (-7.5,.3) rectangle (7.5,-2.3);
    \matrix[matrix of math nodes,
        matrix anchor = M-1-29.center,
        origin/.style={},
        throw/.style={},
        pivot/.style={draw,circle,inner sep=0.25mm,minimum size=2mm},       
        nodes in empty cells,
        inner sep=0pt,
        nodes={anchor=center},
        column sep={.4cm,between origins},
        row sep={.4cm,between origins},
    ] (M) at (0,0) {
    \& \& \& \& \& 1 \& \& 1 \& \& 1 \& \& 1 \& \& 1 \& \& 1 \& \& 1 \& \& 1 \& \& 1 \& \& 1 \& \& 1 \& \& 1 \& \& 1 \& \& 1 \& \& 1 \& \& 1 \& \& 1 \& \& 1 \& \& 1 \& \& 1 \& \& 1 \& \& 1 \& \& 1 \& \& 1 \& \& \\
    \& \& \& \& 18 \& \& 1 \& \& 5 \& \& 2 \& \& 3 \& \& 3 \& \& 3 \& \& 1 \& \& 18 \& \& 1 \& \& 5 \& \& 2 \& \& 3 \& \& 3 \& \& 3 \& \& 1 \& \& 18 \& \& 1 \& \& 5 \& \& 2 \& \& 3 \& \& 3 \& \& 3 \& \& 1 \& \& \\
    \& \& \& 7 \& \& 16 \& \& 1 \& \& 8 \& \& 3 \& \& 6 \& \& 5 \& \& 2 \& \& 7 \& \& 16 \& \& 1 \& \& 8 \& \& 3 \& \& 6 \& \& 5 \& \& 2 \& \& 7 \& \& 16 \& \& 1 \& \& 8 \& \& 3 \& \& 6 \& \& 5 \& \& 2 \& \& \\
    \& \& 4 \& \& 6 \& \& 9 \& \& 1 \& \& 10 \& \& 4 \& \& 7 \& \& 3 \& \& 4 \& \& 6 \& \& 9 \& \& 1 \& \& 10 \& \& 4 \& \& 7 \& \& 3 \& \& 4 \& \& 6 \& \& 9 \& \& 1 \& \& 10 \& \& 4 \& \& 7 \& \& 3 \& \& \\
    \& 2 \& \& 3 \& \& 3 \& \& 4 \& \& 1 \& \& 11 \& \& 2 \& \& 4 \& \& 2 \& \& 3 \& \& 3 \& \& 4 \& \& 1 \& \& 11 \& \& 2 \& \& 4 \& \& 2 \& \& 3 \& \& 3 \& \& 4 \& \& 1 \& \& 11 \& \& 2 \& \& 4 \& \& \\
    1 \& \& 1 \& \& 1 \& \& 1 \& \& 1 \& \& 1 \& \& 1 \& \& 1 \& \& 1 \& \& 1 \& \& 1 \& \& 1 \& \& 1 \& \& 1 \& \& 1 \& \& 1 \& \& 1 \& \& 1 \& \& 1 \& \& 1 \& \& 1 \& \& 1 \& \& 1 \& \& 1 \& \& \\
    };
    \draw[thick,dark green,fill=dark green!25,opacity=.25,rounded corners=3mm] ($(M-1-6.center)+(-.5,.25)$) -- ($(M-1-20.center)+(.5,.25)$) -- ($(M-6-15.center)+(.25,-.25)$) -- ($(M-6-11.center)+(-.25,-.25)$) -- cycle;
    \draw[thick,dark red,fill=dark red!25,opacity=.25,rounded corners=3mm] ($(M-2-5.center)+(0,.5)$) -- ($(M-6-9.center)+(.5,-.25)$) -- ($(M-6-1.center)+(-.5,-.25)$) -- cycle;
    \draw[thick,dark green,fill=dark green!25,opacity=.25,rounded corners=3mm] ($(M-1-22.center)+(-.5,.25)$) -- ($(M-1-36.center)+(.5,.25)$) -- ($(M-6-31.center)+(.25,-.25)$) -- ($(M-6-27.center)+(-.25,-.25)$) -- cycle;
    \draw[thick,dark red,fill=dark red!25,opacity=.25,rounded corners=3mm] ($(M-2-21.center)+(0,.5)$) -- ($(M-6-25.center)+(.5,-.25)$) -- ($(M-6-17.center)+(-.5,-.25)$) -- cycle;
    \draw[thick,dark green,fill=dark green!25,opacity=.25,rounded corners=3mm] ($(M-1-38.center)+(-.5,.25)$) -- ($(M-1-52.center)+(.5,.25)$) -- ($(M-6-47.center)+(.25,-.25)$) -- ($(M-6-43.center)+(-.25,-.25)$) -- cycle;
    \draw[thick,dark red,fill=dark red!25,opacity=.25,rounded corners=3mm] ($(M-2-37.center)+(0,.5)$) -- ($(M-6-41.center)+(.5,-.25)$) -- ($(M-6-33.center)+(-.5,-.25)$) -- cycle;
    \end{tikzpicture}
    \qedhere
    \]
\end{enumerate}
\end{cons}

\newpage

\subsection{Relation to the Grassmannian}

For all $\mathsf{G}\in \mathrm{SL}(k)$, if $\mathsf{A}$ is consecutively unimodular, then so is $\mathsf{GA}$, and $\Fr(\mathsf{GA})= \Fr(\mathsf{A})$.
Therefore, $\Fr$ descends to a well-defined map
\[ \mathrm{SL}(k)\backslash \{ \text{consecutively unimodular $k\times n$-matrices}\} \xrightarrow{\Fr} \{\text{$\mathrm{SL}(k)$-friezes of height $(n-k)$}\} \]

\begin{thm}\label{thm: Frbijection1}
The map $\Fr$ is a bijection.
\end{thm}

\begin{proof}[Proof sketch]
An inverse of $\Fr$ is given by sending $\C$ to the submatrix $\C_{[n-k+1,n],[1,n]}$.
\end{proof}

This can be related to the Grassmannian $\mathrm{Gr}(k,n)$ of $k$-planes in $n$-space, via the bijection
\[ \mathrm{GL}(k)\backslash \{ \text{$k\times n$-matrices of rank $k$}\} \xrightarrow{\mathrm{rowspan}} \mathrm{Gr}(k,n) \]
Restricting this to $\mathrm{SL}(k)$-orbits of consecutively unimodular matrices gives a closed inclusion
\[ \mathrm{SL}(k)\backslash \{ \text{consecutively unimodular $k\times n$-matrices}\} \xrightarrow{\mathrm{rowspan}} \mathrm{Gr}(k,n) \]
If we let $\mathrm{Gr}_{\uni}(k,n)$ denote the image of this map, then Theorem \ref{thm: Frbijection1} descends to a bijection
\[ \Fr:\mathrm{Gr}_{\uni}(k,n) \xrightarrow{\sim} \{\text{$\mathrm{SL}(k)$-friezes of height $(n-k)$}\} \]

\begin{rem}
This parametrization of $\mathrm{SL}(k)$-friezes by a subvariety of a Grassmannian has appeared in \cite[Prop.~3.2.1]{MGOST14} and \cite[Theorem 3.1]{BFGST21}.
\end{rem}

Now that we have introduced the twist of a consecutively unimodular matrix, we can relate it to the ($n$-truncated) dual of a frieze.
Given a $k\times n$ matrix $\mathsf{A}$ of rank $k$, define a \textbf{positive complement} of $\mathsf{A}$ to be any $(n-k)\times n$ matrix $\mathsf{A}^\pperp$ such that, for every $k$-element set $I\subset [n]$,
\[ \det(\mathsf{A}_I) = \det(\mathsf{A}^\pperp_{[n]\smallsetminus I}) \]
Such a matrix may be constructed by first finding a matrix $\mathsf{B}$ whose rows span the orthogonal complement to the rowspan of $\mathsf{A}$, negating odd numbered columns of $\mathsf{B}$, and then rescaling until the above identities are satisfied. The following facts are straightforward from the definition.
\begin{itemize}
    \item If $\mathsf{A}^\pperp$ is a positive complement of $\mathsf{A}$, then $\mathsf{A}$ is a positive complement of $\mathsf{A}^\pperp$.
    \item Positive complements exist and are unique up to left multiplication by $\mathrm{SL}(n-k)$.
    \item The positive complement of a consecutively unimodular matrix is consecutively unimodular.
\end{itemize}
Therefore, sending a matrix to its positive complement descends to well-defined isomorphisms
\[ \pperp: \mathrm{Gr}(k,n) \xrightarrow{\sim} \mathrm{Gr}(n-k,n) \text{ and } \pperp: \mathrm{Gr}_{\uni}(k,n) \xrightarrow{\sim} \mathrm{Gr}_{\uni}(n-k,n) \]
whose inverse is $\pperp$. This is related to $\dagger$ and $\tau$ by the following commutative diagram of bijections.
%
\[ \begin{tikzpicture}[xscale=3.5, yscale=1.75]
    \node (a0) at (0,0) {$\mathrm{Gr}_{\text{uni}}(k,n)$};
    \node (a1) at (1,.33) {$\mathrm{Gr}_{\text{uni}}(k,n)$};
    \node (c1) at (1,-.33) {$\mathrm{Gr}_{\text{uni}}(n-k,n)$};
    \node (a2) at (2,0) {$\mathrm{Gr}_{\text{uni}}(n-k,n)$};
    \node (b0) at (0,-1) {$\{\text{$\mathrm{SL}(k)$-friezes of height $(n-k)$}\}$};
    \node (b2) at (2,-1) {$\{\text{$\mathrm{SL}(n-k)$-friezes of height $k$}\}$};

    \draw[->] (a0) to node[above] {$\tau$} (a1);
    \draw[<->] (a1) to node[above] {$\ddagger$} (a2);
    \draw[->] (a2) to node[below] {$\tau$} (c1);
    \draw[<->] (c1) to node[below] {$\ddagger$} (a0);
    \draw[<->] (b0) to node[above] {$\dagger$} (b2);
    
    \draw[->] (a0) to node[left] {$\Fr$} (b0);
    \draw[->] (a2) to node[left] {$\Fr$} (b2);
\end{tikzpicture}\]
The commutativity will be proven as a special case of Theorem \ref{thm: dualities} and Proposition \ref{prop: twistduality}.

\newpage

\section{Juggler's friezes}

\label{section: jugfrieze}

\begin{center}
\emph{Theorems in Sections \ref{section: jugfrieze} and \ref{section: jugcons} are proven in Section \ref{section: proofs}.}
\end{center}

We now describe a generalization of $\mathrm{SL}(k)$-friezes, which we call \emph{juggler's friezes}, for which analogs of each of the preceding definitions and constructions hold. 
The simplest motivation for this generalization is to observe that Construction \ref{cons: unwrap} can be applied to a more general class of matrices (\emph{$\pi$-unimodular matrices}), yielding grids of numbers with certain minors necessarily equal to $0$, $1$, or $-1$.
However, instead of starting with this construction, we mirror the order of the preceding exposition: a definition in terms of certain minors, followed by several equivalent characterizations, and then conclude with the analog of Construction \ref{cons: unwrap}.

\subsection{Juggling functions}

First, we need a combinatorial object to describe the shape of our generalized friezes.
A \textbf{juggling function of period\footnote{We emphasize that the choice of a distinguished period $n$ is part of the data of a juggling function.} $n$} is a bijection $\pi:\mathbb{Z}\rightarrow \mathbb{Z}$ such that, for all $i$, $i\leq \pi(i) \leq i+n$ and $\pi(i+n)=\pi(i)+n$. Juggling functions are a set of combinatorial objects in bijection with \emph{positroids}; matroids with positive representations (see \cite{KLS13}).

\begin{rem}
A juggling function $\pi$ can be thought of as a transcription of how a juggler catches and throws balls over time. At each moment $a$, the juggler catches the ball they threw at moment $\pi^{-1}(a)$ and immediately throws it again to be caught at moment $\pi(a)$...unless $\pi(a)=a$, in which case they neither catch nor throw a ball.
\end{rem}

The \textbf{number of balls} in a juggling function of period $n$ is 
\[ \frac{1}{n} \sum_{i=1}^n (\pi(i)-i) \]
which is always an integer. The \textbf{siteswap notation} for a juggling function is the list of integers
\[ (\pi(1)-1,\pi(2)-2,...,\pi(n)-n)\]
A \textbf{loop} (resp.~a \textbf{coloop}) of a juggling function is an $a\in \mathbb{Z}$ such that $\pi(a)=a$ (resp.~$\pi(a)=a+n$).\footnote{This terminology is inherited from the matroid associated to $\pi$.}

\begin{rem}
Our definitions and theorems will frequently have extra cases for loops or coloops, but these rarely arise in interesting examples and can be safely ignored in a first reading.
\end{rem}

\begin{ex}
The juggling function with siteswap notation `53635514' is the function with values
\[ \begin{array}{|c|cccccccc|}
\hline
i & 1 & 2 & 3 & 4 & 5 & 6 & 7 & 8  \\
\hline
\pi(i) & 6 & 5 & 9 & 7 & 10 & 11 & 8 & 12 \\
\hline
\end{array}\]
extended to the rest of $\mathbb{Z}$ via the rule that $\pi(i+ 8) = \pi(i)+ 8$.
\end{ex}

Given a juggling function $\pi$ of period $n$, the \textbf{dual juggling function} $\pi^\dagger$
\[ \pi^\dagger(a) := \pi^{-1}(a) +n \]
is a juggling function with period $n$.
If $\pi$ has $h$-many balls, then $\pi^\dagger$ has $(n-h)$-many balls.
\begin{ex}
The dual juggling function of the preceding example is $8$-periodic with values
\[ \begin{array}{|c|cccccccc|}
\hline
i & 1 & 2 & 3 & 4 & 5 & 6 & 7 & 8  \\
\hline
\pi^\dagger(i) & 3 & 5 & 6 & 8 & 10 & 9 & 12 & 15 \\
\hline
\end{array}\]
which has siteswap notation `23345357'.
\end{ex}

\subsection{$\pi$-prefriezes}

Before defining juggler's friezes, we generalize `prefriezes' to configurations of numbers with a ragged lower edge.
Unlike prefriezes in the previous sense, which always had $1$s on the boundary, this new object may have $1$s or $-1$s on the boundary, 
depending on the parity of the following set. Given a juggling function $\pi$ and integers $a,b\in \mathbb{Z}$, define
\[ S_\pi (a,b) : = \{ i\in \mathbb{Z} \mid a<i \text{ and } \pi(i)< b \} = (a,b) \cap \pi^{-1}(a,b) \]

\begin{rem}
In the juggling perspective, $S_\pi(a,b)$ is the set of times a ball is thrown after moment $a$ that will be caught before moment $b$. 
\end{rem}

\begin{defn}
Given a juggling function $\pi$, 
a \textbf{$\pi$-prefrieze} consists of rows of numbers offset in a diamond grid $\C$ with entries indexed as in Section \ref{section: matrices}\footnote{Or, equivalently, a $\mathbb{Z}\times \mathbb{Z}$-matrix rotated $45^\circ$ counterclockwise.}, such that
\[ \C_{a,b} = \left\{
\begin{array}{cl}
1 & \text{if } a=b\\
(-1)^{|S_\pi(b,a)|} & \text{if }a=\pi(b)\\
0 & \text{if } 
a\not\in [b,\pi(b)] \text{ or } b\not\in [\pi^{-1}(a),a] 
\end{array}
\right\} \]
\end{defn}
\noindent Note that this definition allows $\C_{a,b}$ to be arbitrary if $\pi^{-1}(a) < b < a < \pi(b)$.

\begin{warn}
The period $n$ of the juggling function $\pi$ is not used in the definition of a $\pi$-prefrieze; however, this data will be essential for the \emph{frieze} and \emph{tameness} conditions in the next section.
\end{warn}


\begin{cons}\label{cons: pipre}
We outline a more visually-oriented construction of a $\pi$-prefrieze. 
First, construct the \textbf{diagram of $\pi$} as follows.
\begin{itemize}
    \item Draw a row of green circles (i.e.~at position $\C_{a,a}$ for each $a\in \mathbb{Z}$).
    \item Draw a blue circle at position $\C_{\pi(a),a}$ for each $a\in \mathbb{Z}$.\footnote{If $\pi(a)=a$, the circle $\C_{a,a}$ will be both blue and green, with no adjacent lines.}
    \item Draw a line between each pair of circles which share a diagonal (i.e.\ in the same row or column of the corresponding matrix).
%
%
%
%
\end{itemize}
As a running example, the juggling function $\pi$ with siteswap notation `53635514' has diagram
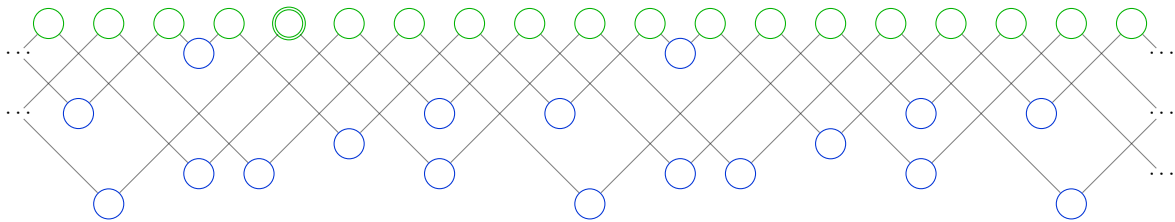
\[\begin{tikzpicture}[baseline=(current bounding box.center),
    ampersand replacement=\&,
    ]
    \clip[use as bounding box] (-9.5,-2.7) rectangle (6.3,0.3);
    \matrix[matrix of math nodes,
        matrix anchor = M-2-24.center,
        nodes in empty cells,
        inner sep=0pt,
        gthrow/.style={dark green,draw,circle,inner sep=0mm,minimum size=4mm},
        rthrow/.style={dark red,draw,circle,inner sep=0mm,minimum size=4mm},
        bthrow/.style={dark blue,draw,circle,inner sep=0mm,minimum size=4mm},
        pthrow/.style={dark purple,draw,circle,inner sep=0mm,minimum size=4mm},
        nodes={anchor=center,node font=\scriptsize},
        column sep={0.4cm,between origins},
        row sep={0.4cm,between origins},
    ] (M) at (0,0) {
 \&  \&  \&  \&  \&  \&  \&  \&  \&  \&  \&  \&  \&  \&  \&  \&  \&  \&  \&  \&  \&  \&  \&  \&  \&  \&  \&  \&  \&  \&  \&  \&  \&  \&  \&  \&  \&  \&  \\
 \& |[gthrow]|  \&  \& |[gthrow]|  \&  \& |[gthrow]|  \&  \& |[gthrow]|  \&  \& |[gthrow,double]|  \&  \& |[gthrow]|  \&  \& |[gthrow]|  \&  \& |[gthrow]|  \&  \& |[gthrow]|  \&  \& |[gthrow]|  \&  \& |[gthrow]|  \&  \& |[gthrow]|  \&  \& |[gthrow]|  \&  \& |[gthrow]|  \&  \& |[gthrow]|  \&  \& |[gthrow]|  \&  \& |[gthrow]|  \&  \& |[gthrow]|  \&  \& |[gthrow]|  \&  \\
\cdots \&  \&  \&  \&  \&  \& |[bthrow]|  \&  \&  \&  \&  \&  \&  \&  \&  \&  \&  \&  \&  \&  \&  \&  \& |[bthrow]|  \&  \&  \&  \&  \&  \&  \&  \&  \&  \&  \&  \&  \&  \&  \&  \& \cdots \\
 \&  \&  \&  \&  \&  \&  \&  \&  \&  \&  \&  \&  \&  \&  \&  \&  \&  \&  \&  \&  \&  \&  \&  \&  \&  \&  \&  \&  \&  \&  \&  \&  \&  \&  \&  \&  \&  \&  \\
\cdots \&  \& |[bthrow]|  \&  \&  \&  \&  \&  \&  \&  \&  \&  \&  \&  \& |[bthrow]|  \&  \&  \&  \& |[bthrow]|  \&  \&  \&  \&  \&  \&  \&  \&  \&  \&  \&  \& |[bthrow]|  \&  \&  \&  \& |[bthrow]|  \&  \&  \&  \& \cdots \\
 \&  \&  \&  \&  \&  \&  \&  \&  \&  \&  \& |[bthrow]|  \&  \&  \&  \&  \&  \&  \&  \&  \&  \&  \&  \&  \&  \&  \&  \& |[bthrow]|  \&  \&  \&  \&  \&  \&  \&  \&  \&  \&  \&  \\
 \&  \&  \&  \&  \&  \& |[bthrow]|  \&  \& |[bthrow]|  \&  \&  \&  \&  \&  \& |[bthrow]|  \&  \&  \&  \&  \&  \&  \&  \& |[bthrow]|  \&  \& |[bthrow]|  \&  \&  \&  \&  \&  \& |[bthrow]|  \&  \&  \&  \&  \&  \&  \&  \& \cdots \\
 \&  \&  \& |[bthrow]|  \&  \&  \&  \&  \&  \&  \&  \&  \&  \&  \&  \&  \&  \&  \&  \& |[bthrow]|  \&  \&  \&  \&  \&  \&  \&  \&  \&  \&  \&  \&  \&  \&  \&  \& |[bthrow]|  \&  \&  \&  \\
 \&  \&  \&  \&  \&  \&  \&  \&  \&  \&  \&  \&  \&  \&  \&  \&  \&  \&  \&  \&  \&  \&  \&  \&  \&  \&  \&  \&  \&  \&  \&  \&  \&  \&  \&  \&  \&  \& \\
    };

    \draw[traj] (M-3-1) to (M-2-2) to (M-7-7) to (M-2-12) to (M-5-15) to (M-2-18) to (M-7-23) to (M-2-28) to (M-5-31) to (M-2-34) to (M-7-39);
    \draw[traj] (M-5-1) to (M-2-4) to (M-7-9) to (M-2-14) to (M-8-20) to (M-2-26) to (M-7-31) to (M-2-36) to (M-5-39);
    \draw[traj] (M-3-1) to (M-5-3) to (M-2-6) to (M-3-7) to (M-2-8) to (M-6-12) to (M-2-16) to (M-5-19) to (M-2-22) to (M-3-23) to (M-2-24) to (M-6-28) to (M-2-32) to (M-5-35) to (M-2-38) to (M-3-39);
    \draw[traj] (M-5-1) to (M-8-4) to (M-2-10) to (M-7-15) to (M-2-20) to (M-7-25) to (M-2-30) to (M-8-36) to (M-5-39);
\end{tikzpicture}\]
The doubled green circle above denotes the position $\C_{1,1}$.

\newpage

A $\pi$-prefrieze can be constructed from the diagram of $\pi$ by putting:
\begin{itemize}
    \item a $1$ in each green circle,
    \item a $(-1)^s$ in each blue circle, where $s$ is the number of blue circles in the cone above it, and 
    \item any number at each intersection between lines.
\end{itemize}
All other entries are implicitly zero. An example of a $\pi$-prefrieze on the diagram above is below.
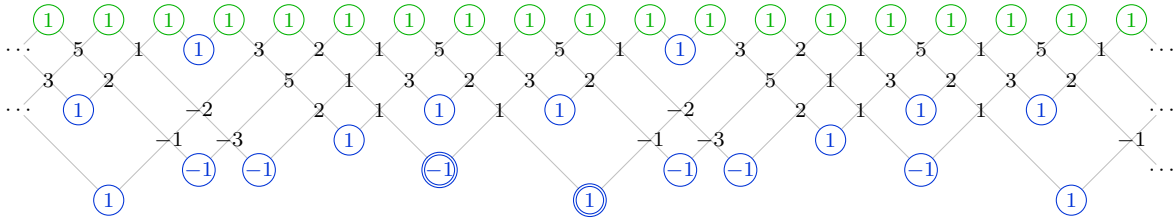
\[\begin{tikzpicture}[baseline=(current bounding box.south),
    ampersand replacement=\&,
    ]
    \clip[use as bounding box] (-9.5,-2.7) rectangle (6.3,0.3);
    \matrix[matrix of math nodes,
        matrix anchor = M-2-24.center,
        nodes in empty cells,
        inner sep=0pt,
        gthrow/.style={dark green,draw,circle,inner sep=0mm,minimum size=4mm},
        rthrow/.style={dark red,draw,circle,inner sep=0mm,minimum size=4mm},
        bthrow/.style={dark blue,draw,circle,inner sep=0mm,minimum size=4mm},
        pthrow/.style={dark purple,draw,circle,inner sep=0mm,minimum size=4mm},
        nodes={anchor=center,node font=\scriptsize},
        column sep={0.4cm,between origins},
        row sep={0.4cm,between origins},
    ] (M) at (0,0) {
 \&  \&  \&  \&  \&  \&  \&  \&  \&  \&  \&  \&  \&  \&  \&  \&  \&  \&  \&  \&  \&  \&  \&  \&  \&  \&  \&  \&  \&  \&  \&  \&  \&  \&  \&  \&  \&  \&  \\
 \& |[gthrow]| 1 \&  \& |[gthrow]| 1 \&  \& |[gthrow]| 1 \&  \& |[gthrow]| 1 \&  \& |[gthrow]| 1 \&  \& |[gthrow]| 1 \&  \& |[gthrow]| 1 \&  \& |[gthrow]| 1 \&  \& |[gthrow]| 1 \&  \& |[gthrow]| 1 \&  \& |[gthrow]| 1 \&  \& |[gthrow]| 1 \&  \& |[gthrow]| 1 \&  \& |[gthrow]| 1 \&  \& |[gthrow]| 1 \&  \& |[gthrow]| 1 \&  \& |[gthrow]| 1 \&  \& |[gthrow]| 1 \&  \& |[gthrow]| 1 \&  \\
\cdots \&  \& 5 \&  \& 1 \&  \& |[bthrow]| 1 \&  \& 3 \&  \& 2 \&  \& 1 \&  \& 5 \&  \& 1 \&  \& 5 \&  \& 1 \&  \& |[bthrow]| 1 \&  \& 3 \&  \& 2 \&  \& 1 \&  \& 5 \&  \& 1 \&  \& 5 \&  \& 1 \&  \& \cdots \\
 \& 3 \&  \& 2 \&  \&  \&  \&  \&  \& 5 \&  \& 1 \&  \& 3 \&  \& 2 \&  \& 3 \&  \& 2 \&  \&  \&  \&  \&  \& 5 \&  \& 1 \&  \& 3 \&  \& 2 \&  \& 3 \&  \& 2 \&  \&  \&  \\
\cdots \&  \& |[bthrow]| 1 \&  \&  \&  \& -2 \&  \&  \&  \& 2 \&  \& 1 \&  \& |[bthrow]| 1 \&  \& 1 \&  \& |[bthrow]| 1 \&  \&  \&  \& -2 \&  \&  \&  \& 2 \&  \& 1 \&  \& |[bthrow]| 1 \&  \& 1 \&  \& |[bthrow]| 1 \&  \&  \&  \& \cdots \\
 \&  \&  \&  \&  \& -1 \&  \& -3 \&  \&  \&  \& |[bthrow]| 1 \&  \&  \&  \&  \&  \&  \&  \&  \&  \& -1 \&  \& -3 \&  \&  \&  \& |[bthrow]| 1 \&  \&  \&  \&  \&  \&  \&  \&  \&  \& -1 \&  \\
 \&  \&  \&  \&  \&  \& |[bthrow]| -1 \&  \& |[bthrow]| -1 \&  \&  \&  \&  \&  \& |[bthrow,double]| -1 \&  \&  \&  \&  \&  \&  \&  \& |[bthrow]| -1 \&  \& |[bthrow]| -1 \&  \&  \&  \&  \&  \& |[bthrow]| -1 \&  \&  \&  \&  \&  \&  \&  \& \cdots \\
 \&  \&  \& |[bthrow]| 1 \&  \&  \&  \&  \&  \&  \&  \&  \&  \&  \&  \&  \&  \&  \&  \& |[bthrow,double]| 1 \&  \&  \&  \&  \&  \&  \&  \&  \&  \&  \&  \&  \&  \&  \&  \& |[bthrow]| 1 \&  \&  \&  \\
 \&  \&  \&  \&  \&  \&  \&  \&  \&  \&  \&  \&  \&  \&  \&  \&  \&  \&  \&  \&  \&  \&  \&  \&  \&  \&  \&  \&  \&  \&  \&  \&  \&  \&  \&  \&  \&  \& \\
    };

    \draw[traj,opacity=.25] (M-3-1) to (M-2-2) to (M-7-7) to (M-2-12) to (M-5-15) to (M-2-18) to (M-7-23) to (M-2-28) to (M-5-31) to (M-2-34) to (M-7-39);
    \draw[traj,opacity=.25] (M-5-1) to (M-2-4) to (M-7-9) to (M-2-14) to (M-8-20) to (M-2-26) to (M-7-31) to (M-2-36) to (M-5-39);
    \draw[traj,opacity=.25] (M-3-1) to (M-5-3) to (M-2-6) to (M-3-7) to (M-2-8) to (M-6-12) to (M-2-16) to (M-5-19) to (M-2-22) to (M-3-23) to (M-2-24) to (M-6-28) to (M-2-32) to (M-5-35) to (M-2-38) to (M-3-39);
    \draw[traj,opacity=.25] (M-5-1) to (M-8-4) to (M-2-10) to (M-7-15) to (M-2-20) to (M-7-25) to (M-2-30) to (M-8-36) to (M-5-39);
\end{tikzpicture}
\]
Of the two doubled blue circles,
$\C_{6,1}=(-1)^1$, as there is a single blue circle in the cone above (at $\C_{5,2}$), and $\C_{9,3}=(-1)^2$, as there are two blue circles in the cone above (at $\C_{7,4}$ and $\C_{8,7}$).
\end{cons}

\subsection{Juggler's friezes}

We can now state the primary definition of this note: a generalization of friezes to shapes defined by juggling functions.

\begin{defn}\label{defn: pifrieze}
Given a juggler's function $\pi$ of period $n$, a \textbf{$\pi$-frieze} is a $\pi$-prefrieze satisfying the following conditions.
\begin{itemize}
    \item (the \textbf{frieze} condition) For all $a, b\in \mathbb{Z}$ with $\pi^\dagger(a) \leq b < a+n \leq \pi(b)$,
    \[ \det (\C_{[a,b] \smallsetminus I , [a,b] \smallsetminus \pi^\dagger(I)}) = 1 \]
    where $I:= \{ i \in \mathbb{Z} \mid a \leq i \text{ and } \pi^\dagger(i) \leq b\}$; note this coincides with $S_{\pi^\dagger}(a-1,b+1)$.
    \item (the \textbf{tameness} condition) For all $a,b\in \mathbb{Z}$ with $\pi^\dagger(a) <b<a+n$ or $b<a+n<\pi(b)$, 
    \[ \det (\C_{(a,b] \smallsetminus J , [a,b) \smallsetminus \pi^\dagger(J)}) = 0 \]
    where $J:= \{ i \in \mathbb{Z} \mid a < i \text{ and } \pi^\dagger(i) < b\}$; note this coincides with $S_{\pi^\dagger}(a,b)$.
\end{itemize}
\end{defn}
\noindent 
While the determinants in the definition are not solid minors, a diagrammatic construction of these determinants is given in the next two sections.
Absent an explicit choice of $\pi$, we refer to this genre of friezes as \textbf{juggler's friezes}.

\begin{ex}
The running example of a $\pi$-prefrieze in Construction \ref{cons: pipre} is, in fact, a $\pi$-frieze.
\end{ex}

%

\begin{warn}
Note that we include tameness in the definition, where others might call these \emph{tame juggler's friezes}. One reason for this is that \emph{wild juggler's friezes} ($\pi$-prefriezes only satisfying the frieze condition) can be poorly behaved, even compared with wild $\mathrm{SL}(k)$-friezes. E.g.~tame $\pi$-friezes may not form an irreducible component of the space of wild $\pi$-friezes.
\end{warn}

\begin{rem}\label{rem: redundant}
If $\C$ is a $\pi$-prefrieze, then $\det (\C_{[a,b] \smallsetminus I , [a,b]
\smallsetminus \pi^\dagger(I)}) = 1$ for any $a,b$ with $a\leq b <a+n$; that
is, the conditions $\pi^\dagger(a)\leq b$ and $a+n\leq \pi(b)$ may be dropped
from the frieze condition. However, this larger class of determinants
immediately reduces to the given ones. When $\pi^\dagger(a)\not\leq b$ or
$a+n\not\leq \pi(b)$, the matrix $\C_{[a,b] \smallsetminus I , [a,b]
\smallsetminus \pi^\dagger(I)}$ can be block decomposed into a lower
unitriangular matrix and a matrix of the form $\C_{[a',b'] \smallsetminus I' ,
[a',b'] \smallsetminus \pi^\dagger(I')}$, where $\pi^\dagger(a')\leq b'$ and
$a'+n\leq \pi(b')$.\footnote{However, the inequalities in the tameness
condition cannot be dropped or even weakened to `and'; see Section
\ref{section: visualtame}.}
\end{rem}

\subsection{Visualizing the frieze condition}\label{section: visualfrieze}

To visualize the frieze condition, first take the diagram of $\pi^\dagger$, reflect it vertically, make the blue circles red, and attach it to the $\pi$-prefrieze $\C$ along the green circles. In the running example:
    
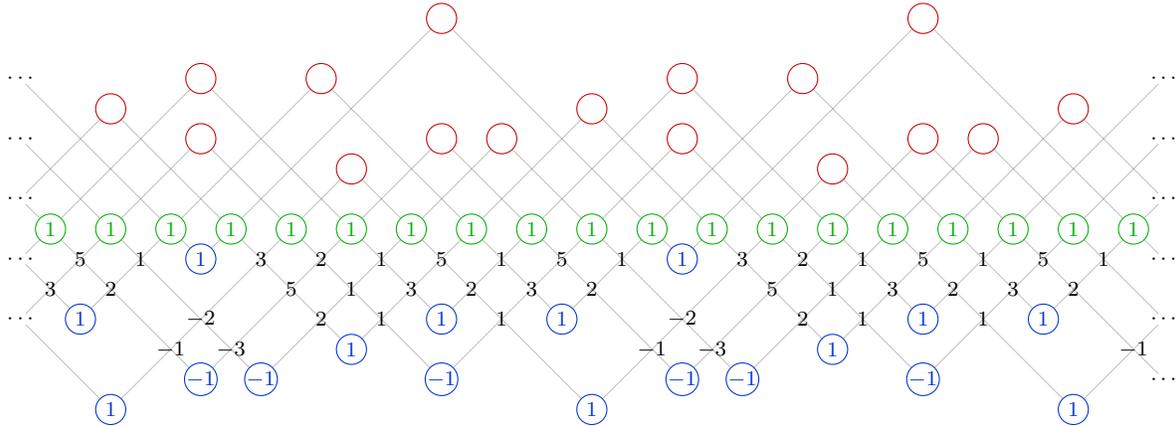
\[\begin{tikzpicture}[baseline=(current bounding box.center),
    ampersand replacement=\&,
    ]
    \matrix[matrix of math nodes,
        matrix anchor = M-1-20.center,
        nodes in empty cells,
        inner sep=0pt,
        gthrow/.style={dark green,draw,circle,inner sep=0mm,minimum size=4mm},
        rthrow/.style={dark red,draw,circle,inner sep=0mm,minimum size=4mm},
        bthrow/.style={dark blue,draw,circle,inner sep=0mm,minimum size=4mm},
        bbthrow/.style={dark blue,draw,circle,inner sep=0mm,minimum size=4mm},
        pthrow/.style={dark purple,draw,circle,inner sep=0mm,minimum size=4mm},
        nodes={anchor=center,node font=\scriptsize},
        column sep={0.4cm,between origins},
        row sep={0.4cm,between origins},
    ] (M) at (0,0) {
 \&  \&  \&  \&  \&  \&  \&  \&  \&  \&  \&  \&  \&  \& |[rthrow]| \&  \&  \&  \&  \&  \&  \&  \&  \&  \&  \&  \&  \&  \&  \&  \& |[rthrow]| \&  \&  \&  \&  \&  \&  \&  \&  \\
 \&  \&  \&  \&  \&  \&  \&  \&  \&  \&  \&  \&  \&  \&  \&  \&  \&  \&  \&  \&  \&  \&  \&  \&  \&  \&  \&  \&  \&  \&  \&  \&  \&  \&  \&  \&  \&  \&  \\
\cdots \&  \&  \&  \&  \&  \& |[rthrow]| \&  \&  \&  \& |[rthrow]| \&  \&  \&  \&  \&  \&  \&  \&  \&  \&  \&  \& |[rthrow]| \&  \&  \&  \& |[rthrow]| \&  \&  \&  \&  \&  \&  \&  \&  \&  \&  \&  \& \cdots \\
 \&  \&  \& |[rthrow]| \&  \&  \&  \&  \&  \&  \&  \&  \&  \&  \&  \&  \&  \&  \&  \& |[rthrow]| \&  \&  \&  \&  \&  \&  \&  \&  \&  \&  \&  \&  \&  \&  \&  \& |[rthrow]|  \&  \&  \&  \\
\cdots \&  \&  \&  \&  \&  \& |[rthrow]| \&  \&  \&  \&  \&  \&  \&  \& |[rthrow]| \&  \& |[rthrow]| \&  \&  \&  \&  \&  \& |[rthrow]| \&  \&  \&  \&  \&  \&  \&  \& |[rthrow]| \&  \& |[rthrow]| \&  \&  \&  \&  \&  \& \cdots \\
 \&  \&  \&  \&  \&  \&  \&  \&  \&  \&  \& |[rthrow]| \&  \&  \&  \&  \&  \&  \&  \&  \&  \&  \&  \&  \&  \&  \&  \& |[rthrow]| \&  \&  \&  \&  \&  \&  \&  \&  \&  \&  \&  \\
\cdots \&  \&  \&  \&  \&  \&  \&  \&  \&  \&  \&  \&  \&  \&  \&  \&  \&  \&  \&  \&  \&  \&  \&  \&  \&  \&  \&  \&  \&  \&  \&  \&  \&  \&  \&  \&  \&  \& \cdots \\
 \& |[gthrow]| 1 \&  \& |[gthrow]| 1 \&  \& |[gthrow]| 1 \&  \& |[gthrow]| 1 \&  \& |[gthrow]| 1 \&  \& |[gthrow]| 1 \&  \& |[gthrow]| 1 \&  \& |[gthrow]| 1 \&  \& |[gthrow]| 1 \&  \& |[gthrow]| 1 \&  \& |[gthrow]| 1 \&  \& |[gthrow]| 1 \&  \& |[gthrow]| 1 \&  \& |[gthrow]| 1 \&  \& |[gthrow]| 1 \&  \& |[gthrow]| 1 \&  \& |[gthrow]| 1 \&  \& |[gthrow]| 1 \&  \& |[gthrow]| 1 \&  \\
\cdots \&  \& 5 \&  \& 1 \&  \& |[bthrow]| 1 \&  \& 3 \&  \& 2 \&  \& 1 \&  \& 5 \&  \& 1 \&  \& 5 \&  \& 1 \&  \& |[bthrow]| 1 \&  \& 3 \&  \& 2 \&  \& 1 \&  \& 5 \&  \& 1 \&  \& 5 \&  \& 1 \&  \& \cdots \\
 \& 3 \&  \& 2 \&  \&  \&  \&  \&  \& 5 \&  \& 1 \&  \& 3 \&  \& 2 \&  \& 3 \&  \& 2 \&  \&  \&  \&  \&  \& 5 \&  \& 1 \&  \& 3 \&  \& 2 \&  \& 3 \&  \& 2 \&  \&  \&  \\
\cdots \&  \& |[bthrow]| 1 \&  \&  \&  \& -2 \&  \&  \&  \& 2 \&  \& 1 \&  \& |[bthrow]| 1 \&  \& 1 \&  \& |[bthrow]| 1 \&  \&  \&  \& -2 \&  \&  \&  \& 2 \&  \& 1 \&  \& |[bthrow]| 1 \&  \& 1 \&  \& |[bthrow]| 1 \&  \&  \&  \& \cdots \\
 \&  \&  \&  \&  \& -1 \&  \& -3 \&  \&  \&  \& |[bthrow]| 1 \&  \&  \&  \&  \&  \&  \&  \&  \&  \& -1 \&  \& -3 \&  \&  \&  \& |[bthrow]| 1 \&  \&  \&  \&  \&  \&  \&  \&  \&  \& -1 \&  \\
 \&  \&  \&  \&  \&  \& |[bbthrow]| -1 \&  \& |[bbthrow]| -1 \&  \&  \&  \&  \&  \& |[bbthrow]| -1 \&  \&  \&  \&  \&  \&  \&  \& |[bbthrow]| -1 \&  \& |[bbthrow]| -1 \&  \&  \&  \&  \&  \& |[bbthrow]| -1 \&  \&  \&  \&  \&  \&  \&  \& \cdots \\
 \&  \&  \& |[bthrow]| 1 \&  \&  \&  \&  \&  \&  \&  \&  \&  \&  \&  \&  \&  \&  \&  \& |[bthrow]| 1 \&  \&  \&  \&  \&  \&  \&  \&  \&  \&  \&  \&  \&  \&  \&  \& |[bthrow]| 1 \&  \&  \&  \\
 \&  \&  \&  \&  \&  \&  \&  \&  \&  \&  \&  \&  \&  \&  \&  \&  \&  \&  \&  \&  \&  \&  \&  \&  \&  \&  \&  \&  \&  \&  \&  \&  \&  \&  \&  \&  \&  \& \\
    };

    \draw[traj,opacity=.25] (M-9-1) to (M-8-2) to (M-13-7) to (M-8-12) to (M-11-15) to (M-8-18) to (M-13-23) to (M-8-28) to (M-11-31) to (M-8-34) to (M-13-39);
    \draw[traj,opacity=.25] (M-11-1) to (M-8-4) to (M-13-9) to (M-8-14) to (M-14-20) to (M-8-26) to (M-13-31) to (M-8-36) to (M-11-39);
    \draw[traj,opacity=.25] (M-9-1) to (M-11-3) to (M-8-6) to (M-9-7) to (M-8-8) to (M-12-12) to (M-8-16) to (M-11-19) to (M-8-22) to (M-9-23) to (M-8-24) to (M-12-28) to (M-8-32) to (M-11-35) to (M-8-38) to (M-9-39);
    \draw[traj,opacity=.25] (M-11-1) to (M-14-4) to (M-8-10) to (M-13-15) to (M-8-20) to (M-13-25) to (M-8-30) to (M-14-36) to (M-11-39);
    
    \draw[traj,opacity=.25] (M-7-1) to (M-8-2) to (M-3-7) to (M-8-12) to (M-5-15) to (M-8-18) to (M-3-23) to (M-8-28) to (M-5-31) to (M-8-34) to (M-3-39);
    \draw[traj,opacity=.25] (M-7-1) to (M-4-4) to (M-8-8) to (M-1-15) to (M-8-22) to (M-3-27) to (M-8-32) to (M-4-36) to (M-7-39);
    \draw[traj,opacity=.25] (M-5-1) to (M-8-4) to (M-5-7) to (M-8-10) to (M-6-12) to (M-8-14) to (M-5-17) to (M-8-20) to (M-5-23) to (M-8-26) to (M-6-28) to (M-8-30) to (M-5-33) to (M-8-36) to (M-5-39);
    \draw[traj,opacity=.25] (M-3-1) to (M-8-6) to (M-3-11) to (M-8-16) to (M-4-20) to (M-8-24) to (M-1-31) to (M-8-38) to (M-7-39);

\end{tikzpicture}\]
For each $a,b$ with $b\leq a <b-n$, consider $\C_{[a,b],[a,b]}$ (the diamond centered on the green circles indexed by the interval $[a,b]$) and delete any diagonal containing a red circle. The resulting matrix is $\C_{[a,b] \smallsetminus I , [a,b] \smallsetminus \pi^\dagger(I)}$ and the frieze condition is that all matrices of this form have determinant 1.
    
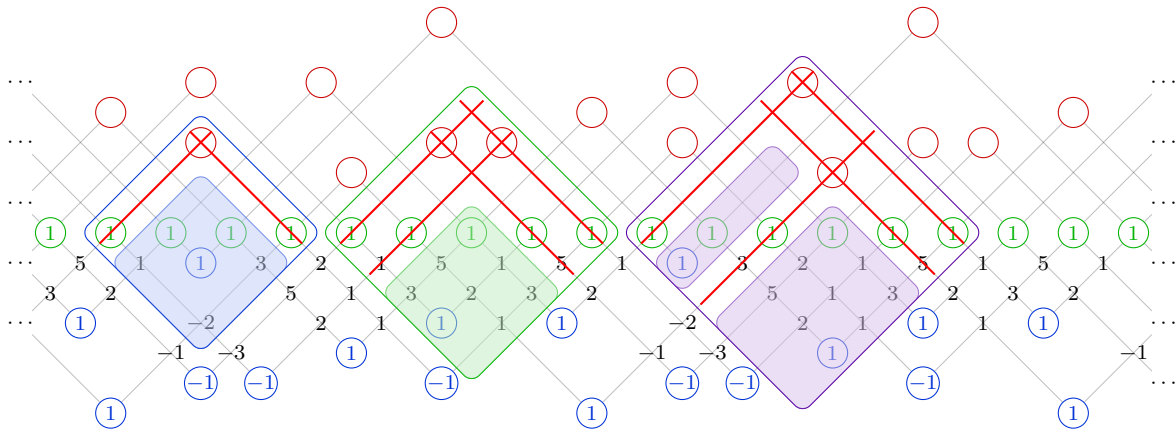
\[\begin{tikzpicture}[baseline=(current bounding box.center),
    ampersand replacement=\&,
    ]
    \matrix[matrix of math nodes,
        matrix anchor = M-1-20.center,
        nodes in empty cells,
        inner sep=0pt,
        gthrow/.style={dark green,draw,circle,inner sep=0mm,minimum size=4mm},
        rthrow/.style={dark red,draw,circle,inner sep=0mm,minimum size=4mm},
        bthrow/.style={dark blue,draw,circle,inner sep=0mm,minimum size=4mm},
        bbthrow/.style={dark blue,draw,circle,inner sep=0mm,minimum size=4mm},
        pthrow/.style={dark purple,draw,circle,inner sep=0mm,minimum size=4mm},
        nodes={anchor=center,node font=\scriptsize,minimum size=3mm},
        column sep={0.4cm,between origins},
        row sep={0.4cm,between origins},
    ] (M) at (0,0) {
 \&  \&  \&  \&  \&  \&  \&  \&  \&  \&  \&  \&  \&  \& |[rthrow]| \&  \&  \&  \&  \&  \&  \&  \&  \&  \&  \&  \&  \&  \&  \&  \& |[rthrow]| \&  \&  \&  \&  \&  \&  \&  \&  \\
 \&  \&  \&  \&  \&  \&  \&  \&  \&  \&  \&  \&  \&  \&  \&  \&  \&  \&  \&  \&  \&  \&  \&  \&  \&  \&  \&  \&  \&  \&  \&  \&  \&  \&  \&  \&  \&  \&  \\
\cdots \&  \&  \&  \&  \&  \& |[rthrow]| \&  \&  \&  \& |[rthrow]| \&  \&  \&  \&  \&  \&  \&  \&  \&  \&  \&  \& |[rthrow]| \&  \&  \&  \& |[rthrow]| \&  \&  \&  \&  \&  \&  \&  \&  \&  \&  \&  \& \cdots \\
 \&  \&  \& |[rthrow]| \&  \&  \&  \&  \&  \&  \&  \&  \&  \&  \&  \&  \&  \&  \&  \& |[rthrow]| \&  \&  \&  \&  \&  \&  \&  \&  \&  \&  \&  \&  \&  \&  \&  \& |[rthrow]|  \&  \&  \&  \\
\cdots \&  \&  \&  \&  \&  \& |[rthrow]| \&  \&  \&  \&  \&  \&  \&  \& |[rthrow]| \&  \& |[rthrow]| \&  \&  \&  \&  \&  \& |[rthrow]| \&  \&  \&  \&  \&  \&  \&  \& |[rthrow]| \&  \& |[rthrow]| \&  \&  \&  \&  \&  \& \cdots \\
 \&  \&  \&  \&  \&  \&  \&  \&  \&  \&  \& |[rthrow]| \&  \&  \&  \&  \&  \&  \&  \&  \&  \&  \&  \&  \&  \&  \&  \& |[rthrow]| \&  \&  \&  \&  \&  \&  \&  \&  \&  \&  \&  \\
\cdots \&  \&  \&  \&  \&  \&  \&  \&  \&  \&  \&  \&  \&  \&  \&  \&  \&  \&  \&  \&  \&  \&  \&  \&  \&  \&  \&  \&  \&  \&  \&  \&  \&  \&  \&  \&  \&  \& \cdots \\
 \& |[gthrow]| 1 \&  \& |[gthrow]| 1 \&  \& |[gthrow]| 1 \&  \& |[gthrow]| 1 \&  \& |[gthrow]| 1 \&  \& |[gthrow]| 1 \&  \& |[gthrow]| 1 \&  \& |[gthrow]| 1 \&  \& |[gthrow]| 1 \&  \& |[gthrow]| 1 \&  \& |[gthrow]| 1 \&  \& |[gthrow]| 1 \&  \& |[gthrow]| 1 \&  \& |[gthrow]| 1 \&  \& |[gthrow]| 1 \&  \& |[gthrow]| 1 \&  \& |[gthrow]| 1 \&  \& |[gthrow]| 1 \&  \& |[gthrow]| 1 \&  \\
\cdots \&  \& 5 \&  \& 1 \&  \& |[bthrow]| 1 \&  \& 3 \&  \& 2 \&  \& 1 \&  \& 5 \&  \& 1 \&  \& 5 \&  \& 1 \&  \& |[bthrow]| 1 \&  \& 3 \&  \& 2 \&  \& 1 \&  \& 5 \&  \& 1 \&  \& 5 \&  \& 1 \&  \& \cdots \\
 \& 3 \&  \& 2 \&  \&  \&  \&  \&  \& 5 \&  \& 1 \&  \& 3 \&  \& 2 \&  \& 3 \&  \& 2 \&  \&  \&  \&  \&  \& 5 \&  \& 1 \&  \& 3 \&  \& 2 \&  \& 3 \&  \& 2 \&  \&  \&  \\
\cdots \&  \& |[bthrow]| 1 \&  \&  \&  \& -2 \&  \&  \&  \& 2 \&  \& 1 \&  \& |[bthrow]| 1 \&  \& 1 \&  \& |[bthrow]| 1 \&  \&  \&  \& -2 \&  \&  \&  \& 2 \&  \& 1 \&  \& |[bthrow]| 1 \&  \& 1 \&  \& |[bthrow]| 1 \&  \&  \&  \& \cdots \\
 \&  \&  \&  \&  \& -1 \&  \& -3 \&  \&  \&  \& |[bthrow]| 1 \&  \&  \&  \&  \&  \&  \&  \&  \&  \& -1 \&  \& -3 \&  \&  \&  \& |[bthrow]| 1 \&  \&  \&  \&  \&  \&  \&  \&  \&  \& -1 \&  \\
 \&  \&  \&  \&  \&  \& |[bbthrow]| -1 \&  \& |[bbthrow]| -1 \&  \&  \&  \&  \&  \& |[bbthrow]| -1 \&  \&  \&  \&  \&  \&  \&  \& |[bbthrow]| -1 \&  \& |[bbthrow]| -1 \&  \&  \&  \&  \&  \& |[bbthrow]| -1 \&  \&  \&  \&  \&  \&  \&  \& \cdots \\
 \&  \&  \& |[bthrow]| 1 \&  \&  \&  \&  \&  \&  \&  \&  \&  \&  \&  \&  \&  \&  \&  \& |[bthrow]| 1 \&  \&  \&  \&  \&  \&  \&  \&  \&  \&  \&  \&  \&  \&  \&  \& |[bthrow]| 1 \&  \&  \&  \\
 \&  \&  \&  \&  \&  \&  \&  \&  \&  \&  \&  \&  \&  \&  \&  \&  \&  \&  \&  \&  \&  \&  \&  \&  \&  \&  \&  \&  \&  \&  \&  \&  \&  \&  \&  \&  \&  \& \\
    };

    \draw[traj,opacity=.25] (M-9-1) to (M-8-2) to (M-13-7) to (M-8-12) to (M-11-15) to (M-8-18) to (M-13-23) to (M-8-28) to (M-11-31) to (M-8-34) to (M-13-39);
    \draw[traj,opacity=.25] (M-11-1) to (M-8-4) to (M-13-9) to (M-8-14) to (M-14-20) to (M-8-26) to (M-13-31) to (M-8-36) to (M-11-39);
    \draw[traj,opacity=.25] (M-9-1) to (M-11-3) to (M-8-6) to (M-9-7) to (M-8-8) to (M-12-12) to (M-8-16) to (M-11-19) to (M-8-22) to (M-9-23) to (M-8-24) to (M-12-28) to (M-8-32) to (M-11-35) to (M-8-38) to (M-9-39);
    \draw[traj,opacity=.25] (M-11-1) to (M-14-4) to (M-8-10) to (M-13-15) to (M-8-20) to (M-13-25) to (M-8-30) to (M-14-36) to (M-11-39);
    
    \draw[traj,opacity=.25] (M-7-1) to (M-8-2) to (M-3-7) to (M-8-12) to (M-5-15) to (M-8-18) to (M-3-23) to (M-8-28) to (M-5-31) to (M-8-34) to (M-3-39);
    \draw[traj,opacity=.25] (M-7-1) to (M-4-4) to (M-8-8) to (M-1-15) to (M-8-22) to (M-3-27) to (M-8-32) to (M-4-36) to (M-7-39);
    \draw[traj,opacity=.25] (M-5-1) to (M-8-4) to (M-5-7) to (M-8-10) to (M-6-12) to (M-8-14) to (M-5-17) to (M-8-20) to (M-5-23) to (M-8-26) to (M-6-28) to (M-8-30) to (M-5-33) to (M-8-36) to (M-5-39);
    \draw[traj,opacity=.25] (M-3-1) to (M-8-6) to (M-3-11) to (M-8-16) to (M-4-20) to (M-8-24) to (M-1-31) to (M-8-38) to (M-7-39);

    \draw[dark blue, rounded corners] (M-8-3.center) -- (M-4-7.center) -- (M-8-11.center) -- (M-12-7.center) -- cycle;
    \draw[dark blue, fill=dark blue!25,opacity=.5,rounded corners] (M-9-4.center) -- (M-6-7.center) -- (M-9-10.center) -- (M-12-7.center) -- cycle;
    \draw[red,thick] (M-5-7.north west) to (M-8-10.south east);
    \draw[red,thick] (M-5-7.north east) to (M-8-4.south west);

    \draw[dark green,rounded corners] (M-8-11.center) -- (M-3-16.center) -- (M-8-21.center) -- (M-13-16.center) -- cycle;
    \draw[dark green, fill=dark green!25,opacity=.50,rounded corners] (M-10-13.center) -- (M-7-16.center) -- (M-10-19.center) -- (M-13-16.center) -- cycle;
    \draw[red,thick] (M-4-16.north west) to (M-8-20.south east);
    \draw[red,thick] (M-5-17.north east) to (M-9-13.south west);
    \draw[red,thick] (M-5-15.north west) to (M-9-19.south east);
    \draw[red,thick] (M-4-16.north east) to (M-8-12.south west);
    
    \draw[dark purple, rounded corners] (M-8-21.center) -- (M-2-27.center) -- (M-8-33.center) -- (M-14-27.center) -- cycle;
    \draw[dark purple, fill=dark purple!25,opacity=.5,rounded corners] (M-9-22.center) -- (M-5-26.center) -- (M-6-27.center) -- (M-10-23.center) -- cycle;
    \draw[dark purple, fill=dark purple!25,opacity=.5,rounded corners] (M-11-24.center) -- (M-7-28.center) -- (M-10-31.center) -- (M-14-27.center) -- cycle;
    \draw[red,thick] (M-3-27.north east) to (M-8-22.south west);
    \draw[red,thick] (M-3-27.north west) to (M-8-32.south east);
    \draw[red,thick] (M-4-26.north west) to (M-9-31.south east);
    \draw[red,thick] (M-5-29.north east) to (M-10-24.south west);

\end{tikzpicture}\]
Depicted above are the submatrices corresponding to the intervals $\textcolor{dark blue}{[-2,1]}$, $\textcolor{dark green}{[2,6]}$, and $\textcolor{dark purple}{[7,12]}$; one may check that all three have determinant 1.

\begin{rem}
This construction makes it easy to check whether $\pi^\dagger(a)\leq b$ and $a + n \leq \pi(b)$. The first inequality (respectively, second inequality) is equivalent to the presence of a red circle along the upper left diagonal (resp.\ upper right diagonal) of the diamond. Therefore, these inequalities are equivalent to requiring that the upper left and upper right diagonals are deleted.\footnote{Read as a matrix, this is equivalent to requiring that the top row and the right column of $\C_{[a,b],[a,b]}$ are deleted.} However, as observed in Remark \ref{rem: redundant}, ignoring these inequalities does not change the frieze condition.
\end{rem}

\subsection{Visualizing the tameness condition}\label{section: visualtame}

As before, attach the reflected diagram of $\pi^\dagger$ to the top of the $\pi$-prefrieze. For each $a,b$ with $b\leq a <b-n$, consider $\C_{[a,b],[a,b]}$ (the diamond centered on the green circles indexed by the interval $[a,b]$) as before. 
Next, delete the upper left diagonal and the upper right diagonal (regardless of whether they contain a red circle). Then, delete any diagonal containing a red circle in the remaining matrix. The resulting matrix is $\C_{(a,b] \smallsetminus J , [a,b) \smallsetminus \pi^\dagger(J)}$, and the tameness condition is that this determinant $0$ whenever $\pi^\dagger(a)<b$ or $a+n<\pi(b)$.
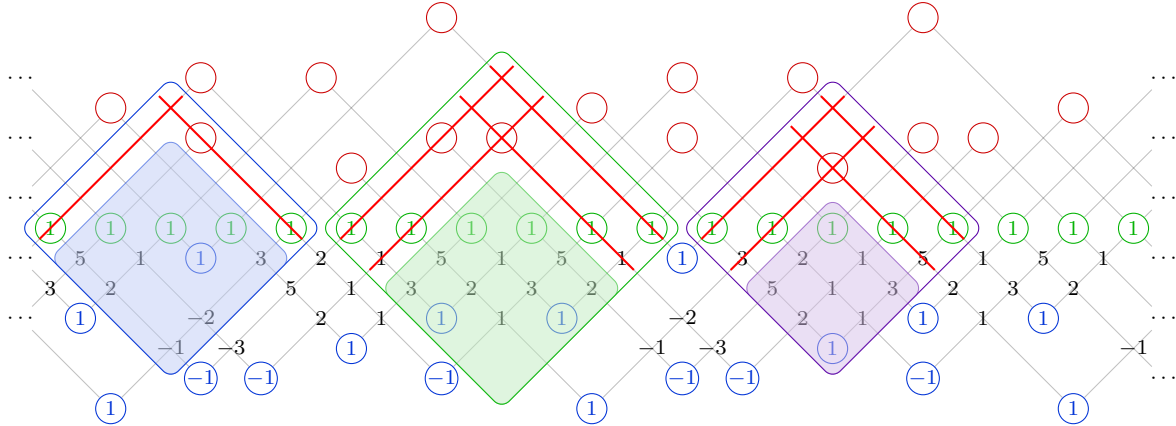
    \[\begin{tikzpicture}[baseline=(current bounding box.center),
    ampersand replacement=\&,
    ]
    \matrix[matrix of math nodes,
        matrix anchor = M-1-20.center,
        nodes in empty cells,
        inner sep=0pt,
        gthrow/.style={dark green,draw,circle,inner sep=0mm,minimum size=4mm},
        rthrow/.style={dark red,draw,circle,inner sep=0mm,minimum size=4mm},
        bthrow/.style={dark blue,draw,circle,inner sep=0mm,minimum size=4mm},
        bbthrow/.style={dark blue,draw,circle,inner sep=0mm,minimum size=4mm},
        pthrow/.style={dark purple,draw,circle,inner sep=0mm,minimum size=4mm},
        nodes={anchor=center,node font=\scriptsize,minimum size=3mm},
        column sep={0.4cm,between origins},
        row sep={0.4cm,between origins},
    ] (M) at (0,0) {
 \&  \&  \&  \&  \&  \&  \&  \&  \&  \&  \&  \&  \&  \& |[rthrow]| \&  \&  \&  \&  \&  \&  \&  \&  \&  \&  \&  \&  \&  \&  \&  \& |[rthrow]| \&  \&  \&  \&  \&  \&  \&  \&  \\
 \&  \&  \&  \&  \&  \&  \&  \&  \&  \&  \&  \&  \&  \&  \&  \&  \&  \&  \&  \&  \&  \&  \&  \&  \&  \&  \&  \&  \&  \&  \&  \&  \&  \&  \&  \&  \&  \&  \\
\cdots \&  \&  \&  \&  \&  \& |[rthrow]| \&  \&  \&  \& |[rthrow]| \&  \&  \&  \&  \&  \&  \&  \&  \&  \&  \&  \& |[rthrow]| \&  \&  \&  \& |[rthrow]| \&  \&  \&  \&  \&  \&  \&  \&  \&  \&  \&  \& \cdots \\
 \&  \&  \& |[rthrow]| \&  \&  \&  \&  \&  \&  \&  \&  \&  \&  \&  \&  \&  \&  \&  \& |[rthrow]| \&  \&  \&  \&  \&  \&  \&  \&  \&  \&  \&  \&  \&  \&  \&  \& |[rthrow]|  \&  \&  \&  \\
\cdots \&  \&  \&  \&  \&  \& |[rthrow]| \&  \&  \&  \&  \&  \&  \&  \& |[rthrow]| \&  \& |[rthrow]| \&  \&  \&  \&  \&  \& |[rthrow]| \&  \&  \&  \&  \&  \&  \&  \& |[rthrow]| \&  \& |[rthrow]| \&  \&  \&  \&  \&  \& \cdots \\
 \&  \&  \&  \&  \&  \&  \&  \&  \&  \&  \& |[rthrow]| \&  \&  \&  \&  \&  \&  \&  \&  \&  \&  \&  \&  \&  \&  \&  \& |[rthrow]| \&  \&  \&  \&  \&  \&  \&  \&  \&  \&  \&  \\
\cdots \&  \&  \&  \&  \&  \&  \&  \&  \&  \&  \&  \&  \&  \&  \&  \&  \&  \&  \&  \&  \&  \&  \&  \&  \&  \&  \&  \&  \&  \&  \&  \&  \&  \&  \&  \&  \&  \& \cdots \\
 \& |[gthrow]| 1 \&  \& |[gthrow]| 1 \&  \& |[gthrow]| 1 \&  \& |[gthrow]| 1 \&  \& |[gthrow]| 1 \&  \& |[gthrow]| 1 \&  \& |[gthrow]| 1 \&  \& |[gthrow]| 1 \&  \& |[gthrow]| 1 \&  \& |[gthrow]| 1 \&  \& |[gthrow]| 1 \&  \& |[gthrow]| 1 \&  \& |[gthrow]| 1 \&  \& |[gthrow]| 1 \&  \& |[gthrow]| 1 \&  \& |[gthrow]| 1 \&  \& |[gthrow]| 1 \&  \& |[gthrow]| 1 \&  \& |[gthrow]| 1 \&  \\
\cdots \&  \& 5 \&  \& 1 \&  \& |[bthrow]| 1 \&  \& 3 \&  \& 2 \&  \& 1 \&  \& 5 \&  \& 1 \&  \& 5 \&  \& 1 \&  \& |[bthrow]| 1 \&  \& 3 \&  \& 2 \&  \& 1 \&  \& 5 \&  \& 1 \&  \& 5 \&  \& 1 \&  \& \cdots \\
 \& 3 \&  \& 2 \&  \&  \&  \&  \&  \& 5 \&  \& 1 \&  \& 3 \&  \& 2 \&  \& 3 \&  \& 2 \&  \&  \&  \&  \&  \& 5 \&  \& 1 \&  \& 3 \&  \& 2 \&  \& 3 \&  \& 2 \&  \&  \&  \\
\cdots \&  \& |[bthrow]| 1 \&  \&  \&  \& -2 \&  \&  \&  \& 2 \&  \& 1 \&  \& |[bthrow]| 1 \&  \& 1 \&  \& |[bthrow]| 1 \&  \&  \&  \& -2 \&  \&  \&  \& 2 \&  \& 1 \&  \& |[bthrow]| 1 \&  \& 1 \&  \& |[bthrow]| 1 \&  \&  \&  \& \cdots \\
 \&  \&  \&  \&  \& -1 \&  \& -3 \&  \&  \&  \& |[bthrow]| 1 \&  \&  \&  \&  \&  \&  \&  \&  \&  \& -1 \&  \& -3 \&  \&  \&  \& |[bthrow]| 1 \&  \&  \&  \&  \&  \&  \&  \&  \&  \& -1 \&  \\
 \&  \&  \&  \&  \&  \& |[bbthrow]| -1 \&  \& |[bbthrow]| -1 \&  \&  \&  \&  \&  \& |[bbthrow]| -1 \&  \&  \&  \&  \&  \&  \&  \& |[bbthrow]| -1 \&  \& |[bbthrow]| -1 \&  \&  \&  \&  \&  \& |[bbthrow]| -1 \&  \&  \&  \&  \&  \&  \&  \& \cdots \\
 \&  \&  \& |[bthrow]| 1 \&  \&  \&  \&  \&  \&  \&  \&  \&  \&  \&  \&  \&  \&  \&  \& |[bthrow]| 1 \&  \&  \&  \&  \&  \&  \&  \&  \&  \&  \&  \&  \&  \&  \&  \& |[bthrow]| 1 \&  \&  \&  \\
 \&  \&  \&  \&  \&  \&  \&  \&  \&  \&  \&  \&  \&  \&  \&  \&  \&  \&  \&  \&  \&  \&  \&  \&  \&  \&  \&  \&  \&  \&  \&  \&  \&  \&  \&  \&  \&  \& \\
    };

    \draw[traj,opacity=.25] (M-9-1) to (M-8-2) to (M-13-7) to (M-8-12) to (M-11-15) to (M-8-18) to (M-13-23) to (M-8-28) to (M-11-31) to (M-8-34) to (M-13-39);
    \draw[traj,opacity=.25] (M-11-1) to (M-8-4) to (M-13-9) to (M-8-14) to (M-14-20) to (M-8-26) to (M-13-31) to (M-8-36) to (M-11-39);
    \draw[traj,opacity=.25] (M-9-1) to (M-11-3) to (M-8-6) to (M-9-7) to (M-8-8) to (M-12-12) to (M-8-16) to (M-11-19) to (M-8-22) to (M-9-23) to (M-8-24) to (M-12-28) to (M-8-32) to (M-11-35) to (M-8-38) to (M-9-39);
    \draw[traj,opacity=.25] (M-11-1) to (M-14-4) to (M-8-10) to (M-13-15) to (M-8-20) to (M-13-25) to (M-8-30) to (M-14-36) to (M-11-39);
    
    \draw[traj,opacity=.25] (M-7-1) to (M-8-2) to (M-3-7) to (M-8-12) to (M-5-15) to (M-8-18) to (M-3-23) to (M-8-28) to (M-5-31) to (M-8-34) to (M-3-39);
    \draw[traj,opacity=.25] (M-7-1) to (M-4-4) to (M-8-8) to (M-1-15) to (M-8-22) to (M-3-27) to (M-8-32) to (M-4-36) to (M-7-39);
    \draw[traj,opacity=.25] (M-5-1) to (M-8-4) to (M-5-7) to (M-8-10) to (M-6-12) to (M-8-14) to (M-5-17) to (M-8-20) to (M-5-23) to (M-8-26) to (M-6-28) to (M-8-30) to (M-5-33) to (M-8-36) to (M-5-39);
    \draw[traj,opacity=.25] (M-3-1) to (M-8-6) to (M-3-11) to (M-8-16) to (M-4-20) to (M-8-24) to (M-1-31) to (M-8-38) to (M-7-39);

    \draw[dark blue, rounded corners] (M-8-1.center) -- (M-3-6.center) -- (M-8-11.center) -- (M-13-6.center) -- cycle;
    \draw[dark blue, fill=dark blue!25,opacity=.5,rounded corners] (M-9-2.center) -- (M-5-6.center) -- (M-9-10.center) -- (M-13-6.center) -- cycle;
    \draw[red,thick] (M-4-6.north west) to (M-8-10.south east);
    \draw[red,thick] (M-4-6.north east) to (M-8-2.south west);
    
    \draw[dark green, rounded corners] (M-8-11.center) -- (M-2-17.center) -- (M-8-23.center) -- (M-14-17.center) -- cycle;
    \draw[dark green, fill=dark green!25,opacity=.5,rounded corners] (M-10-13.center) -- (M-6-17.center) -- (M-10-21.center) -- (M-14-17.center) -- cycle;
    \draw[red,thick] (M-3-17.north west) to (M-8-22.south east);
    \draw[red,thick] (M-3-17.north east) to (M-8-12.south west);
    \draw[red,thick] (M-4-16.north west) to (M-9-21.south east);
    \draw[red,thick] (M-4-18.north east) to (M-9-13.south west);
    
    \draw[dark purple, rounded corners] (M-8-23.center) -- (M-3-28.center) -- (M-8-33.center) -- (M-13-28.center) -- cycle;
    \draw[dark purple, fill=dark purple!25,opacity=.5,rounded corners] (M-10-25.center) -- (M-7-28.center) -- (M-10-31.center) -- (M-13-28.center) -- cycle;
    \draw[red,thick] (M-4-28.north west) to (M-8-32.south east);
    \draw[red,thick] (M-4-28.north east) to (M-8-24.south west);
    \draw[red,thick] (M-5-27.north west) to (M-9-31.south east);
    \draw[red,thick] (M-5-29.north east) to (M-9-25.south west);

\end{tikzpicture}\]
Depicted above are the submatrices corresponding to the intervals $\textcolor{dark blue}{[-3,1]}$, $\textcolor{dark green}{[2,7]}$, and $\textcolor{dark purple}{[8,12]}$; one may check that they have determinants $0$, $0$, and $1$, respectively.
The latter determinant does not contradict the tameness condition because it fails both conditions: $\pi^\dagger(8) \not< 12$ and $8+8 \not < \pi(12)$.

\begin{rem}\label{rem: tamenessidiagram}
A pair $(a,b)$ satisfies $\pi^\dagger(a)< b$ or $a + n < \pi(b)$ if the corresponding diamond contains a red circle which is in the upper left diagonal or the upper right diagonal, but not both. Note that the diamond associated to $(8,12)$ does not have a red circle in either top diagonal. 
\end{rem}


\subsection{Recovering $\mathrm{SL}(k)$-friezes}

While the definition of a juggler's frieze is less elegant than that of a $\mathrm{SL}(k)$-frieze, the latter is a special case of the former.
A juggling function is \textbf{uniform} if there is some $h$ for which $\pi(a) = a+h$ for all $a$; note that this $h$ is the number of balls of $\pi$. 

\begin{warn}
With uniform juggling functions, it is particularly important to remember the
choice of period is part of the data of a juggling function, since
$\pi(a+m)=\pi(a)+m$ for all $m\in \mathbb{Z}$!
\end{warn}

\begin{prop}
Let $\pi$ be a uniform juggling function with $h$-many balls and period $n$.
\begin{enumerate}
    \item A $\pi$-prefrieze is the same as a prefrieze of height $h$.
    \item A $\pi$-frieze is the same as an $\mathrm{SL}(n-h)$-frieze of height $h$.
\end{enumerate}
\end{prop}

\begin{proof}
%
Since $S_\pi(a,\pi(a)) = S_\pi(a,a+k)=\varnothing$, a $\pi$-prefrieze consists of $(h+1)$-many rows in a diamond grid, in which the top and bottom rows consist of $1$s; that is, a prefrieze of height $h$.

For any $a$, $\pi^\dagger(a) = \pi^{-1}(a)+ n = a+n-k$; that is, $\pi^\dagger$ is the uniform juggling function with $(n-h)$-many balls and period $n$. Then there is a $\pi$-frieze condition for each pair $a,b$ with
\begin{equation*}
a+n-k \leq b < a+n \leq b+k 
\end{equation*}
which can be rewritten as $a< b-n+k+1 < a+k$.
The corresponding identity in the $\pi$-frieze condition is then
\[ \det( \C_{[b-n+k+1,b] , [a,a+n-k-1]}) =1 \]
The $\pi$-frieze condition coincides with Condition (1) in Definition \ref{defn: frieze} for $k=n-h$.

As for tameness, observe that $\pi^\dagger(a) < b< a+n$ iff $b<a_n<\pi(b)$. Assuming either holds, the corresponding identity in the tameness condition is 
\[  \det( \C_{[b-n+k,b] , [a,a+n-k]}) = 0 \]
The tameness condition for $\pi$-friezes  coincides with Condition (2) in Definition \ref{defn: frieze} for $k=n-h$.
\end{proof}

\begin{ex}
Let $\pi$ be the uniform $8$-periodic juggling function with $5$ balls. The
$\mathrm{SL}(3)$-frieze of height $5$ from Example \ref{ex: intro2} is seen to be a $\pi$-prefrieze by drawing the diagram of $\pi$.
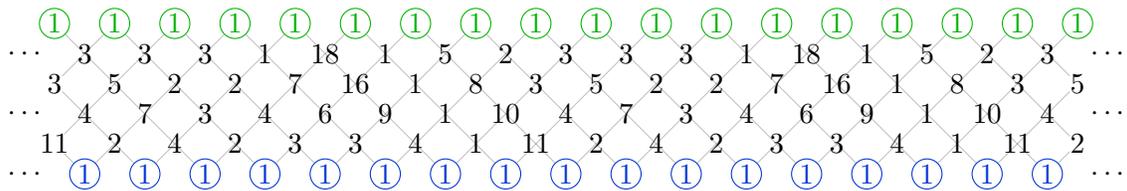
    \[
    \begin{tikzpicture}[baseline=(M-6-1.base),
    ampersand replacement=\&,
    ]
    \clip[use as bounding box] (-7.5,.3) rectangle (7.5,-2.3);
    \matrix[matrix of math nodes,
        matrix anchor = M-1-29.center,
        origin/.style={},
        gthrow/.style={dark green,draw,circle,inner sep=0mm,minimum size=4mm},
        rthrow/.style={dark red,draw,circle,inner sep=0mm,minimum size=4mm},
        bthrow/.style={dark blue,draw,circle,inner sep=0mm,minimum size=4mm},
        pivot/.style={draw,circle,inner sep=0.25mm,minimum size=2mm},       
        nodes in empty cells,
        inner sep=0pt,
        nodes={anchor=center},
        column sep={.4cm,between origins},
        row sep={.4cm,between origins},
    ] (M) at (0,0) {
    \& \& \& \& \& |[gthrow]| 1 \& \& |[gthrow]| 1 \& \& 1 \& \& |[gthrow]| 1 \& \& |[gthrow]| 1 \& \& |[gthrow]| 1 \& \& |[gthrow]| 1 \& \& |[gthrow]| 1 \& \& |[gthrow]| 1 \& \& |[gthrow]| 1 \& \& |[gthrow]| 1 \& \& |[gthrow]| 1 \& \& |[gthrow]| 1 \& \& |[gthrow]| 1 \& \& |[gthrow]| 1 \& \& |[gthrow]| 1 \& \& |[gthrow]| 1 \& \& |[gthrow]| 1 \& \& |[gthrow]| 1 \& \& |[gthrow]| 1 \& \& |[gthrow]| 1 \& \& 1 \& \& |[gthrow]| 1 \& \& |[gthrow]| 1 \& \& \\
    \& \& \& \& 18 \& \& 1 \& \& 5 \& \& \cdots \& \& 3 \& \& 3 \& \& 3 \& \& 1 \& \& 18 \& \& 1 \& \& 5 \& \& 2 \& \& 3 \& \& 3 \& \& 3 \& \& 1 \& \& 18 \& \& 1 \& \& 5 \& \& 2 \& \& 3 \& \& \cdots \& \& 3 \& \& 1 \& \& \\
    \& \& \& 7 \& \& 16 \& \& 1 \& \& 8 \& \& 3 \& \& 5 \& \& 2 \& \& 2 \& \& 7 \& \& 16 \& \& 1 \& \& 8 \& \& 3 \& \& 5 \& \& 2 \& \& 2 \& \& 7 \& \& 16 \& \& 1 \& \& 8 \& \& 3 \& \& 5 \& \& 2 \& \& 2 \& \& \\
    \& \& 4 \& \& 6 \& \& 9 \& \& 1 \& \& \cdots \& \& 4 \& \& 7 \& \& 3 \& \& 4 \& \& 6 \& \& 9 \& \& 1 \& \& 10 \& \& 4 \& \& 7 \& \& 3 \& \& 4 \& \& 6 \& \& 9 \& \& 1 \& \& 10 \& \& 4 \& \& \cdots \& \& 3 \& \& \\
    \& 2 \& \& 3 \& \& 3 \& \& 4 \& \& 1 \& \& 11 \& \& 2 \& \& 4 \& \& 2 \& \& 3 \& \& 3 \& \& 4 \& \& 1 \& \& 11 \& \& 2 \& \& 4 \& \& 2 \& \& 3 \& \& 3 \& \& 4 \& \& 1 \& \& 11 \& \& 2 \& \& 4 \& \& \\
    1 \& \& 1 \& \& 1 \& \& 1 \& \& 1 \& \& \cdots \& \& |[bthrow]| 1 \& \& |[bthrow]| 1 \& \& |[bthrow]| 1 \& \& |[bthrow]| 1 \& \& |[bthrow]| 1 \& \& |[bthrow]| 1 \& \& |[bthrow]| 1 \& \& |[bthrow]| 1 \& \& |[bthrow]| 1 \& \& |[bthrow]| 1 \& \& |[bthrow]| 1 \& \& |[bthrow]| 1 \& \& |[bthrow]| 1 \& \& |[bthrow]| 1 \& \& |[bthrow]| 1 \& \& |[bthrow]| 1 \& \& |[bthrow]| 1 \& \& \cdots \& \& \\
    };

    \draw[traj,opacity=.25] (M-1-12) to (M-6-17) to (M-1-22) to (M-6-27) to (M-1-32) to (M-6-37) to (M-1-42) to (M-5-46);
    \draw[traj,opacity=.25] (M-3-12) to (M-6-15) to (M-1-20) to (M-6-25) to (M-1-30) to (M-6-35) to (M-1-40) to (M-6-45) to (M-5-46);
    \draw[traj,opacity=.25] (M-5-12) to (M-6-13) to (M-1-18) to (M-6-23) to (M-1-28) to (M-6-33) to (M-1-38) to (M-6-43) to (M-3-46);
    \draw[traj,opacity=.25] (M-5-12) to (M-1-16) to (M-6-21) to (M-1-26) to (M-6-31) to (M-1-36) to (M-6-41) to (M-1-46);
    \draw[traj,opacity=.25] (M-3-12) to (M-1-14) to (M-6-19) to (M-1-24) to (M-6-29) to (M-1-34) to (M-6-39) to (M-1-44) to (M-3-46);
    \end{tikzpicture}
    \]
For this choice of $\pi$, the construction in Section \ref{section: visualfrieze} produces $3\times 3$ solid submatrices.
    
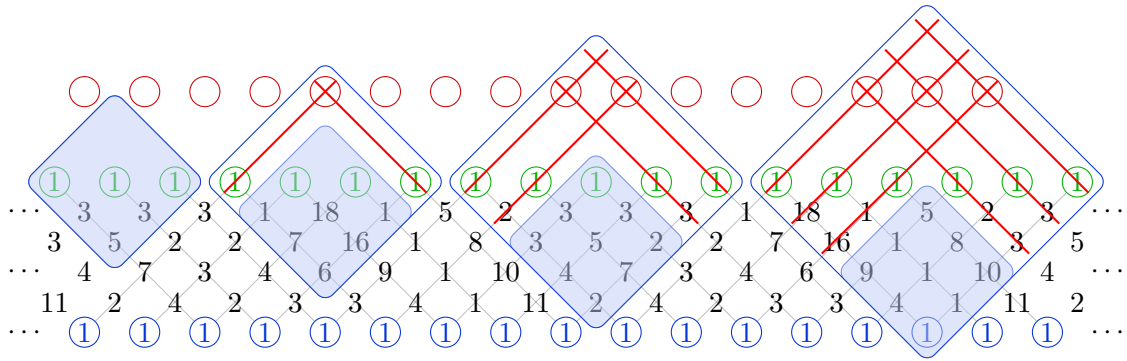
\[
    \begin{tikzpicture}[baseline=(M-6-1.base),
    ampersand replacement=\&,
    ]
    \clip[use as bounding box] (-7.5,0) rectangle (7.5,-4.85);
    \matrix[matrix of math nodes,
        matrix anchor = M-1-29.center,
        origin/.style={},
        gthrow/.style={dark green,draw,circle,inner sep=0mm,minimum size=4mm},
        rthrow/.style={dark red,draw,circle,inner sep=0mm,minimum size=4mm},
        bthrow/.style={dark blue,draw,circle,inner sep=0mm,minimum size=4mm},
        pivot/.style={draw,circle,inner sep=0.25mm,minimum size=2mm},       
        nodes in empty cells,
        inner sep=0pt,
        nodes={anchor=center,minimum size=3mm},
        column sep={.4cm,between origins},
        row sep={.4cm,between origins},
    ] (M) at (0,0) {
    \& \& \& \&  \& \&  \& \&  \& \&  \& \&  \& \&  \& \&  \& \&  \& \&  \& \&  \& \&  \& \&  \& \&  \& \&  \& \&  \& \&  \& \&  \& \&  \& \&  \& \&  \& \&  \& \&  \& \&  \& \&  \& \& \\
    \& \& \& \&  \& \&  \& \&  \& \&  \& \&  \& \&  \& \&  \& \&  \& \&  \& \&  \& \&  \& \&  \& \&  \& \&  \& \&  \& \&  \& \&  \& \&  \& \&  \& \&  \& \&  \& \&  \& \&  \& \&  \& \& \\
    \& \& \& \&  \& \&  \& \&  \& \&  \& \&  \& \&  \& \&  \& \&  \& \&  \& \&  \& \&  \& \&  \& \&  \& \&  \& \&  \& \&  \& \&  \& \&  \& \&  \& \&  \& \&  \& \&  \& \&  \& \&  \& \& \\
    \& \&  \& \&  \& \&  \& \&  \& \&  \& \& |[rthrow]|  \& \& |[rthrow]|  \& \& |[rthrow]|  \& \& |[rthrow]|  \& \& |[rthrow]|  \& \& |[rthrow]|  \& \& |[rthrow]|  \& \& |[rthrow]|  \& \& |[rthrow]|  \& \& |[rthrow]|  \& \& |[rthrow]|  \& \& |[rthrow]|  \& \& |[rthrow]|  \& \& |[rthrow]|  \& \& |[rthrow]|  \& \& |[rthrow]|  \& \&  \& \&  \& \&  \& \&  \& \& \\
    \& \& \& \&  \& \&  \& \&  \& \&  \& \&  \& \&  \& \&  \& \&  \& \&  \& \&  \& \&  \& \&  \& \&  \& \&  \& \&  \& \&  \& \&  \& \&  \& \&  \& \&  \& \&  \& \&  \& \&  \& \&  \& \& \\
    \& \& \& \&  \& \&  \& \&  \& \&  \& \&  \& \&  \& \&  \& \&  \& \&  \& \&  \& \&  \& \&  \& \&  \& \&  \& \&  \& \&  \& \&  \& \&  \& \&  \& \&  \& \&  \& \&  \& \&  \& \&  \& \& \\
    \&  \& \&  \& \& |[gthrow]| 1 \& \& |[gthrow]| 1 \& \& 1 \& \& |[gthrow]| 1 \& \& |[gthrow]| 1 \& \& |[gthrow]| 1 \& \& |[gthrow]| 1 \& \& |[gthrow]| 1 \& \& |[gthrow]| 1 \& \& |[gthrow]| 1 \& \& |[gthrow]| 1 \& \& |[gthrow]| 1 \& \& |[gthrow]| 1 \& \& |[gthrow]| 1 \& \& |[gthrow]| 1 \& \& |[gthrow]| 1 \& \& |[gthrow]| 1 \& \& |[gthrow]| 1 \& \& |[gthrow]| 1 \& \& |[gthrow]| 1 \& \& |[gthrow]| 1 \& \& 1 \& \& |[gthrow]| 1 \& \& |[gthrow]| 1 \& \& \\
    \& \& \& \& 18 \& \& 1 \& \& 5 \& \& \cdots \& \& 3 \& \& 3 \& \& 3 \& \& 1 \& \& 18 \& \& 1 \& \& 5 \& \& 2 \& \& 3 \& \& 3 \& \& 3 \& \& 1 \& \& 18 \& \& 1 \& \& 5 \& \& 2 \& \& 3 \& \& \cdots \& \& 3 \& \& 1 \& \& \\
    \& \& \& 7 \& \& 16 \& \& 1 \& \& 8 \& \& 3 \& \& 5 \& \& 2 \& \& 2 \& \& 7 \& \& 16 \& \& 1 \& \& 8 \& \& 3 \& \& 5 \& \& 2 \& \& 2 \& \& 7 \& \& 16 \& \& 1 \& \& 8 \& \& 3 \& \& 5 \& \& 2 \& \& 2 \& \& \\
    \& \& 4 \& \& 6 \& \& 9 \& \& 1 \& \& \cdots \& \& 4 \& \& 7 \& \& 3 \& \& 4 \& \& 6 \& \& 9 \& \& 1 \& \& 10 \& \& 4 \& \& 7 \& \& 3 \& \& 4 \& \& 6 \& \& 9 \& \& 1 \& \& 10 \& \& 4 \& \& \cdots \& \& 3 \& \& \\
    \& 2 \& \& 3 \& \& 3 \& \& 4 \& \& 1 \& \& 11 \& \& 2 \& \& 4 \& \& 2 \& \& 3 \& \& 3 \& \& 4 \& \& 1 \& \& 11 \& \& 2 \& \& 4 \& \& 2 \& \& 3 \& \& 3 \& \& 4 \& \& 1 \& \& 11 \& \& 2 \& \& 4 \& \& \\
    1 \& \& 1 \& \& 1 \& \& 1 \& \& 1 \& \& \cdots \& \& |[bthrow]| 1 \& \& |[bthrow]| 1 \& \& |[bthrow]| 1 \& \& |[bthrow]| 1 \& \& |[bthrow]| 1 \& \& |[bthrow]| 1 \& \& |[bthrow]| 1 \& \& |[bthrow]| 1 \& \& |[bthrow]| 1 \& \& |[bthrow]| 1 \& \& |[bthrow]| 1 \& \& |[bthrow]| 1 \& \& |[bthrow]| 1 \& \& |[bthrow]| 1 \& \& |[bthrow]| 1 \& \& |[bthrow]| 1 \& \& |[bthrow]| 1 \& \& \cdots \& \& \\
    \& \& \& \&  \& \&  \& \&  \& \&  \& \&  \& \&  \& \&  \& \&  \& \&  \& \&  \& \&  \& \&  \& \&  \& \&  \& \&  \& \&  \& \&  \& \&  \& \&  \& \&  \& \&  \& \&  \& \&  \& \&  \& \& \\
    };
    
    \foreach \a [evaluate=\a as \y using int(\a+5)] in {12,14,...,40}
        {
            \draw[traj,opacity=.25] (M-7-\a) to (M-12-\y);
        }
    \foreach \a [evaluate=\a as \y using int(\a+5)] in {13,15,...,41}
        {
            \draw[traj,opacity=.25] (M-12-\a) to (M-7-\y);
        }


    \draw[dark blue, rounded corners] (M-7-11.center) -- (M-4-14.center) -- (M-7-17.center) -- (M-10-14.center) -- cycle;
    \draw[dark blue, fill=dark blue!25,opacity=.5, rounded corners] (M-7-11.center) -- (M-4-14.center) -- (M-7-17.center) -- (M-10-14.center) -- cycle;

    \draw[dark blue, rounded corners] (M-7-17.center) -- (M-3-21.center) -- (M-7-25.center) -- (M-11-21.center) -- cycle;
    \draw[dark blue, fill=dark blue!25,opacity=.5,rounded corners] (M-8-18.center) -- (M-5-21.center) -- (M-8-24.center) -- (M-11-21.center) -- cycle;
    \draw[red,thick] (M-4-21.north west) to (M-7-24.south east);
    \draw[red,thick] (M-4-21.north east) to (M-7-18.south west);
    
    \draw[dark blue, rounded corners] (M-7-25.center) -- (M-2-30.center) -- (M-7-35.center) -- (M-12-30.center) -- cycle;
    \draw[dark blue, fill=dark blue!25,opacity=.5,rounded corners] (M-9-27.center) -- (M-6-30.center) -- (M-9-33.center) -- (M-12-30.center) -- cycle;
    \draw[red,thick] (M-4-29.north west) to (M-8-33.south east);
    \draw[red,thick] (M-3-30.north west) to (M-7-34.south east);
    \draw[red,thick] (M-3-30.north east) to (M-7-26.south west);
    \draw[red,thick] (M-4-31.north east) to (M-8-27.south west);

    \draw[dark blue, rounded corners] (M-7-35.center) -- (M-1-41.center) -- (M-7-47.center) -- (M-13-41.center) -- cycle;
    \draw[dark blue, fill=dark blue!25,opacity=.5,rounded corners] (M-10-38.center) -- (M-7-41.center) -- (M-10-44.center) -- (M-13-41.center) -- cycle;
    \draw[red,thick] (M-4-39.north west) to (M-9-44.south east);
    \draw[red,thick] (M-3-40.north west) to (M-8-45.south east);
    \draw[red,thick] (M-2-41.north west) to (M-7-46.south east);
    \draw[red,thick] (M-2-41.north east) to (M-7-36.south west);
    \draw[red,thick] (M-3-42.north east) to (M-8-37.south west);
    \draw[red,thick] (M-4-43.north east) to (M-9-38.south west);

    \end{tikzpicture}
    \]
Similarly, the construction in Section \ref{section: visualtame} produces $4 \times 4$ solid submatrices.
    
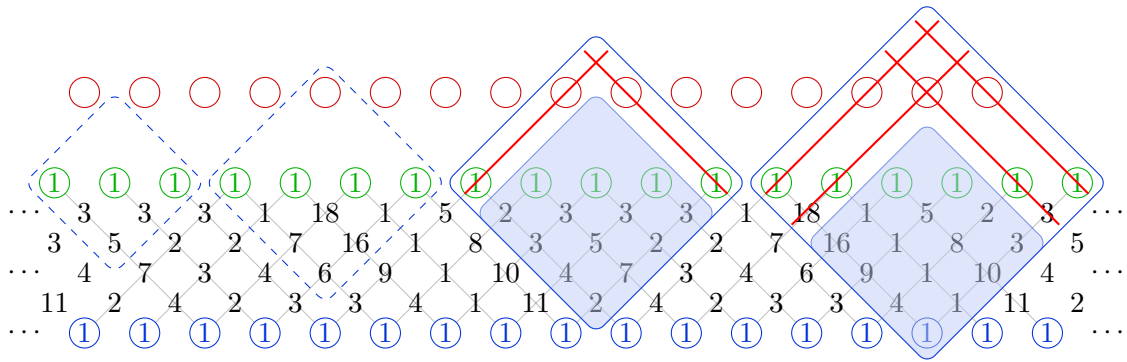
\[
    \begin{tikzpicture}[baseline=(M-6-1.base),
    ampersand replacement=\&,
    ]
    \clip[use as bounding box] (-7.5,0) rectangle (7.5,-4.85);
    \matrix[matrix of math nodes,
        matrix anchor = M-1-29.center,
        origin/.style={},
        gthrow/.style={dark green,draw,circle,inner sep=0mm,minimum size=4mm},
        rthrow/.style={dark red,draw,circle,inner sep=0mm,minimum size=4mm},
        bthrow/.style={dark blue,draw,circle,inner sep=0mm,minimum size=4mm},
        pivot/.style={draw,circle,inner sep=0.25mm,minimum size=2mm},       
        nodes in empty cells,
        inner sep=0pt,
        nodes={anchor=center,minimum size=3mm},
        column sep={.4cm,between origins},
        row sep={.4cm,between origins},
    ] (M) at (0,0) {
    \& \& \& \&  \& \&  \& \&  \& \&  \& \&  \& \&  \& \&  \& \&  \& \&  \& \&  \& \&  \& \&  \& \&  \& \&  \& \&  \& \&  \& \&  \& \&  \& \&  \& \&  \& \&  \& \&  \& \&  \& \&  \& \& \\
    \& \& \& \&  \& \&  \& \&  \& \&  \& \&  \& \&  \& \&  \& \&  \& \&  \& \&  \& \&  \& \&  \& \&  \& \&  \& \&  \& \&  \& \&  \& \&  \& \&  \& \&  \& \&  \& \&  \& \&  \& \&  \& \& \\
    \& \& \& \&  \& \&  \& \&  \& \&  \& \&  \& \&  \& \&  \& \&  \& \&  \& \&  \& \&  \& \&  \& \&  \& \&  \& \&  \& \&  \& \&  \& \&  \& \&  \& \&  \& \&  \& \&  \& \&  \& \&  \& \& \\
    \& \&  \& \&  \& \&  \& \&  \& \&  \& \& |[rthrow]|  \& \& |[rthrow]|  \& \& |[rthrow]|  \& \& |[rthrow]|  \& \& |[rthrow]|  \& \& |[rthrow]|  \& \& |[rthrow]|  \& \& |[rthrow]|  \& \& |[rthrow]|  \& \& |[rthrow]|  \& \& |[rthrow]|  \& \& |[rthrow]|  \& \& |[rthrow]|  \& \& |[rthrow]|  \& \& |[rthrow]|  \& \& |[rthrow]|  \& \&  \& \&  \& \&  \& \&  \& \& \\
    \& \& \& \&  \& \&  \& \&  \& \&  \& \&  \& \&  \& \&  \& \&  \& \&  \& \&  \& \&  \& \&  \& \&  \& \&  \& \&  \& \&  \& \&  \& \&  \& \&  \& \&  \& \&  \& \&  \& \&  \& \&  \& \& \\
    \& \& \& \&  \& \&  \& \&  \& \&  \& \&  \& \&  \& \&  \& \&  \& \&  \& \&  \& \&  \& \&  \& \&  \& \&  \& \&  \& \&  \& \&  \& \&  \& \&  \& \&  \& \&  \& \&  \& \&  \& \&  \& \& \\
    \&  \& \&  \& \& |[gthrow]| 1 \& \& |[gthrow]| 1 \& \& 1 \& \& |[gthrow]| 1 \& \& |[gthrow]| 1 \& \& |[gthrow]| 1 \& \& |[gthrow]| 1 \& \& |[gthrow]| 1 \& \& |[gthrow]| 1 \& \& |[gthrow]| 1 \& \& |[gthrow]| 1 \& \& |[gthrow]| 1 \& \& |[gthrow]| 1 \& \& |[gthrow]| 1 \& \& |[gthrow]| 1 \& \& |[gthrow]| 1 \& \& |[gthrow]| 1 \& \& |[gthrow]| 1 \& \& |[gthrow]| 1 \& \& |[gthrow]| 1 \& \& |[gthrow]| 1 \& \& 1 \& \& |[gthrow]| 1 \& \& |[gthrow]| 1 \& \& \\
    \& \& \& \& 18 \& \& 1 \& \& 5 \& \& \cdots \& \& 3 \& \& 3 \& \& 3 \& \& 1 \& \& 18 \& \& 1 \& \& 5 \& \& 2 \& \& 3 \& \& 3 \& \& 3 \& \& 1 \& \& 18 \& \& 1 \& \& 5 \& \& 2 \& \& 3 \& \& \cdots \& \& 3 \& \& 1 \& \& \\
    \& \& \& 7 \& \& 16 \& \& 1 \& \& 8 \& \& 3 \& \& 5 \& \& 2 \& \& 2 \& \& 7 \& \& 16 \& \& 1 \& \& 8 \& \& 3 \& \& 5 \& \& 2 \& \& 2 \& \& 7 \& \& 16 \& \& 1 \& \& 8 \& \& 3 \& \& 5 \& \& 2 \& \& 2 \& \& \\
    \& \& 4 \& \& 6 \& \& 9 \& \& 1 \& \& \cdots \& \& 4 \& \& 7 \& \& 3 \& \& 4 \& \& 6 \& \& 9 \& \& 1 \& \& 10 \& \& 4 \& \& 7 \& \& 3 \& \& 4 \& \& 6 \& \& 9 \& \& 1 \& \& 10 \& \& 4 \& \& \cdots \& \& 3 \& \& \\
    \& 2 \& \& 3 \& \& 3 \& \& 4 \& \& 1 \& \& 11 \& \& 2 \& \& 4 \& \& 2 \& \& 3 \& \& 3 \& \& 4 \& \& 1 \& \& 11 \& \& 2 \& \& 4 \& \& 2 \& \& 3 \& \& 3 \& \& 4 \& \& 1 \& \& 11 \& \& 2 \& \& 4 \& \& \\
    1 \& \& 1 \& \& 1 \& \& 1 \& \& 1 \& \& \cdots \& \& |[bthrow]| 1 \& \& |[bthrow]| 1 \& \& |[bthrow]| 1 \& \& |[bthrow]| 1 \& \& |[bthrow]| 1 \& \& |[bthrow]| 1 \& \& |[bthrow]| 1 \& \& |[bthrow]| 1 \& \& |[bthrow]| 1 \& \& |[bthrow]| 1 \& \& |[bthrow]| 1 \& \& |[bthrow]| 1 \& \& |[bthrow]| 1 \& \& |[bthrow]| 1 \& \& |[bthrow]| 1 \& \& |[bthrow]| 1 \& \& |[bthrow]| 1 \& \& \cdots \& \& \\
    \& \& \& \&  \& \&  \& \&  \& \&  \& \&  \& \&  \& \&  \& \&  \& \&  \& \&  \& \&  \& \&  \& \&  \& \&  \& \&  \& \&  \& \&  \& \&  \& \&  \& \&  \& \&  \& \&  \& \&  \& \&  \& \& \\
    };
    
    \foreach \a [evaluate=\a as \y using int(\a+5)] in {12,14,...,40}
        {
            \draw[traj,opacity=.25] (M-7-\a) to (M-12-\y);
        }
    \foreach \a [evaluate=\a as \y using int(\a+5)] in {13,15,...,41}
        {
            \draw[traj,opacity=.25] (M-12-\a) to (M-7-\y);
        }
        
    \draw[dashed, dark blue, rounded corners] (M-7-11.center) -- (M-4-14.center) -- (M-7-17.center) -- (M-10-14.center) -- cycle;

    \draw[dashed,dark blue, rounded corners] (M-7-17.center) -- (M-3-21.center) -- (M-7-25.center) -- (M-11-21.center) -- cycle;
    
    \draw[dark blue, rounded corners] (M-7-25.center) -- (M-2-30.center) -- (M-7-35.center) -- (M-12-30.center) -- cycle;
    \draw[dark blue, fill=dark blue!25,opacity=.5,rounded corners] (M-8-26.center) -- (M-4-30.center) -- (M-8-34.center) -- (M-12-30.center) -- cycle;
    \draw[red,thick] (M-3-30.north west) to (M-7-34.south east);
    \draw[red,thick] (M-3-30.north east) to (M-7-26.south west);

    \draw[dark blue, rounded corners] (M-7-35.center) -- (M-1-41.center) -- (M-7-47.center) -- (M-13-41.center) -- cycle;
    \draw[dark blue, fill=dark blue!25,opacity=.5,rounded corners] (M-9-37.center) -- (M-5-41.center) -- (M-9-45.center) -- (M-13-41.center) -- cycle;
    \draw[red,thick] (M-3-40.north west) to (M-8-45.south east);
    \draw[red,thick] (M-2-41.north west) to (M-7-46.south east);
    \draw[red,thick] (M-2-41.north east) to (M-7-36.south west);
    \draw[red,thick] (M-3-42.north east) to (M-8-37.south west);

    \end{tikzpicture}
    \]
    We see that this $\mathrm{SL}(5)$-frieze is also a $\pi$-frieze.
    Note that the dashed diamonds above do not satisfy $\pi^\dagger(a)<b$ or $a+n<\pi(b)$, so they do not contribute a minor to the tameness condition.
\end{ex}

\subsection{Juggler's friezes and duality}

Juggler's friezes may be characterized by an analogous duality condition to friezes; however, the definition of the dual $\C^\dagger$ must be modified slightly to account for the possibility of loops in $\pi$.
Given a $\pi$-prefrieze $\C$, define the \textbf{dual} $\C^\dagger$ as follows. 
\[ 
\C^\dagger_{a,b} 
:= \begin{cases}
\det(\C_{[a,b-1],[a+1,b]}) & \text{if $a \geq b >a-n$},\\
(-1)^{\text{(\# of balls in $\pi$)}-1} & \text{if $\pi(a)=a=b+n$},\\
0 & \text{otherwise}.
\end{cases}
\]
Note that the choice of $n$ which was explicit in the case of $\mathrm{SL}(k)$-friezes is now implicit in the juggling function $\pi$.
If $\pi$ has no loops, then the second case above may be ignored and $\C^\dagger$ coincides with the definition of the `$n$-truncated dual' given in Section \ref{section: friezedual}.

If $\C$ is a general $\pi$-prefrieze, then $\C^\dagger$ need not be a $\pi'$-prefrieze for any juggling function $\pi'$. However, the dual acts by an involution on the set of juggler's friezes, by the following theorem.

\begin{namedthm}{Theorem~\ref{thm: juggleduality}}
If $\C$ is a $\pi$-frieze, then the dual $\C^\dagger$ is a $\pi^\dagger$-frieze, and $(\C^\dagger)^\dagger=\C$.
\end{namedthm}

Furthermore, $\pi$-friezes are characterized among $\pi$-prefriezes by this duality; see Lemma \ref{lemma: juggleduality}.

\begin{ex}
\label{ex: dualjugfrieze}
The dual to the $\pi$-frieze in Construction \ref{cons: pipre} is given below.
\[\begin{tikzpicture}[baseline=(current bounding box.south),
    ampersand replacement=\&,
    ]
    \clip[use as bounding box] (-9.5,-3.1) rectangle (6.3,0.3);
    \matrix[matrix of math nodes,
        matrix anchor = M-2-24.center,
        nodes in empty cells,
        inner sep=0pt,
        gthrow/.style={dark green,draw,circle,inner sep=0mm,minimum size=4mm},
        rthrow/.style={dark red,draw,circle,inner sep=0mm,minimum size=4mm},
        bthrow/.style={dark blue,draw,circle,inner sep=0mm,minimum size=4mm},
        pthrow/.style={dark purple,draw,circle,inner sep=0mm,minimum size=4mm},
        nodes={anchor=center,node font=\scriptsize},
        column sep={0.4cm,between origins},
        row sep={0.4cm,between origins},
    ] (M) at (0,0) {
    \&  \&  \&  \&  \&  \&  \&  \&  \&  \&  \&  \&  \&  \&  \&  \&  \&  \&  \&  \&  \&  \&  \&  \&  \&  \&  \&  \&  \&  \&  \&  \&  \&  \&  \&  \&  \&  \&  \\
 \& |[gthrow]| 1 \&  \& |[gthrow]| 1 \&  \& |[gthrow]| 1 \&  \& |[gthrow]| 1 \&  \& |[gthrow]| 1 \&  \& |[gthrow]| 1 \&  \& |[gthrow]| 1 \&  \& |[gthrow]| 1 \&  \& |[gthrow]| 1 \&  \& |[gthrow]| 1 \&  \& |[gthrow]| 1 \&  \& |[gthrow]| 1 \&  \& |[gthrow]| 1 \&  \& |[gthrow]| 1 \&  \& |[gthrow]| 1 \&  \& |[gthrow]| 1 \&  \& |[gthrow]| 1 \&  \& |[gthrow]| 1 \&  \& |[gthrow]| 1 \&  \\
    \cdots \& \& 5 \& \& 1 \& \& 5 \& \& 1 \& \& 1 \& \& 3 \& \& 2 \& \& 1 \& \& 5 \& \& 1 \& \& 5 \& \& 1 \& \& 1 \& \& 3 \& \& 2 \& \& 1 \& \& 5 \& \& 1 \& \& \cdots \\
     \& 2 \& \& 3 \& \& 2 \& \& 3 \& \& 1 \& \& 3 \& \& 1 \& \& |[bthrow]| 1 \& \& 2 \& \& 3 \& \& 2 \& \& 3 \& \& 1 \& \& 3 \& \& 1 \& \& |[bthrow]| 1 \& \& 2 \& \& 3 \& \& 2 \& \\
    \cdots  \& \& |[bthrow]| 1 \& \& |[bthrow]| 1 \& \& 1 \& \& 3 \& \& |[bthrow]| 1 \& \& 1 \& \&  \& \&  \& \& |[bthrow]| 1 \& \& |[bthrow]| 1 \& \& 1 \& \& 3 \& \& |[bthrow]| 1 \& \& 1 \& \&  \& \&  \& \& |[bthrow]| 1 \& \& |[bthrow]| 1 \& \& \cdots \\
    \&  \& \&  \& \&  \& \& |[bthrow]| 1 \& \&  \& \&  \& \&  \& \& -1 \& \&  \& \&  \& \&  \& \& |[bthrow]| 1 \& \&  \& \&  \& \&  \& \& -1 \& \&  \& \&  \& \&  \&  \\
     \& \&  \& \&  \& \&  \& \&  \& \& |[bthrow]| -1 \& \&  \& \& |[bthrow]| -1 \& \&  \& \&  \& \&  \& \&  \& \&  \& \& |[bthrow]| -1 \& \&  \& \& |[bthrow]| -1 \& \&  \& \&  \& \&  \& \&  \\
    \&  \& \&  \& \&  \& \&  \& \&  \& \&  \& \&  \& \&  \& \&  \& \&  \& \&  \& \&  \& \&  \& \&  \& \&  \& \&  \& \&  \& \&  \& \&  \&  \\
     \& \& |[bthrow]| -1 \& \&  \& \&  \& \&  \& \&  \& \&  \& \&  \& \&  \& \& |[bthrow]| -1 \& \&  \& \&  \& \&  \& \&  \& \&  \& \&  \& \&  \& \& |[bthrow]| -1 \& \&  \& \& \\
    };

    \draw[traj,opacity=.25] (M-3-1) to (M-2-2) to (M-5-5) to (M-2-8) to (M-5-11) to (M-2-14) to (M-4-16) to (M-2-18) to (M-5-21) to (M-2-24) to (M-5-27) to (M-2-30) to (M-4-32) to (M-2-34) to (M-5-37) to (M-3-39);
    \draw[traj,opacity=.25] (M-5-1) to (M-2-4) to (M-6-8) to (M-2-12) to (M-9-19) to (M-2-26) to (M-7-31) to (M-2-36) to (M-5-39);
    \draw[traj,opacity=.25] (M-3-1) to (M-5-3) to (M-2-6) to (M-7-11) to (M-2-16) to (M-5-19) to (M-2-22) to (M-7-27) to (M-2-32) to (M-5-35) to (M-2-38) to (M-3-39);
    \draw[traj,opacity=.25] (M-7-1) to (M-9-3) to (M-2-10) to (M-7-15) to (M-2-20) to (M-6-24) to (M-2-28) to (M-9-35) to (M-5-39);
\end{tikzpicture}
\]
This is easily checked to be a $\pi^\dagger$-prefrieze, where $\pi^\dagger$ is has siteswap notation 23345357. The frieze and tameness conditions are less obvious; however, by Lemma \ref{lemma: juggleduality}, the fact that $\C^\dagger$ is a $\pi^\dagger$-prefrieze immediately implies that $\C$ (respectively, $\C^\dagger$) is a $\pi$-frieze (respectively, a $\pi^\dagger$-frieze).
\end{ex}


\subsection{Juggler's friezes as linear recurrences}

Regarded as a matrix (as in Section \ref{section: matrices}), a juggler's frieze $\C$ determines a linear recurrence $\C \mathsf{x}=0$. As in Theorem \ref{thm: friezeperiodic}, juggler's friezes are characterized among $\pi$-prefriezes as those whose solutions are \emph{superperiodic} in the following sense.

\begin{namedthm}{Theorem~\ref{thm: juggleqp}}
Let $\pi$ be an $n$-periodic juggling function with $h$-many balls.
A $\pi$-prefrieze $\C$ is a $\pi$-frieze iff every solution to the linear recurrence $\C \mathsf{x=0}$ 
satisfies $x_{a+n} = (-1)^{n-h-1}x_a$, $\forall a\in \mathbb{Z}$.
\end{namedthm}

As a consequence of the theorem, the space of solutions to a $\pi$-frieze is
invariant under shifting the indices of a sequence by $n$; that is, sending
each sequence $x$ to the sequence $S(x)$ with $S(x)_a := x_{a+n}$. Since
juggler's friezes are determined by their space of solutions
(Lemma \ref{lemma: linrec}, part \ref{lemma: linrec1}),
this immediately implies that juggler's friezes are periodic.

\begin{namedthm}{Theorem~\ref{thm: periodicity}}
If $\pi$ is an $n$-periodic juggling function, then the entries of a $\pi$-frieze $\C$ are $n$-periodic; that is, $\C_{a,b} = \C_{a+n,b+n}$.
\end{namedthm}
\begin{rem}
Note that the entries of a juggler's frieze are periodic, even though the solutions to the corresponding linear recurrence are only superperiodic.
\end{rem}

Using the dual juggler's frieze and superperiodicity, we can describe a spanning family of solutions to the linear recurrence.

\begin{namedthm}{Theorem~\ref{thm: qpsols}}
Let $\pi$ be an $n$-periodic jugging function with $h$-many balls, and let $\C$ be a $\pi$-frieze. For any $b$ with $b<\pi(b)$, the sequence $\mathsf{x}$ defined by 
\begin{itemize}
    \item $x_a := (-1)^{a+b} \C^\dagger_{a,b} $ when $b\leq a< b+k+h$, and 
    \item $x_{a+n} = (-1)^{n-h-1}x_a$ for all $a$
\end{itemize}
is a solution to $\C \mathsf{x=0}$. These solutions collectively span the space of solutions to $\C \mathsf{x=0}$.
\end{namedthm}

\noindent When $b=\pi(b)$, the sequence defined above is not a solution to $\C \mathsf{x=0}$.

\begin{rem}
This spanning set can be refined to a basis; see Remark \ref{rem: schedule}.
\end{rem}


\section{Constructing juggler's friezes}

\label{section: jugcons}

As in Section \ref{section: constructfrieze}, we may construct all juggler's friezes by twisting certain matrices.



\subsection{$\pi$-unimodular matrices}

First, we need a class of matrices that parametrize juggler's friezes.
Given a juggling function $\pi$, the \textbf{landing schedule at $a\in \mathbb{Z}$} is the set
\[ L_a := \{ b\in \mathbb{Z} \mid \pi^{-1}(b)<a \leq b \} \]
Intuitively, $L_a$ consists of the times at which the balls in the air right before time $a$ will be caught.
Matching earlier notation, we let $\overline{L_a}\subset \{1,2,...,n\}$ denote the residues of $L_a$ mod $n$.

We list some basic properties of landing schedules (see \cite{KLS13}).
\begin{enumerate}
    \item For all $a\in\mathbb{Z}$, $|L_a|$ equals the number of balls in $\pi$.
    \item If $a<\pi(a)$, then $a\in L_a$ and $L_{a+1} = L_a \smallsetminus \{a\} \cup \{\pi(a)\}$.
    \item If $a=\pi(a)$, then $a\not\in L_a$ and $L_{a+1}=L_a$.
\end{enumerate}

\begin{rem}
The list $(\overline{L}_1,\overline{L}_2,....,\overline{L}_n)$ is the \emph{Grassmann necklace}, a type of combinatorial object in bijection with positroids and juggling functions (see \cite{KLS13}).
\end{rem}

\begin{defn}
\label{defn: piunimodular}
Let $\pi$ be $n$-periodic with $k$-many balls. A $k\times n$-matrix $\mathsf{A}$ is \textbf{$\pi$-unimodular} if\footnote{Recall that, for any $I \subset \mathbb{Z}$, $\mathsf{A}_I$ denotes the submatrix of $\mathsf{A}$ with columns congruent to $I$ mod $n$, in the order in which they appear in $\mathsf{A}$.}
\begin{itemize}
    \item $\det(\mathsf{A}_{L_a})=1$ for all $a\in \mathbb{Z}$, and 
    \item for all intervals $[a,b]\subset \mathbb{Z}$ with $b<a+n$, $\mathrm{rank}(\mathsf{A}_{[a,b]}) \leq |L_a \cap [a,b]|$.
\end{itemize}
\end{defn}

\begin{rem}
The second condition is one of many equivalent ways to state that the rowspan of $\mathsf{A}$ is in the positroid variety corresponding to $\pi$; see \cite[Section 5]{KLS13}.
\end{rem}



\begin{ex}
\label{ex: piunimodular}
If $\pi$ is the juggling function with siteswap notation 23345357, then
\[ \begin{array}{|c|cccccccc|}
\hline
a & 1 & 2 & 3 & 4 & 5 & 6 & 7 & 8 \\
\hline 
\overline{L_a} & 1247 & 2347 & 3457 & 4567 & 5678 & 2678 & 1278 & 1248 \\
\hline
\end{array}\]
The following matrix is $\pi$-unimodular; note the 8 minors indexed by the above $\overline{L_a}$ are 1.
\[
\begin{bmatrix}
1 & 0 & -1 & 0 & 1 & 2 & 0 & -3 \\
0 & 1 & 2 & 0 & -1 & -1 & 0 & 1 \\
0 & 0 & 0 & 1 & 2 & 1 & 0 & -1 \\
0 & 0 & 0 & 0 & 0 & 0 & 1 & 1
\end{bmatrix}\qedhere \]
\end{ex}

\subsection{Constructing juggler's friezes}
\label{section: jugtwist}

Given a $\pi$-unimodular $k\times n$-matrix $\mathsf{A}$, define $\Fr(\mathsf{A})$ to be the diamond grid of numbers (or $\mathbb{Z}\times \mathbb{Z}$ matrix) with the following entries.
\begin{equation}\label{eq: detformula}
\Fr(\mathsf{A})_{a,b} := \begin{cases}
(-1)^{|S_{\pi^\dagger}(b,a)|} \det(\mathsf{A}_{L_a\smallsetminus \{a\} \cup \{b\} }) & \text{if $b\leq a < b+n$ and $a<\pi(a)$}, \\
1 & \text{if }\pi(a)=a=b, \\
(-1)^{k} & \text{if }\pi(a)=a=b+n, \\
0 & \text{otherwise}.
\end{cases}
\end{equation}
In this definition, if $b\in L_a\smallsetminus \{a\}$, we set $\det(\mathsf{A}_{L_a\smallsetminus \{a\} \cup \{b\} })=0$.
 If there are no loops in $\pi$ (that is, no $a$ with $\pi(a)=a$), then this definition simplifies to
\[ \Fr(\mathsf{A})_{a,b} = \begin{cases}
(-1)^{|S_{\pi^\dagger}(b,a)|} \det(\mathsf{A}_{L_a\smallsetminus \{a\} \cup \{b\} }) & \text{if $b\leq a < b+n$}, \\
0 & \text{otherwise}.
\end{cases} \]


\begin{namedthm}{Theorem~\ref{thm: detpifrieze}}
If $\mathsf{A}$ is a $\pi$-unimodular matrix, then $\Fr(\mathsf{A})$ is a $\pi^\dagger$-frieze.
%
%
\end{namedthm}

\begin{ex}
If $\mathsf{A}$ is the $\pi$-unimodular matrix in the preceding example, then a few entries of $\Fr(\mathsf{A})$ may be computed as follows.
\begin{align*}
\Fr(\mathsf{A})_{2,1} &= (-1)^{0} \det(\mathsf{A}_{L_2 \smallsetminus \{2\} \cup \{1\}}) 
= \det(\mathsf{A}_{\{1347\}}) = 2
%
 \\
\Fr(\mathsf{A})_{2,-2} &= (-1)^1\det (\mathsf{A}_{L_2 \smallsetminus \{2\} \cup \{ -2\} })
=(-1) \det(\mathsf{A}_{\{3467\}}) = -3
\\
\Fr(\mathsf{A})_{6,1} &= (-1)^1 \det(\mathsf{A}_{L_6 \smallsetminus \{6\} \cup \{1\}} ) = (-1)\det(\mathsf{A}_{L_7}) = -1
\qedhere
\end{align*}
\end{ex}


The juggler's frieze $\Fr(\mathsf{A})$ defined in Theorem \ref{thm: detpifrieze} can be constructed from (a more general version of) the twist. 
The \textbf{(right) twist} of a $\pi$-unimodular $k\times n$-matrix $\mathsf{A}$ is the $k\times n$-matrix $\rt(\mathsf{A})$ whose columns are defined by the dot product equations
\[ \rt(\mathsf{A})_a \cdot \mathsf{A}_b =
\begin{cases}
1 & \text{if }b\in L_a \cap \{a\} \\
0 & \text{if }b\in L_a \smallsetminus \{a\} 
\\
\end{cases}
\]
Since $\det(\mathsf{A}_{L_a})=1$ for each $a$, the above system uniquely determines each column of $\rt(\mathsf{A})$. Note that, if $a$ is not a loop of $\pi$, then $a\in L_a$ and the first condition above reduces to $\rt(\mathsf{A})_a\cdot \mathsf{A}_a=1$. If $a$ is a loop of $\pi$, then $\rt(\mathsf{A})_a$ is orthogonal to a basis and so $\rt(\mathsf{A})_a=0$.


\vspace{-0.3em}
\begin{rem}
If $\mathsf{A}$ is $\pi$-unimodular, then $\tau(\mathsf{A})$ is $\pi$-unimodular.
\end{rem}
\vspace{-0.3em}

The twist can be used to give an alternate construction of $\Fr(\mathsf{A})$ (Construction \ref{cons: piunwrap}, next page), which is both faster in practice and the original motivation for the definition of juggler's friezes.

\begin{namedthm}{Theorem~\ref{thm: equivcons}}
If  $\pi$ is loop-free and $\mathsf{A}$ is a $\pi$-unimodular matrix, then Construction \ref{cons: piunwrap} produces $\Fr(\mathsf{A})$, the $\pi^\dagger$-frieze defined in Theorem \ref{thm: detpifrieze}.
\end{namedthm}

\vspace{-0.3em}
\noindent See Remark \ref{rem: loopcons} for a modification of this construction which covers loops.

\newpage

\begin{cons}\label{cons: piunwrap}
Let $\mathsf{A}$ be a $\pi$-unimodular $k\times n$-matrix, for some loop-free juggling function $\pi$. As an example with siteswap 23345357,
\[ \mathsf{A} := 
\begin{bmatrix}
1 & 0 & -1 & 0 & 1 & 2 & 0 & -3 \\
0 & 1 & 2 & 0 & -1 & -1 & 0 & 1 \\
0 & 0 & 0 & 1 & 2 & 1 & 0 & -1 \\
0 & 0 & 0 & 0 & 0 & 0 & 1 & 1
\end{bmatrix}\]
\begin{enumerate}
    \item Compute the twist $\rt(\mathsf{A})$ of $\mathsf{A}$.
    \[ \rt(\mathsf{A}) = 
    \begin{bmatrix}
    1 & 2 & 1 & 1 & 0 & -1 & 0 & 0 \\
    0 & 1 & 1 & 3 & 1 & 0 & 0 & 0 \\
    0 & 0 & 0 & 1 & 1 & 3 & 1 & 0 \\
    0 & 0 & 0 & 0 & 0 & 0 & 1 & 1
    \end{bmatrix}\]
    \item Compute the matrix product $\rt(\mathsf{A})^\top \mathsf{A}$.
    \[ \rt(\mathsf{A})^\top \mathsf{A} = 
    \begin{bmatrix}
    1 & 0 & -1 & 0 & 1 & 2 & 0 & -3 \\
    2 & 1 & 0 & 0 & 1 & 3 & 0 & -5 \\
    1 & 1 & 1 & 0 & 0 & 1 & 0 & -2 \\
    1 & 3 & 5 & 1 & 0 & 0 & 0 & -1 \\
    0 & 1 & 2 & 1 & 1 & 0 & 0 & 0 \\
    -1 & 0 & 1 & 3 & 5 & 1 & 0 & 0 \\
    0 & 0 & 0 & 1 & 2 & 1 & 1 & 0 \\
    0 & 0 & 0 & 0 & 0 & 0 & 1 & 1
    \end{bmatrix}\]
    \item Take the entries above the diagonal of $\rt(\mathsf{A})^\top \mathsf{A}$, multiply them by $(-1)^{k-1}$, and slide them to the left of the rest of the matrix. 
    \[
    \begin{tikzpicture}[baseline=(current bounding box.south),
    ampersand replacement=\&,
    ]
    \matrix[matrix of math nodes,
        matrix anchor = M-1-8.center,
        origin/.style={},
        throw/.style={},
        pivot/.style={draw,circle,inner sep=0.25mm,minimum size=2mm},       
        nodes in empty cells,
        inner sep=0pt,
        nodes={anchor=center},
        column sep={.5cm,between origins},
        row sep={.5cm,between origins},
    ] (M) at (0,0) {
    0 \& 1 \& 0 \& -1 \& -2 \& 0 \& 3 \& 1 \&  \&  \&  \&  \&  \&  \&  \\
    \& 0 \& 0 \& -1 \& -3 \& 0 \& 5 \& 2 \& 1 \&  \&  \&  \&  \&  \&   \\
    \& \& 0 \& 0 \& -1 \& 0 \& 2 \& 1 \& 1 \& 1 \&  \&  \&  \&  \&  \\
    \& \& \& 0 \& 0 \& 0 \& 1 \& 1 \& 3 \& 5 \& 1 \&  \&  \&  \&  \\
    \& \& \& \& 0 \& 0 \& 0 \& 0 \& 1 \& 2 \& 1 \& 1 \&  \&  \& \\ 
    \& \& \& \& \& 0 \& 0 \& -1 \& 0 \& 1 \& 3 \& 5 \& 1 \&  \& \\ 
    \& \& \& \& \& \& 0 \& 0 \& 0 \& 0 \& 1 \& 2 \& 1 \& 1 \& \\ 
    \& \& \& \& \& \& \& 0 \& 0 \& 0 \& 0 \& 0 \& 0 \& 1 \& 1 \\ 
    };
    \draw[thick,dark green,fill=dark green!25,opacity=.25,rounded corners=5mm] ($(M-1-8.center)+(-.25,.5)$) -- ($(M-8-15.center)+(.5,-.25)$) -- ($(M-8-8.center)+(-.25,-.25)$) -- cycle;
    \draw[thick,dark red,fill=dark red!25,opacity=.25,rounded corners] ($(M-1-7.center)+(.25,.25)$) -- ($(M-7-7.center)+(.25,-.5)$) -- ($(M-1-1.center)+(-.5,.25)$) -- cycle;
    \end{tikzpicture}
    \]
    \item Rotate the resulting parallelogram $45^\circ$ counter-clockwise, delete superfluous $0$s, and duplicate the remaining numbers $n$-periodically in an infinite horizontal strip.
    \[
    \begin{tikzpicture}[baseline=(current bounding box.south),
    ampersand replacement=\&,
    ]
    \clip[use as bounding box] (-7.1,.3) rectangle (7.1,-2.7);
    \matrix[matrix of math nodes,
        matrix anchor = M-1-29.center,
        origin/.style={},
        throw/.style={},
        pivot/.style={draw,circle,inner sep=0.25mm,minimum size=2mm},       
        nodes in empty cells,
        inner sep=0pt,
        nodes={anchor=center},
        column sep={.4cm,between origins},
        row sep={.4cm,between origins},
    ] (M) at (0,0) {
    \& \& \& \& \& 1 \& \& 1 \& \& 1 \& \& 1 \& \& 1 \& \& 1 \& \& 1 \& \& 1 \& \& 1 \& \& 1 \& \& 1 \& \& 1 \& \& 1 \& \& 1 \& \& 1 \& \& 1 \& \& 1 \& \& 1 \& \& 1 \& \& 1 \& \& 1 \& \& 1 \& \& 1 \& \& 1 \& \& \\
    \& \& \& \& 3 \& \& 2 \& \& 1 \& \& 5 \& \& 1 \& \& 5 \& \& 1 \& \& 1 \& \& 3 \& \& 2 \& \& 1 \& \& 5 \& \& 1 \& \& 5 \& \& 1 \& \& 1 \& \& 3 \& \& 2 \& \& 1 \& \& 5 \& \& 1 \& \& 5 \& \& 1 \& \& 1 \& \& \\
    \& \& \&  \& \& 5 \& \& 1 \& \& 3 \& \& 2 \& \& 3 \& \& 2 \& \&  \& \&  \& \& 5 \& \& 1 \& \& 3 \& \& 2 \& \& 3 \& \& 2 \& \&  \& \&  \& \& 5 \& \& 1 \& \& 3 \& \& 2 \& \& 3 \& \& 2 \& \&  \& \& \\
    \& \& -2 \& \&  \& \& 2 \& \& 1 \& \& 1 \& \& 1 \& \& 1 \& \&  \& \& -2 \& \&  \& \& 2 \& \& 1 \& \& 1 \& \& 1 \& \& 1 \& \&  \& \& -2 \& \&  \& \& 2 \& \& 1 \& \& 1 \& \& 1 \& \& 1 \& \&  \& \& \\
    \& -1 \& \& -3 \& \&  \& \& 1 \& \&  \& \&  \& \&  \& \&  \& \& -1 \& \& -3 \& \&  \& \& 1 \& \&  \& \&  \& \&  \& \&  \& \& -1 \& \& -3 \& \&  \& \& 1 \& \&  \& \&  \& \&  \& \&  \& \& \\
     \& \& -1 \& \& -1 \& \&  \& \&  \& \& -1 \& \&  \& \&  \& \&  \& \& -1 \& \& -1 \& \&  \& \&  \& \& -1 \& \&  \& \&  \& \&  \& \& -1 \& \& -1 \& \&  \& \&  \& \& -1 \& \&  \& \&  \& \& \\
    \&  \& \&  \& \&  \& \&  \& \&  \& \&  \& \&  \& \& 1 \& \&  \& \&  \& \&  \& \&  \& \&  \& \&  \& \&  \& \& 1 \& \&  \& \&  \& \&  \& \&  \& \&  \& \&  \& \&  \& \& 1 \& \& \\
    };
    \draw[thick,dark green,fill=dark green!25,opacity=.25,rounded corners=3mm] ($(M-1-6.center)+(-.5,.25)$) -- ($(M-1-20.center)+(.5,.25)$) -- ($(M-7-14.center)+(.25,-.25)$) -- ($(M-7-12.center)+(-.25,-.25)$) -- cycle;
    \draw[thick,dark red,fill=dark red!25,opacity=.25,rounded corners=3mm] ($(M-2-5.center)+(0,.5)$) -- ($(M-7-10.center)+(.5,-.25)$) -- ($(M-7-1.center)+(-.5,-.25)$) -- cycle;
    \draw[thick,dark green,fill=dark green!25,opacity=.25,rounded corners=3mm] ($(M-1-22.center)+(-.5,.25)$) -- ($(M-1-36.center)+(.5,.25)$) -- ($(M-7-30.center)+(.25,-.25)$) -- ($(M-7-28.center)+(-.25,-.25)$) -- cycle;
    \draw[thick,dark red,fill=dark red!25,opacity=.25,rounded corners=3mm] ($(M-2-21.center)+(0,.5)$) -- ($(M-7-26.center)+(.5,-.25)$) -- ($(M-7-16.center)+(-.5,-.25)$) -- cycle;
    \draw[thick,dark green,fill=dark green!25,opacity=.25,rounded corners=3mm] ($(M-1-38.center)+(-.5,.25)$) -- ($(M-1-52.center)+(.5,.25)$) -- ($(M-7-46.center)+(.25,-.25)$) -- ($(M-7-44.center)+(-.25,-.25)$) -- cycle;
    \draw[thick,dark red,fill=dark red!25,opacity=.25,rounded corners=3mm] ($(M-2-37.center)+(0,.5)$) -- ($(M-7-42.center)+(.5,-.25)$) -- ($(M-7-32.center)+(-.5,-.25)$) -- cycle;
    \end{tikzpicture}
    \qedhere
    \]
\end{enumerate}
\end{cons}

\subsection{Relation to the Grassmannian}

Since the formula for $\Fr(\mathsf{A})$ in Theorem \ref{thm: detpifrieze} only depends on determinants of $k\times k$-submatrices of $\mathsf{A}$, the juggler's frieze $\Fr(\mathsf{A})$ only depends on the left $\mathrm{SL}(k)$-orbit of $\mathsf{A}$ within the set of $\pi$-unimodular matrices.
This can be sharpened to a bijection.

\begin{namedthm}{Theorem~\ref{thm: Frbijection}}
For each juggling function $\pi$ with $k$-many balls,
the map $\Fr$ descends to a bijection
\[ \mathrm{SL}(k)\backslash \{ \text{$\pi$-unimodular matrices}\} \xrightarrow{\Fr} \{\text{$\pi^\dagger$-friezes}\} \]
\end{namedthm}


As a consequence, $\Fr$ factors through a well-known description of the Grassmannian.
\[ \mathrm{SL}(k)\backslash \{ \text{$\pi$-unimodular matrices}\}
\subset \mathrm{GL}(k)\backslash \{\text{$k\times n$-matrices of rank $k$}\}
 \xrightarrow{\mathrm{rowspan}} \mathrm{Gr}(k,n) \]
 Since the $\pi$-unimodular matrices may be defined in terms of $k\times k$-minors, the image of this inclusion is a closed subvariety of $\mathrm{Gr}(k,n)$, which we denote by $\mathrm{Gr}_\uni(\pi)$. That is,
 \[\mathrm{Gr}_{\uni}(\pi) := \{\text{rowspans of $\pi$-unimodular matrices}\} \subset \mathrm{Gr}(k,n) \]
 


\begin{coro}
The map $\Fr$ induces a bijection $\mathrm{Gr}_\uni(\pi)\xrightarrow{\sim} \{\text{$\pi^\dagger$-friezes}\}$.
\end{coro}

\begin{rem}
The variety $\mathrm{Gr}_\uni(\pi)$ is related to the \emph{positroid variety}
of $\pi$ \cite{KLS13}, which we denote $\mathrm{Gr}(\pi)$, consisting of
rowspans of matrices satisfying only the second condition in Definition
\ref{defn: piunimodular}. The variety $\mathrm{Gr}_\uni(\pi)$ is then the subset
of $\mathrm{Gr}(\pi)$ on which the Pl\"ucker coordinates of the form
$\Delta_{L_a}$ are all equal to $1$; it was dubbed the \emph{critical variety} in \cite{Gal21}.
The bijection $\Fr$ can be extended to $\mathrm{Gr}(\pi)$, by replacing $\pi^\dagger$-friezes with \emph{quasiperiodic reduced
recurrence matrices with juggling function $\pi^\dagger$}. Many of the results
in this paper extend, as we hope to show in \cite{MulRes2}.
\end{rem}

Sending a matrix to its positive complement descends to isomorphisms
\[ \ddagger : \mathrm{Gr}(\pi) \xrightarrow{\sim} \mathrm{Gr}(\pi^\dagger) 
\text{ and }
\ddagger : \mathrm{Gr}_{\uni}(\pi) \xrightarrow{\sim} \mathrm{Gr}_{\uni}(\pi^\dagger) 
\]
with inverses given by applying $\ddagger$ again.
%
%
%
%
Duality for friezes is then related to the twist as follows.

\begin{namedthm}
{Theorem~\ref{thm: dualities}}
Let $\mathsf{A}$ be a $\pi$-unimodular matrix. Then
\[ \mathrm{F}(\mathsf{A})^\dagger = \mathrm{F}(\rt(\mathsf{A}) ^\pperp) = \mathrm{F}(\rt^{-1}(\mathsf{A} ^\pperp)) \]
\end{namedthm}

%
%

The preceding theorem implies the commutativity of the following diagram of bijections.
\[ \begin{tikzpicture}[xscale=3.5, yscale=1.75]
    \node (a0) at (0,0) {$\mathrm{Gr}_{\text{uni}}(\pi)$};
    \node (a1) at (1,.33) {$\mathrm{Gr}_{\text{uni}}(\pi)$};
    \node (c1) at (1,-.33) {$\mathrm{Gr}_{\text{uni}}(\pi^\dagger)$};
    \node (a2) at (2,0) {$\mathrm{Gr}_{\text{uni}}(\pi^\dagger)$};
    \node (b0) at (0,-1) {$\{ \text{$\pi^\dagger$-friezes}\} $};
    \node (b2) at (2,-1) {$\{ \text{$\pi$-friezes}\} $};

    \draw[-angle 90] (a0) to node[above] {$\tau$} (a1);
    \draw[angle 90-angle 90] (a1) to node[above] {$\ddagger$} (a2);
    \draw[-angle 90] (a2) to node[below] {$\tau$} (c1);
    \draw[angle 90-angle 90] (c1) to node[below] {$\ddagger$} (a0);
    \draw[angle 90-angle 90] (b0) to node[above] {$\dagger$} (b2);
    
    \draw[-angle 90] (a0) to node[left] {$\mathrm{F}$} (b0);
    \draw[-angle 90] (a2) to node[left] {$\mathrm{F}$} (b2);
\end{tikzpicture}\]

\begin{ex}
A computation of $\mathsf{A}^\ddagger$ and $\Fr(\lt(\mathsf{A}^\ddagger))$ for the matrix $\mathsf{A}$ in Example \ref{ex: piunimodular} is given in the next section.
The latter may be seen to coincide with $\Fr(\mathsf{A})^\dagger$, appearing in Example \ref{ex: dualjugfrieze}.
\end{ex}

\subsection{Positivity, enumeration, and clusters}\label{section: clusters}


As mentioned in Section \ref{section: breezy}, many notable results count positive integral $\mathrm{SL}(k)$-friezes, often via bijections with other interesting sets of objects. 



To pose an analogous problem for juggler's friezes, we need to amend the notion of positivity.
\begin{defn}
A $\pi$-prefrieze over $\mathbb{R}$ is \textbf{positive} if $(-1)^{|S_\pi(a,b)|}\C_{a,b} >0$ when not forced to be $0$.
\end{defn}

\begin{rem}
When $\pi$ is uniform, $S_\pi(a,b)$ is empty whenever $\C_{a,b}$ is not zero, and so the positivity condition becomes $\C_{a,b}> 0$ when $\C_{a,b}$ is not forced to be $0$.
\end{rem}

Many of the constructions in this paper have been defined so as to preserve positivity. For example, the dual of a positive juggler's frieze is positive, and $\Fr$ defines a bijection between the positive part of $\mathrm{Gr}_{\uni}(\pi)$ and the set of positive $\pi^\dagger$-friezes.


Restricting to the positive integral friezes gives a potentially interesting enumerative problem.

\begin{prob}
For which juggling functions $\pi$ are there finitely many positive integral $\pi$-friezes? Given such $\pi$, how many are there? Are they in bijection with other objects of interest?
\end{prob}


%
%

We claim that one robust source of juggler's friezes is a \emph{cluster structure} on $\mathrm{Gr}_{\uni}(\pi)$, which we hope to prove in a subsequent work \cite{MulRes2}.


\begin{claim}
The variety $\mathrm{Gr}_{\uni}(\pi)$ admits a cluster structure, 
each cluster on $\mathrm{Gr}_{\uni}(\pi)$ determines a positive integral $\pi^\dagger$-frieze, and the map from clusters to positive integral $\pi^\dagger$-friezes is an inclusion.
\end{claim}

\begin{rem}
We can say a bit about this cluster structure. 
The positroid variety $\mathrm{Gr}(\pi)$ has a cluster structure \cite{Sco06,Pos,Lec16, SSBW20, GL19}. Since $\mathrm{Gr}_{\uni}(\pi)\subset \mathrm{Gr}(\pi)$ is the subvariety on which the frozen cluster variables of $\mathrm{Gr}(\pi)$ are $1$, the cluster structure descends to $\mathrm{Gr}_{\uni}(\pi)$ and has the same mutation type.
In particular, the above claim implies that there are infinitely many positive integral $\pi^\dagger$-friezes whenever $\mathrm{Gr}(\pi)$ has infinite mutation type.
%
%
\end{rem}


%
%
%


\section*{Intermission: Computing a dual juggler's frieze via the twist}

\label{section: twistcomputation}

As an excuse to postpone the (rather technical) proofs of the main theorems, we provide a worked example of Theorem \ref{thm: dualities}.
Consider the matrix $\mathsf{A}$ from Example \ref{ex: piunimodular}, copied below.
\[ \mathsf{A} := 
\begin{bmatrix}
1 & 0 & -1 & 0 & 1 & 2 & 0 & -3 \\
0 & 1 & 2 & 0 & -1 & -1 & 0 & 1 \\
0 & 0 & 0 & 1 & 2 & 1 & 0 & -1 \\
0 & 0 & 0 & 0 & 0 & 0 & 1 & 1
\end{bmatrix}\]
Since $\mathsf{A}$ is in reduced row echelon form, a basis of the kernel of $\mathsf{A}$ may be computed via back substitution. We put this basis together into the rows of the following matrix.
    \[ \mathsf{B}
     = 
    \begin{bmatrix}
    1 & -2 & 1 & 0 & 0 & 0 & 0 & 0 \\
    -1 & 1 & 0 & -2 & 1 & 0 & 0 & 0 \\
    -2 & 1 & 0 & -1 & 0 & 1 & 0 & 0 \\
    3 & -1 & 0 & 1 & 0 & 0 & -1 & 1 
    \end{bmatrix}
    \]
A positive complement of $\mathsf{A}$ may be given by negating columns 1, 3, 5, and 7 of $\mathsf{B}$. However, it will be convenient to also negate rows 1 and 2 (as this will make the entries of the inverse twist non-negative). The resulting positive complement (and its inverse twist) are below.
    \[ \mathsf{A}^\ddagger
     = 
    \begin{bmatrix}
    1 & 2 & 1 & 0 & 0 & 0 & 0 & 0 \\
    -1 & -1 & 0 & 2 & 1 & 0 & 0 & 0 \\
    2 & 1 & 0 & -1 & 0 & 1 & 0 & 0 \\
    -3 & -1 & 0 & 1 & 0 & 0 & 1 & 1 
    \end{bmatrix}
    \hspace{1cm}
    \lt(\mathsf{A}^\ddagger) = 
    \begin{bmatrix}
    1 & 1 & 1 & 0 & 0 & 0 & 0 & 0 \\
    0 & 1 & 3 & 1 & 1 & 0 & 0 & 0 \\
    0 & 0 & 1 & 2 & 5 & 1 & 0 & 0 \\
    0 & 0 & 0 & 1 & 3 & 1 & 1 & 1 \\
    \end{bmatrix}\]

We now use Construction \ref{cons: piunwrap} to construct $\Fr(\lt(\mathsf{A}^\ddagger))$. We already know the twist of $\lt(\mathsf{A}^\ddagger)$, since $\rt(\lt(\mathsf{A}^\ddagger)) = \mathsf{A}^\ddagger$. Next, the matrix product $\rt(\lt(\mathsf{A}^\ddagger))^\top \lt(\mathsf{A}^\ddagger) = (\mathsf{A}^\ddagger)^\top \lt(\mathsf{A}^\ddagger)$.
    \begin{align*}
     (\mathsf{A}^\ddagger)^\top \lt(\mathsf{A}^\ddagger) 
    &=
    \begin{bmatrix}
    1 & 0 & 0 & 0 & 0 & -1 & -3 & -3 \\
    2 & 1 & 0 & 0 & 1 & 0 & -1 & -1 \\
    1 & 1 & 1 & 0 & 0 & 0 & 0 & 0 \\
    0 & 2 & 5 & 1 & 0 & 0 & 1 & 1 \\
    0 & 1 & 3 & 1 & 1 & 0 & 0 & 0 \\
    0 & 0 & 1 & 2 & 5 & 1 & 0 & 0 \\
    0 & 0 & 0 & 1 & 3 & 1 & 1 & 1 \\
    0 & 0 & 0 & 1 & 3 & 1 & 1 & 1 \\
    \end{bmatrix}
    \end{align*}

    Next, we take the entries above the diagonal of $\rt(\mathsf{A})^\top \mathsf{A}$, multiply them by $(-1)^{k-1}=(-1)^3$, and slide them to the left of the rest of the matrix. 
    
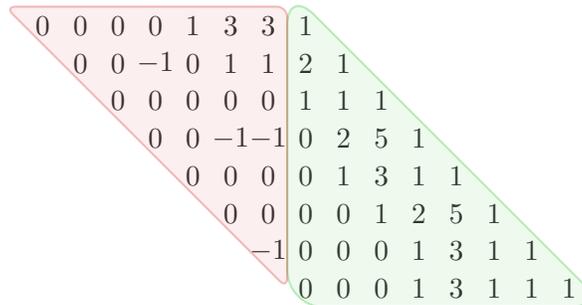
\[
    \begin{tikzpicture}[baseline=(current bounding box.south),
    ampersand replacement=\&,
    ]
    \matrix[matrix of math nodes,
        matrix anchor = M-1-8.center,
        origin/.style={},
        throw/.style={},
        pivot/.style={draw,circle,inner sep=0.25mm,minimum size=2mm},       
        nodes in empty cells,
        inner sep=0pt,
        nodes={anchor=center},
        column sep={.5cm,between origins},
        row sep={.5cm,between origins},
    ] (M) at (0,0) {
    0 \& 0 \& 0 \& 0 \& 1 \& 3 \& 3 \& 1 \&  \&  \&  \&  \&  \&  \&  \\
    \& 0 \& 0 \& -1 \& 0 \& 1 \& 1 \& 2 \& 1 \&  \&  \&  \&  \&  \&   \\
    \& \& 0 \& 0 \& 0 \& 0 \& 0 \& 1 \& 1 \& 1 \&  \&  \&  \&  \&  \\
    \& \& \& 0 \& 0 \& -1 \& -1 \& 0 \& 2 \& 5 \& 1 \&  \&  \&  \&  \\
    \& \& \& \& 0 \& 0 \& 0 \& 0 \& 1 \& 3 \& 1 \& 1 \&  \&  \& \\ 
    \& \& \& \& \& 0 \& 0 \& 0 \& 0 \& 1 \& 2 \& 5 \& 1 \&  \& \\ 
    \& \& \& \& \& \& -1 \& 0 \& 0 \& 0 \& 1 \& 3 \& 1 \& 1 \& \\ 
    \& \& \& \& \& \& \& 0 \& 0 \& 0 \& 1 \& 3 \& 1 \& 1 \& 1 \\ 
    };
    \draw[thick,dark green,fill=dark green!25,opacity=.25,rounded corners=5mm] ($(M-1-8.center)+(-.25,.5)$) -- ($(M-8-15.center)+(.5,-.25)$) -- ($(M-8-8.center)+(-.25,-.25)$) -- cycle;
    \draw[thick,dark red,fill=dark red!25,opacity=.25,rounded corners] ($(M-1-7.center)+(.25,.25)$) -- ($(M-7-7.center)+(.25,-.5)$) -- ($(M-1-1.center)+(-.5,.25)$) -- cycle;
    \end{tikzpicture}
    \]
    Finally, we rotate the resulting parallelogram $45^\circ$ counter-clockwise, delete superfluous $0$s, and duplicate the remaining numbers $n$-periodically in an infinite horizontal strip, yielding $\Fr(\lt(\mathsf{A}^\ddagger))$.
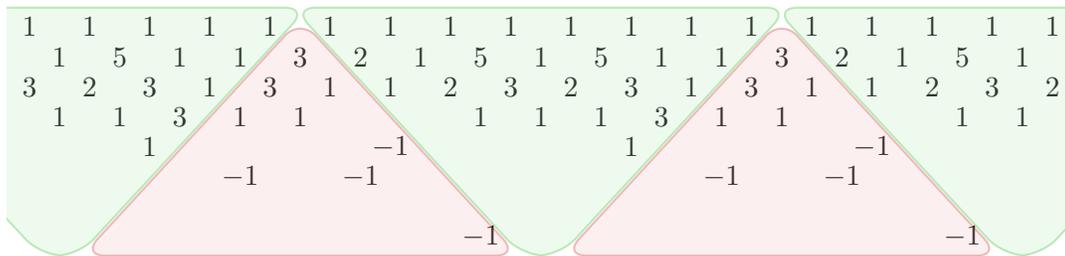
    \[
    \begin{tikzpicture}[baseline=(current bounding box.south),
    ampersand replacement=\&,
    ]
    \clip[use as bounding box] (-7.1,.3) rectangle (7.1,-3.1);
    \matrix[matrix of math nodes,
        matrix anchor = M-1-29.center,
        origin/.style={},
        throw/.style={},
        pivot/.style={draw,circle,inner sep=0.25mm,minimum size=2mm},       
        nodes in empty cells,
        inner sep=0pt,
        nodes={anchor=center},
        column sep={.4cm,between origins},
        row sep={.4cm,between origins},
    ] (M) at (0,0) {
    \& \& \& \& \& 1 \& \& 1 \& \& 1 \& \& 1 \& \& 1 \& \& 1 \& \& 1 \& \& 1 \& \& 1 \& \& 1 \& \& 1 \& \& 1 \& \& 1 \& \& 1 \& \& 1 \& \& 1 \& \& 1 \& \& 1 \& \& 1 \& \& 1 \& \& 1 \& \& 1 \& \& 1 \& \& 1 \& \& \\
    \& \& \& \& 3 \& \& 2 \& \& 1 \& \& 5 \& \& 1 \& \& 5 \& \& 1 \& \& 1 \& \& 3 \& \& 2 \& \& 1 \& \& 5 \& \& 1 \& \& 5 \& \& 1 \& \& 1 \& \& 3 \& \& 2 \& \& 1 \& \& 5 \& \& 1 \& \& 5 \& \& 1 \& \& 1 \& \& \\
    \& \& \& 3 \& \& 1 \& \& 1 \& \& 2 \& \& 3 \& \& 2 \& \& 3 \& \& 1 \& \& 3 \& \& 1 \& \& 1 \& \& 2 \& \& 3 \& \& 2 \& \& 3 \& \& 1 \& \& 3 \& \& 1 \& \& 1 \& \& 2 \& \& 3 \& \& 2 \& \& 3 \& \& 1 \& \& \\
    \& \& 1 \& \& 1 \& \&  \& \&  \& \& 1 \& \& 1 \& \& 1 \& \& 3 \& \& 1 \& \& 1 \& \&  \& \&  \& \& 1 \& \& 1 \& \& 1 \& \& 3 \& \& 1 \& \& 1 \& \&  \& \&  \& \& 1 \& \& 1 \& \& 1 \& \& 3 \& \& \\
    \&  \& \&  \& \&  \& \& -1 \& \&  \& \&  \& \&  \& \& 1 \& \&  \& \&  \& \&  \& \& -1 \& \&  \& \&  \& \&  \& \& 1 \& \&  \& \&  \& \&  \& \& -1 \& \&  \& \&  \& \&  \& \& 1 \& \& \\
     \& \& -1 \& \&  \& \& -1 \& \&  \& \&  \& \&  \& \&  \& \&  \& \& -1 \& \&  \& \& -1 \& \&  \& \&  \& \&  \& \&  \& \&  \& \& -1 \& \&  \& \& -1 \& \&  \& \&  \& \&  \& \&  \& \& \\
    \&  \& \&  \& \&  \& \&  \& \&  \& \&  \& \&  \& \&  \& \&  \& \&  \& \&  \& \&  \& \&  \& \&  \& \&  \& \&  \& \&  \& \&  \& \&  \& \&  \& \&  \& \&  \& \&  \& \&  \& \& \\
     \& \&  \& \&  \& \&  \& \&  \& \&  \& \&  \& \&  \& \&  \& \&  \& \&  \& \&  \& \&  \& \& -1 \& \&  \& \&  \& \&  \& \&  \& \&  \& \&  \& \&  \& \& -1 \& \&  \& \&  \& \& \\
    };
    \draw[thick,dark green,fill=dark green!25,opacity=.25,rounded corners=3mm] ($(M-1-6.center)+(-.5,.25)$) -- ($(M-1-20.center)+(.5,.25)$) -- ($(M-8-13.center)+(.25,-.25)$) -- ($(M-8-13.center)+(-.25,-.25)$) -- cycle;
    \draw[thick,dark green,fill=dark green!25,opacity=.25,rounded corners=3mm] ($(M-1-22.center)+(-.5,.25)$) -- ($(M-1-36.center)+(.5,.25)$) -- ($(M-8-29.center)+(.25,-.25)$) -- ($(M-8-29.center)+(-.25,-.25)$) -- cycle;
    \draw[thick,dark red,fill=dark red!25,opacity=.25,rounded corners=3mm] ($(M-2-21.center)+(0,.5)$) -- ($(M-8-27.center)+(.5,-.25)$) -- ($(M-8-15.center)+(-.5,-.25)$) -- cycle;
    \draw[thick,dark green,fill=dark green!25,opacity=.25,rounded corners=3mm] ($(M-1-38.center)+(-.5,.25)$) -- ($(M-1-52.center)+(.5,.25)$) -- ($(M-8-45.center)+(.25,-.25)$) -- ($(M-8-45.center)+(-.25,-.25)$) -- cycle;
    \draw[thick,dark red,fill=dark red!25,opacity=.25,rounded corners=3mm] ($(M-2-37.center)+(0,.5)$) -- ($(M-8-43.center)+(.5,-.25)$) -- ($(M-8-31.center)+(-.5,-.25)$) -- cycle;
    \end{tikzpicture}
    \]
    This juggler's frieze is readily seen to coincide with $\Fr(\mathsf{A})^\dagger$, as given in Example \ref{ex: dualjugfrieze}.

\newpage

\section{Proofs of Theorems}

\label{section: proofs}

In this section, we prove the aforementioned theorems. In the interest of precision, we will exclusively use  `prefriezes', `$\pi$-prefriezes', and related terms to refer to $\mathbb{Z}\times \mathbb{Z}$-matrices, rather than grids of numbers in a diamond pattern.
This primarily affects the meaning of the terms `row', `column', and `diagonal'; e.g.~in this section we would say a prefrieze has two diagonals of $1$s.

\begin{warn}
The reader is cautioned that a majority of what follows consists of verifying that various signs are correct. These arguments can be skipped without losing any key ideas.
\end{warn}


\subsection{Duality for juggler's friezes (Theorem \ref{thm: juggleduality})}


A $\mathbb{Z}\times \mathbb{Z}$-matrix $\C$ is \textbf{lower unitriangular} if $\C_{a,b}=0$ for all $a<b$ and $\C_{a,b}=1$ for all $a=b$. Rotated counterclockwise and deleting $0$s above the diagonal, such a matrix becomes a diamond grid of numbers whose top row consists of $1$s, but may continue infinitely downwards (as well as left and right, of course).
This definition covers half of the definition of a $\pi$-prefrieze. The other half can be split off of the definition as follows.
\begin{lemma}\label{lemma: prefrieze}
A lower unitriangular $\mathbb{Z}\times \mathbb{Z}$-matrix $\C$ is a $\pi$-prefrieze iff, for all $a,b$ with $b\leq \pi^{-1}(a)$ or $\pi(b)\leq a$, 
\begin{equation}\label{eq: prefrieze}
\C_{a,b} = \begin{cases}
(-1)^{|S_\pi(b,a)|} & \text{if }a = \pi(b), \\
0 & \text{if } b<\pi^{-1}(a) \text{ or } \pi(b) < a.
\end{cases}
\end{equation}
\end{lemma}


This condition on entries can be replaced by a condition on larger minors, via the following lemma. 
Recall that $S_\pi(a,b) := \{i \in \mathbb{Z} \mid a<i \text{ and }\pi(i) <b\}$.

\begin{lemma}
\label{lemma: cofrieze}
A lower unitriangular $\mathbb{Z}\times \mathbb{Z}$-matrix $\C$ is a $\pi$-prefrieze iff
 the following hold.
\begin{enumerate}
    \item For all $a<b$, 
    \begin{equation}\label{eq: cofrieze}
    \det (\C_{\pi(S_\pi(a,b)),S_\pi(a,b)}) = 1 
    \end{equation}
    \item For all $a<b$ such that $a<\pi^{-1}(b)$ or $\pi(a)< b$,
    \begin{equation}\label{eq: cotame}
    \det(\C_{\pi(S_\pi(a,b))\cup\{b\},\{a\}\cup S_\pi(a,b) }) = 0 
    \end{equation}
\end{enumerate}
\end{lemma}

\noindent We note that Condition (1) is redundant and can be optimized; see Remark \ref{rem: optimize}.


\begin{proof}
Note that if $|S_\pi(a,b)|=0$, Equation \eqref{eq: cofrieze} is vacuous and Equations \eqref{eq: prefrieze} and \eqref{eq: cotame} coincide.

$(\Rightarrow)$
Assume that $\C$ is a $\pi$-prefrieze, and that $\C$ satisfies Equations \eqref{eq: cofrieze} and \eqref{eq: cotame} for any $a,b$ with $|S_\pi(a,b)|<k$. 
Choose $a<b$ with $|S_\pi(a,b)|=k$, and index the elements in increasing order:
\[ S_\pi(a,b) = \{ i_1 < i_2 < \cdots i_ k \}\]
Note that, for all $j$, $S_\pi(i_{j},b) = \{i_{j+1},i_{j+2},...,i_k\} $ and $S_\pi(i_1,\pi(i_{j}))\subset S_\pi(i_1,b)\subsetneq S_\pi(a,b)$.

Therefore, $|S_\pi(i_1,\pi(i_j))| < |S_\pi(a,b)|=k$, and so $\C_{\pi(i_1),i_1}=(-1)^{|S_\pi(i_1,\pi(i_1))|}$ and $\C_{\pi(i_j),i_1}=0$ for $j>1$ by Equation \eqref{eq: prefrieze}. Taking a cofactor expansion of $\det(\C_{\pi(S_\pi(a,b)),S_\pi(a,b)})$ along the $i_1$th column and using that $\C_{\pi(i_1),i_1}$ is in the $(|S_\pi(i_1,\pi(i_1))|+1)$th row of $\C_{\pi(S_\pi(a,b)),S_\pi(a,b)}$,
\[ \det (\C_{\pi(S_\pi(a,b)),S_\pi(a,b)}) = (-1)^{|S_\pi(i_1,\pi(i_1))|}\C_{\pi(i_1),i_1} \det (\C_{\pi(S_\pi(i_1,b)),S_\pi(i_1,b)}) = \det (\C_{\pi(S_\pi(i_1,b)),S_\pi(i_1,b)}) \]
Since $S_\pi(i_1,b) \subsetneq S_\pi(a,b)$, the right-hand side is $1$ by the assumed case of Equation \eqref{eq: cotame}. Therefore, Equation \eqref{eq: cofrieze} holds for this $a,b$.

Now, assume that $a< \pi^{-1}(b)$ or $\pi(a)< b$, so that $\C_{b,a}=0$ by Equation \eqref{eq: prefrieze}.
Since $|S_\pi(i_{j},b)| < |S_\pi(a,b)|=k$ and $\pi(i_{j})<b$,
$ \C_{b,i_j} = 0 $ for all $j$ by Equation \eqref{eq: prefrieze}.
Taking a cofactor expansion of $\det(\C_{\pi(S_\pi(a,b))\cup \{b\} ,\{a\} \cup S_\pi(a,b)})$ along the $b$th row,
\[ \det(\C_{\pi(S_\pi(a,b))\cup \{b\} ,\{a\} \cup S_\pi(a,b)}) = (-1)^{|S_\pi(a,b)|} \C_{b,a} \det(\C_{\pi(S_\pi(a,b)) ,S_\pi(a,b)}) 
=0
\]
Therefore, Equation \eqref{eq: cotame} holds for this $a,b$. Taking the limit as $k$ goes to infinity implies that Equations \eqref{eq: cofrieze} and \eqref{eq: cotame} hold in general.

$(\Leftarrow)$
Assume that $\C$ satisfies Equations \eqref{eq: cofrieze} and \eqref{eq: cotame} in general, and that $\C$ satisfies Equation \eqref{eq: prefrieze} for any $a,b$ with $|S_\pi(a,b)|<k$. Choose some $a<b$ such that $|S_\pi(a,b)|=k$. Again taking a cofactor expansion of $\det(\C_{\pi(S_\pi(a,b))\cup \{b\} ,\{a\} \cup S_\pi(a,b)})$ along the $b$th row,
\[ \det(\C_{\pi(S_\pi(a,b))\cup \{b\} ,\{a\} \cup S_\pi(a,b)}) = (-1)^{|S_\pi(a,b)|} \C_{b,a} \det(\C_{\pi(S_\pi(a,b)) ,S_\pi(a,b)}) 
= (-1)^{|S_\pi(a,b)|} \C_{b,a}
\]
If $a<\pi^{-1}(b)$ or $\pi(a) < b$, then Equation \eqref{eq: cotame} implies that the left-hand-side is $0$, and so $\C_{b,a}=0$.
If $\pi(a)=b$, then $S_\pi(a,b)\cup \{a\} = S_\pi(a-1,b+1)$ and Equation \eqref{eq: cofrieze} implies that the left-hand-side is $1$; and so $\C_{b,a}= (-1)^{|S_\pi(a,b)|}$.
Therefore, Equation \eqref{eq: prefrieze} holds for this $a,b$.
Taking the inductive limit over all $k$ implies that Equation \eqref{eq: prefrieze} holds in general, and so $\C$ is a $\pi$-prefrieze.
%
%
\end{proof}

\begin{rem}
\label{rem: optimize}
Condition (1) in Lemma \ref{lemma: cofrieze} is redundant as stated, because many pairs $(a,b)$ may yield the same set $S_\pi(a,b)$ and therefore the same Equation \eqref{eq: cofrieze}. One may restrict to a irredundant family by imposing Equation \eqref{eq: cofrieze} for $a,b$ such that $a+1 \leq \pi^{-1}(b-1)$ and $\pi(a+1) \leq b-1$. 
Reparametrizing $a+1\mapsto a$ and $b-1\mapsto b$,
this condition becomes: $\det (\C_{\pi(S),S}) = 1$ for all $a,b$ with $a\leq \pi^{-1}(b)$ and $\pi(a)\leq b$,
where $S= \{ i\in \mathbb{Z} \mid a\leq i \text{ and } \pi(i) \leq b\} = S_\pi(a-1,b+1)$.
\end{rem}

We would like a characterization of when $\C^\dagger$ is a $\pi^\dagger$-prefrieze. To do this, we need a relation between minors of $\C$ and $\C^\dagger$.

\begin{lemma}\label{lemma: minordual}
Given subsets $I,J$ of an interval  $[a,b]$ with $b-a<n$, 
\[\det(\C^\dagger_{I,J}) = \det(\C_{[a,b]\smallsetminus J,[a,b]\smallsetminus I}) \]
\end{lemma}

\begin{proof}
A lower unipotent matrix $\C$ has a well-defined adjugate matrix $\mathsf{Adj}(\C)$, defined by
\[ \mathrm{Adj}(\C)_{a,b} := (-1)^{a+b}\det(\C_{[a+1,b],[a,b-1]})\]
This adjugate is lower unitriangular and enjoys many nice properties, including 
\[ \det(\mathrm{Adj}(\C)_{I,J}) = (-1)^{\sum I + \sum J} \det(\C_{[a,b]\smallsetminus J,[a,b] \smallsetminus I}) \]
Here, $\sum I$ and $\sum J$ denote the sum of the elements of $I$ and $J$, respectively.

For $b\leq a < b+n$, the dual coincides with the adjugate up to sign.
\[ \C^\dagger_{a,b} = \det(\C_{[a+1,b],[a,b-1]}) = (-1)^{a+b} \mathrm{Adj}(\mathsf{C})_{a,b} \]
Therefore, 
\[ \det(\C^\dagger_{I,J}) = (-1)^{\sum I + \sum J} \det(\mathrm{Adj}(\C)_{I,J})= \det(\C_{[a,b]\smallsetminus J,[a,b]\smallsetminus I}) \qedhere\]
\end{proof}


\begin{lemma}\label{lemma: juggleduality}
A $\pi$-prefrieze $\C$ is a $\pi$-frieze iff $\C^\dagger$ is a $\pi^\dagger$-prefrieze. 
\end{lemma}

\begin{proof}
Let $\C$ be a $\pi$-prefrieze. By Lemma \ref{lemma: cofrieze} and Remark \ref{rem: optimize}, $\C^\dagger$ is $\pi^\dagger$-prefrieze iff 
\begin{itemize}
    \item For all $a,b$ with $a\leq \pi^{\dagger-1}(b)$ and $\pi^\dagger(a) \leq b$, and $b<a+n$,
    \[ \det(\C^\dagger_{{\pi^\dagger}(I),I})
    = 1\]
    where $I:= \{ i \in \mathbb{Z} \mid a \leq i \text{ and } \pi^\dagger(i) \leq b\}=S_{\pi^\dagger}(a-1,b+1)$.
    \item For all $a,b$ with $a<\pi^{\dagger-1}(b)$ or $\pi^\dagger(a)< b$, and $b<a+n$,
    \[
    \det(\C^\dagger_{\pi^\dagger(J)\cup\{b\},\{a\}\cup J }) 
    = 0 
    \]
    where $J:= \{ i \in \mathbb{Z} \mid a < i \text{ and } \pi^\dagger(i) < b\}=S_{\pi^\dagger}(a,b)$.
\end{itemize}
Using Lemma \ref{lemma: minordual} to translate this into $\C$, $\C^\dagger$ is a $\pi^\dagger$-prefrieze iff
\begin{itemize}
    \item For all $a,b$ with $a\leq \pi^{\dagger-1}(b)$ and $\pi^\dagger(a) \leq b$, and $b<a+n$,
    \[ \det(\C_{[a,b] \smallsetminus I,[a,b] \smallsetminus {\pi^\dagger}(I)}) 
    = 1\]
    where $I:= \{ i \in \mathbb{Z} \mid a \leq i \text{ and } \pi^\dagger(i) \leq b\}=S_{\pi^\dagger}(a-1,b+1)$.
    \item For all $a,b$ with $a<\pi^{\dagger-1}(b)$ or $\pi^\dagger(a)< b$, and $b<a+n$,
    \[
    \det(\C_{(a,b]\smallsetminus J, [a,b) \smallsetminus \pi^\dagger(J)}) 
    = 0 
    \]
    where $J:= \{ i \in \mathbb{Z} \mid a < i \text{ and } \pi^\dagger(i) < b\}=S_{\pi^\dagger}(a,b)$.
\end{itemize}
These are precisely the frieze and tameness conditions in Definition \ref{defn: pifrieze}.
\end{proof}

As a consequence of Lemma \ref{lemma: juggleduality}, we can take the dual of the dual of a $\pi$-frieze.

\begin{lemma}\label{lemma: doubledual}
If $\C$ is a $\pi$-frieze, then $(\C^\dagger)^\dagger=\C$.
\end{lemma}

\begin{proof}
If $a\geq b > a-n$, then $(\C^\dagger)^\dagger_{a,b} :=\det(\C^\dagger_{[a+1,b],[a,b-1]})$, which equals  $\C_{a,b}$ by 
Lemma \ref{lemma: minordual}.

Consider $a$ with $\pi(a)=a+n$ (a loop of $\pi$), so that $\pi^\dagger(a)=a$ (a coloop of $\pi^\dagger$). Then 
\[ S_{\pi^\dagger}(a,a+n) = \{ i \in \mathbb{Z} \mid a < i \text{ and } \pi^\dagger(i) <a+n\} = \{ i \in \mathbb{Z} \mid a < i \text{ and } \pi^{-1}(i) < a \} 
=L_a\smallsetminus \{a\}
\]
In particular, the cardinality is one less than the number of balls in $\pi$, and so 
\[ (\C^\dagger)^\dagger_{a+n,a} := (-1)^{\text{(\# of balls in $\pi$)} -1} = (-1)^{|S_{\pi^\dagger}(a,a+n)|} =: \C_{a+n,a} \]
All other entries in $(\C^\dagger)^\dagger$ and $\C$ are zero by definition, and so $(\C^\dagger)^\dagger=\C$.
\end{proof}

\begin{thmA}
\label{thm: juggleduality}
If $\C$ is a $\pi$-frieze, then the dual $\C^\dagger$ is a $\pi^\dagger$-frieze, and $(\C^\dagger)^\dagger=\C$.
\end{thmA}

\begin{proof}
If $\C$ is a $\pi$-frieze, then $\C^\dagger$ is a $\pi^\dagger$-prefrieze by Lemma \ref{lemma: juggleduality}. Taking the dual again, $(\C^\dagger)^\dagger=\C$ by Lemma \ref{lemma: doubledual}.
Then $(\C^\dagger)^\dagger$ is a $(\pi^\dagger)^\dagger=\pi$-prefrieze, so Lemma \ref{lemma: juggleduality} implies that $\C^\dagger$ is a $\pi^\dagger$-frieze.
\end{proof}

\subsection{Superperiodicity of solutions (Theorems \ref{thm: juggleqp}, \ref{thm: periodicity}, and \ref{thm: qpsols})}

\def\sol{\mathrm{Sol}}

This theorem follows from a number of results about linear recurrences proven in \cite{MulRes1}. The following lemma translates these results into the context and language of this paper.

\begin{lemma}
\label{lemma: linrec}
Let $\C$ be a $\pi$-prefrieze. 
\begin{enumerate}
    \item
    \label{lemma: linrec1}
    If $\C'$ is another $\pi$-prefrieze with $\ker(\C)=\ker(\C')$, then $\C=\C'$.
    \item There is a unique matrix $\mathbb{Z}\times\mathbb{Z}$-matrix $\sol(\C)$, called the \emph{solution matrix} of $\C$, such that 
    \begin{enumerate}
        \item $\C\sol(\C)=0$,
        \item $\sol(\C)_{a,a}=1$ whenever $\pi(a)\neq a$, and
        \item $\sol(\C)_{a,b}=0$ whenever $\pi(a)=b=a$ or $a<b<\pi(a)$.
    \end{enumerate}
    \item $\sol(\C)$ is also the unique $\mathbb{Z}\times \mathbb{Z}$-matrix such that
    \begin{enumerate}
        \item $\sol(\C)\C=0$,
        \item $\sol(\C)_{a,a}=1$ whenever $\pi(a)\neq a$, and
        \item $\sol(\C)_{a,b}=0$ whenever $\pi(b)=a=b$ or $\pi^{-1}(b)<a<b$.
    \end{enumerate}
    \item The columns of $\sol(\C)$ span the space of solutions to $\C x=0$, whose dimension equals the number of balls in $\pi$.
    \item For all $b<a$, $\sol(\C)_{a,b} = (-1)^{a+b} \det(\C_{[a+1,b],[a,b-1]})$. 
\end{enumerate}
\end{lemma}
\noindent Combining (4) with the definition of $\C^\dagger$, we see that $\sol(\C)_{a,b} = (-1)^{a+b} \C^\dagger_{a,b}$ when $b\leq a<\pi(b)$.

\begin{proof}
A $\pi$-prefrieze is a \emph{reduced recurrence matrix of shape $\pi^{-1}$} in the language of \cite{MulRes1}. 
\begin{enumerate}
    \item is \cite[Theorem 6.5]{MulRes1}; specifically, the uniqueness of the reduced recurrence matrix in a given equivalence class.
    \item is \cite[Theorem 8.3]{MulRes1}.
    \item is \cite[Theorem 10.5]{MulRes1}, since $\pi^{-1}$ is bijective.
    \item is \cite[Theorems 8.12 and 7.14]{MulRes1}. Note that the solution space was shown to be the closure of the span of the columns of $\sol(\C)$ in the product topology. However, since $\pi$ has finitely many balls, the solution space is finite dimensional and therefore closed.
    \item is \cite[Proposition 8.6]{MulRes1}.
    \qedhere
\end{enumerate}
%
%
\end{proof}

Before the proof of Theorem \ref{thm: juggleqp}, the following formula will ensure the signs work out.


\begin{lemma}\label{lemma: balls}
Let $\pi$ be a juggling function of period $n$. Then, for all $a$,
\[  \pi (a) -a + |S_{\pi^\dagger}(\pi(a),a+n)| - |S_\pi(a,\pi(a))| = \text{(\# of balls in $\pi$)} \]
\end{lemma}
\begin{proof}
First, we show that the value of the expression 
\begin{equation}\label{eq: conserved}
 \pi (a) - a + |S_{\pi^\dagger}(\pi(a),a+n)| - |S_\pi(a,\pi(a))| 
\end{equation}
is independent of $a$.

Consider the case that $\pi(a)<\pi(a+1)$. Then $S_\pi(a,\pi(a)) \subseteq S_\pi(a+1,\pi(a+1)) $ and
\begin{align*}
S_\pi(a+1,\pi(a+1)) \smallsetminus S_\pi(a,\pi(a)) 
&= \{ i\in \mathbb{Z} \mid a+1<i \text{ and } \pi(a) < \pi(i) < \pi(a+1) \} 
\end{align*}
Similarly, $S_{\pi^\dagger}(\pi(a),a+n)) \supseteq S_{\pi^\dagger}(\pi(a+1), a+n+1 )) $ and
\begin{align*}
S_{\pi^\dagger}(\pi(a),a+n)) \smallsetminus S_{\pi^\dagger}(\pi(a+1), a+n+1 )) 
&= \{ i\in \mathbb{Z} \mid \pi(a)<i < \pi(a+1) \text{ and } \pi^\dagger(i) < a+n  \} 
\end{align*}
Acting on this set by $\pi^{-1}$ gives
\begin{align*}
\pi^{-1}(S_{\pi^\dagger}(\pi(a),a+n)) &\smallsetminus S_{\pi^\dagger}(\pi(a+1), a+n+1 )) ) \\
&= \{ i\in \mathbb{Z} \mid \pi(a)<\pi(i) < \pi(a+1) \text{ and } i+n < a+n  \} 
\end{align*}
Comparing sets, we see that the interval
$(\pi(a),\pi(a+1) ) $ decomposes as a disjoint union
\[ \left( S_\pi(a+1,\pi(a+1)) \smallsetminus S_\pi(a,\pi(a)) \right) \sqcup \left( \pi^{-1}(S_{\pi^\dagger}(\pi(a),a+n)) \smallsetminus S_{\pi^\dagger}(\pi(a+1), a+n+1 ))) \right) \]
Counting elements on either side, $\pi(a+1) - \pi(a) -1$ equals
\begin{align*}
|S_\pi(a+1,\pi(a+1))| - |S_\pi(a,\pi(a)) | + |S_{\pi^\dagger}(\pi(a),a+n)| - |S_{\pi^\dagger}(\pi(a+1), a+n+1 ) |
\end{align*}
Therefore, the value of Expression \eqref{eq: conserved} is the same for $a$ and $a+1$. An analogous argument applies when $\pi(a) > \pi(a+1)$, and so \eqref{eq: conserved} is constant for all $a$.

Let $h$ denote the value of \eqref{eq: conserved}. Averaging \eqref{eq: conserved} over $a$ from $1$ to $n$,
\begin{align*}
h 
&= \frac{1}{n}\sum_{a=1}^n \pi (a) - a + |S_{\pi^\dagger}(\pi(a),a+n)| - |S_\pi(a,\pi(a))| \\
&= \text{(\# of balls in $\pi$)} + \frac{1}{n} \left( \sum_{a=1}^n  |S_{\pi^\dagger}(\pi(a),a+n)| - \sum_{a=1}^n |S_\pi(a,\pi(a))| \right)
\end{align*}
Observe that $\sum_{a=1}^n |S_\pi(a,\pi(a))|$ counts orbits of \emph{inversions} in $\pi$; that is, pairs $(a,b)$ with $a<b$ and $\pi(a)>\pi(b)$, up to the relation $(a,b) \sim (a+n,b+n)$. Similarly, $\sum_{a=1}^n  |S_{\pi^\dagger}(\pi(a),a+n)| $ counts orbits of {inversions} in $\pi^\dagger$. To see the two sums cancel, observe that $(a,b)$ is an inversion in $\pi$ iff $(\pi(b),\pi(a))$ is an inversion in $\pi^\dagger$.
It follows that $h$ equals the number of balls in $\pi$.
\end{proof}


\begin{thmA}
\label{thm: juggleqp}
Let $\pi$ be an $n$-periodic juggling function with $h$-many balls.
A $\pi$-prefrieze $\C$ is a $\pi$-frieze iff every solution to the linear recurrence $\C \mathsf{x=0}$ 
satisfies $x_{a+n} = (-1)^{n-h-1}x_a$, $\forall a\in \mathbb{Z}$.
\end{thmA}

\begin{proof}
For simplicity, let $k:=\text{(\# of balls in $\pi^\dagger$)} = n-\text{(\# of balls in $\pi$)}$.
Let $\Sigma$ denote the permutation matrix with $\Sigma_{a,a+1}=1$ and all other entries $0$. Then an infinite sequence $x$ (regarded as a vector of height $\mathbb{Z}$) satisfies 
\begin{equation}\label{eq: qp}
x_{a+n} = (-1)^{k-1}x_a
\end{equation}
for all $a$ iff $\Sigma^n x = (-1)^{k-1}x$. 

Let $\C$ be a $\pi$-prefrieze and let $\sol(\C)$ be the solution matrix of $\C$. By Lemma \ref{lemma: linrec}.4, the columns of $\sol(\C)$ span of the solutions to $\C x=0$, and so every solution to $\C x=0$ satisfies Equation \eqref{eq: qp} iff every column of $\sol(\C)$ does. 
Therefore, by Theorem \ref{thm: juggleduality}, to prove Theorem \ref{thm: juggleqp} it will suffice to show that $\C^\dagger$ is a $\pi^\dagger$-prefrieze iff $\Sigma^n \sol(\C) = (-1)^{k-1} \sol(\C)$, which we will show via a lemma.

\begin{lemma}\label{lemma: sol}
$\Sigma^n \sol(\C) = (-1)^{k-1} \sol(\C)$ iff the following conditions hold.
\begin{enumerate}
    \item $\sol(\C)_{a+n,b}=0$ whenever $a<b<\pi(a)$,
    \item $\sol(\C)_{a+n,b}=0$ whenever $\pi^{-1}(b)<a<b$, and
    \item $\sol(\C)_{a+n,a}=(-1)^{k-1}$ whenever $\pi(a)\neq a$.
\end{enumerate}
\end{lemma}

\begin{proof}
Observe that, by Lemma \ref{lemma: linrec}.3,\footnote{Associating products of infinite matrices is not always possible. It is allowed here because $(-1)^{k-1}\Sigma^n$ is a generalized permutation matrix.}
\[ \left( (-1)^{k-1}\Sigma^n\sol(\C)\right) \C = (-1)^{k-1} \Sigma^n \left( \sol(\C) \C \right) = 0 \]
Therefore, $ (-1)^{k-1}\Sigma^n\sol(\C)$ satisfies Condition (a) in Lemma \ref{lemma: linrec}.3. If $\pi(b)=b$, then the $b$th column of $\sol(\C)$ consists entirely of zeros, by \cite[Lemma 8.9]{MulRes1} (this also follows from $\sol(\C)\C=0$), and so $\sol(\C)_{a+n,b}=0$ whenever $\pi(b)=a=b$.

Using that $\left( (-1)^{k-1}\Sigma^n\sol(\C)\right)_{a,b} = (-1)^{k-1} \sol(\C)_{a+n,b}$, the remaining conditions of Lemma \ref{lemma: linrec}.3 imply that 
$\left( (-1)^{k-1}\Sigma^n\sol(\C)\right)=\sol(\C)$ if and only if Conditions (2) and (3) above hold. Once we know $\left( (-1)^{k-1}\Sigma^n\sol(\C)\right)=\sol(\C)$, then Condition (1) follows from Lemma \ref{lemma: linrec}.2. 
\end{proof}

Consider $b<a+n<b+n$. Then 
\begin{align*}
\sol(\C)_{a+n,b} 
&= (-1)^{a+b+n} \det(\C_{[a+n+1,b],[a+n,b-1]}) 
= (-1)^{a+b+n} \C^\dagger_{a+n,b}
\end{align*}
Then Conditions (1) and (2) in Lemma \ref{lemma: sol} may be rewritten as
\begin{enumerate}
    \item $\C^\dagger_{a,b}=0$ whenever $a<b+n<\pi^{\dagger-1}(a)$, and
    \item $\C^\dagger_{a,b}=0$ whenever $\pi^{\dagger}(b)<a<b+n$.
\end{enumerate}
These equalities coincide with the vanishing part of the $\pi^\dagger$-prefrieze condition for $\C^\dagger$ (the second case of Equation \eqref{eq: prefrieze}).

For $a$ with $\pi(a)\neq a$, consider the $(a+n,a)$th entry in the equation $\sol(\C)\C=0$, which is
\[\sum_{b\in \mathbb{Z}} \sol(\C)_{a+n,b}\C_{b,a} =0 \]
Since $\C$ is a $\pi$-prefrieze, $\C_{b,a}=0$ whenever $b\not\in [a,\pi(a)]$. 
Assuming Condition (1) in Lemma \ref{lemma: sol}, the above sum reduces to
\[ \sol(\C)_{a+n,a} + (-1)^{a+\pi(a)+n}\C^\dagger_{a+n,\pi(a)}\C_{\pi(a),a}=0 \]
By the $\pi$-prefrieze condition, $\C_{\pi(a),a}=(-1)^{|S_\pi(a,\pi(a))|}$ and so
\[ \sol(\C)_{a+n,a} = (-1)^{a+\pi(a)+n-1 + |S_\pi(a,\pi(a))| } \C^\dagger_{\pi^\dagger(\pi(a)),\pi(a)} \]
Using Lemma \ref{lemma: balls},
\[ \sol(\C)_{a+n,a} = (-1)^{|S_{\pi^\dagger}(\pi(a),a+n)| + k-1} \C^\dagger_{\pi^\dagger(\pi(a)),\pi(a)} \]
Therefore, $ \sol(\C)_{a+n,a}= (-1)^{k-1}$ iff $\C^\dagger_{\pi^\dagger(\pi(a)),\pi(a)} = (-1)^{|S_{\pi^\dagger}(\pi(a),a+n)|}$, and so Condition (3) in Lemma \ref{lemma: sol} is equivalent to the non-vanishing part of the $\pi^\dagger$-prefrieze condition for $\C^\dagger$ (the first case of Equation \eqref{eq: prefrieze}).
Therefore, $\C^\dagger$ is a $\pi^\dagger$-prefrieze iff the conditions in Lemma \ref{lemma: sol} hold; equivalently, if every solution to $\C x = 0$ satisfies Equation \eqref{eq: qp}.
%
%
\end{proof}

\begin{thmA}
\label{thm: periodicity}
If $\pi$ is an $n$-periodic juggling function, then the entries of a $\pi$-frieze $\C$ are $n$-periodic; that is, $\C_{a,b} = \C_{a+n,b+n}$.
\end{thmA}

\begin{proof}
Let $\C$ be a $\pi$-frieze, and let $\C'$ denote the $\pi$-frieze defined by translating the entries by $n$:
\[ (\C')_{a,b} := \C_{a+n,b+n} \]
Let $\mathsf{v}\in \ker(\C)$. By Theorem \ref{thm: juggleqp}, $v_{a+n} = (-1)^{n-h-1}x_a$ for all $a$. 
\[ (\C'\mathsf{v})_a 
= \sum_{b\in \mathbb{Z}} (\C')_{a,b} \mathsf{v}_b 
= \sum_{b\in \mathbb{Z}} \C_{a+n,b+n} \left((-1)^{n-h-1}\mathsf{v}_{b+n} \right)
= (-1)^{n-h-1}(\C\mathsf{v})_{a+n}
\]
Since $\mathsf{v}\in \ker(\C)$, this is always $0$, and so $\mathsf{v}\in \ker(\C')$. A symmetric argument shows that $\ker(\C)=\ker(\C')$. By Lemma \ref{lemma: linrec}.1, $\C=\C'$.
\end{proof}

As a consequence, the solution matrix of a $\pi$-frieze can constructed by `tiling' the dual $\C^\dagger$, as follows.
%
%
Given an $n$-periodic juggling function $\pi$ with $h$-many balls, 
define the \textbf{tiling} of a $\pi$-frieze $\C$ to be the $\mathbb{Z}\times\mathbb{Z}$-matrix $\mathsf{Tiling}(\C)$ defined by
\[ \mathsf{Tiling}(\C)_{a,b} = \sum_{i\in \mathbb{Z}} (-1)^{a+b+i(n-h-1)}\C_{a+in,b}\]
This first modifies the sign of each $\C_{a,b}$ by $(-1)^{a+b}$, and then extends the columns superperiodically.

\begin{rem}
\cite{MGOST14} and \cite{MG15} define tilings of $\mathrm{SL}(k)$-friezes without modifying the signs by $(-1)^{a+b}$ first. This is because the linear recurrence they associate to a frieze is not $\C\mathsf{x=0}$, but instead use an alternating sign convention.
\end{rem}

\begin{prop}
\label{prop: soltiling}
If $\C$ is a juggler's frieze, then $\sol(\C)= \mathrm{Tiling}(\C^\dagger)$.
\end{prop}

\begin{proof}
First, assume $b\leq a<b+n$. If $b\neq \pi(b)$, then by Lemma \ref{lemma: linrec}.5,
\[ \sol(\C)_{a,b} = (-1)^{a+b}\C^\dagger_{a,b} = \mathrm{Tiling}(\C^\dagger)_{a,b} \]
If $b=\pi(b)$, then $\sol(\C)_{a,b}=0$. If $a>b$, then
\[ \mathrm{Tiling}(\C^\dagger)_{a,b} = \C^\dagger_{a,b} = 0\]
If $a=b=\pi(b)$,  then 
\[ \mathrm{Tiling}(\C^\dagger)_{a,b} = \C^\dagger_{a,b} - \C^\dagger_{a+n,b} = 1 -1 = 0 \]
In both cases, $\sol(\C)_{a,b} = \mathrm{Tiling}(\C^\dagger)_{a,b}$.

Next, we check that the columns of $\mathrm{Tiling}(\C^\dagger)_{a,b}$ satisfy Equation \eqref{eq: qp}:\footnote{Note that the number of balls in $\pi^\dagger$ is $k:=n-h$, not $h$.}
\begin{align*}
\mathrm{Tiling}(\C^\dagger)_{a+n,b} 
&=  \sum_{i\in \mathbb{Z}} (-1)^{a+n+b+i(n-k-1)}\C^\dagger_{a+n+in,b} \\
&=  \sum_{i\in \mathbb{Z}} (-1)^{a+b+i(n-k-1)-(k+1)}\C^\dagger_{a+in,b}
= (-1)^{k-1}\mathrm{Tiling}(\C^\dagger)_{a,b} 
\end{align*}
Since the columns of $\sol(\C)$ also satisfy this equation, the two matrices coincide everywhere.
\end{proof}

\begin{thmA}
\label{thm: qpsols}
Let $\pi$ be an $n$-periodic jugging function with $h$-many balls, and let $\C$ be a $\pi$-frieze. For any $b$ with $b<\pi(b)$, the sequence $\mathsf{x}$ defined by 
\begin{itemize}
    \item $x_a := (-1)^{a+b} \C^\dagger_{a,b} $ when $b\leq a< b+k+h$, and 
    \item $x_{a+n} = (-1)^{n-h-1}x_a$ for all $a$
\end{itemize}
is a solution to $\C \mathsf{x=0}$. These solutions collectively span the space of solutions to $\C \mathsf{x=0}$.
\end{thmA}

\begin{proof}
By Lemma \ref{lemma: linrec}.4, the columns of $\sol(\C)$ span the space of solutions to $\C\mathsf{x=0}$. When $b=\pi(b)$, the $b$th column of $\sol(\C)$ is identically zero, so it suffices to consider the columns for which $b<\pi(b)$. By Proposition \ref{prop: soltiling}, the $b$th column of $\sol(\C)$ coincides with the $b$th column of $\mathrm{Tiling}(\C^\dagger)$, which is the sequence defined in the statement of Theorem \ref{thm: qpsols}.
\end{proof}

\begin{rem}\label{rem: schedule}
By \cite[Prop.~7.12]{MulRes1}, a basis of solutions to $\C\mathsf{x=0}$ is given by the columns of $\sol(\C)$ index by any \emph{schedule} of $\pi$, such as a landing schedule $L_a$. 
\end{rem}

\subsection{Equivalence of constructions (Theorem \ref{thm: equivcons})} 

The relation between the determinantal construction of $\Fr(\mathsf{A})$ and Construction \ref{cons: piunwrap} follows primarily from the following lemma.

\begin{lemma}
\label{lemma: taudotproduct}
Let $\mathsf{A}$ be a $\pi$-unimodular $k\times n$-matrix. For any $1\leq a,b \leq n$, 
\[ \rt(\mathsf{A})_{{a}} \cdot \mathsf{A}_{{b}} 
=\begin{cases}
(-1)^{|S_{\pi^\dagger}(b,a)|}\det(\mathsf{A}_{L_a\smallsetminus \{a\} \cup \{b\} }) & \text{if }{a} \geq {b} \text{ and }a<\pi(a),\\
(-1)^{k-1+|S_{\pi^\dagger}(b,a+n)|}\det(\mathsf{A}_{L_a\smallsetminus \{a\} \cup \{b\} }) & \text{if }{a} < {b} \text{ and }a<\pi(a),\\
0 & \text{if }a=\pi(a).
\end{cases}
\]
\end{lemma}

\begin{proof}
If $a=\pi(a)$ (i.e.~$a$ is a loop of $\pi$), then $a\not\in L_a$ and so the ${a}$th column of $\rt(\mathsf{A})$ is defined to be the unique solution to $\left({\mathsf{A}}_{{L_a}}\right)^\top\mathsf{x}=\mathsf{0}$. Therefore, $\rt(\mathsf{A})_a=0$, and so $\rt(\mathsf{A})_a\cdot \mathsf{A}_b=0$ for all $b$.

If $a<\pi(a)$, then $a\in L_a$ and so the ${a}$th column of $\rt(\mathsf{A})$ is defined to be the unique solution to $\left({\mathsf{A}}_{{L_a}}\right)^\top\mathsf{x}=\mathsf{e}_{{a}}$. Therefore,  $\rt(\mathsf{A})_{{a}} = \left({\mathsf{A}}_{{L_a}}\right)^{\top-1} \mathsf{e}_{{a}}$, and so
\[ \rt(\mathsf{A})_{{a}} \cdot \mathsf{A}_{{b}}  
= (\left({\mathsf{A}}_{{L_a}}\right)^{\top-1} \mathsf{e}_{{a}}) \cdot \mathsf{A}_{{b}} 
= \mathsf{e}_{{a}} \cdot \left({\mathsf{A}}_{{L_a}}\right)^{-1}\mathsf{A}_{{b}}
\]
Equivalently, $\rt(\mathsf{A})_{{a}} \cdot \mathsf{A}_{{b}}$ is the ${{a}}$th entry of the solution to ${\mathsf{A}}_{{L_a}}\mathsf{x} = \mathsf{A}_{{b}}$. By Cramer's Rule, $\rt(\mathsf{A})_{{a}} \cdot \mathsf{A}_{{b}}$ equals the determinant of ${\mathsf{A}}_{{L_a}}$ after replacing the ${a}$th column of $\mathsf{A}$ by the ${b}$th column of $\mathsf{A}$.\footnote{There is also a denominator of $\det(\mathsf{A}_{L_a})$ in Cramer's Rule, but this determinant is 1 by $\pi$-unimodularity.} By our convention for ordering the columns\footnote{Specifically, the columns in $\mathsf{A}_{L_a}$ retain the same ordering as in $\mathsf{A}$.} in $\mathsf{A}_{L_a\smallsetminus \{a\} \cup \{b\}}$, 
\[ \rt(\mathsf{A})_{{a}} \cdot \mathsf{A}_{{b}} = (-1)^{s} \det(\mathsf{A}_{L_a\smallsetminus \{a\} \cup \{b\} })\]
where $s$ is the number of columns of $\mathsf{A}_{L_a}$ between the ${a}$th column and the ${b}$th column.

If $b\in \overline{L_a}$ (i.e.~$b$ is congruent to an element of $L_a$), then $\det(\mathsf{A}_{L_a\smallsetminus \{a\} \cup \{b\} })=0$ and so $s$ is irrelevant.

If ${b}\leq {a}$ and $b\not \in \overline{L_a}$, then $s$ is the number of columns of $A_{L_a}$ right of the ${b}$th column and left of the ${a}$th column. Lifting to $\mathbb{Z}$, this is the cardinality of the set $L_{a-n}\cap [b+1,a-1]$.
\begin{align*}
s &= | L_{a-n} \cap [b+1,a-1] |
= | \{ i\in [b+1,a-1] \mid \pi^{-1}(i)<a-n\} | \\
&= | \{ i\in [b+1,a-1] \mid \pi^{\dagger}(i)<a\} |
= | [b+1,a-1] \cap \pi^\dagger([b+1,a-1]) |
=: |S_{\pi^\dagger}(b,a)|
\end{align*}

If ${b}>{a}$ and $b\not \in \overline{L_a}$, then $s$ is the number of columns of $A_{L_a}$ right of the ${a}$th column and left of the ${b}$th column. Lifting to $\mathbb{Z}$, this is the cardinality of the set $L_{a}\cap [a+1,b-1]$. Since $\{a\}\in L_a$ and $\{b\}\not\in L_a$, $L_a$ decomposes as a disjoint union
\[ L_a = \{ a\} \sqcup (L_a \cap [a+1,b-1]) \sqcup (L_a \cap [b+1,a+n-1]) \]
The cardinality of $L_a$ is the number of balls of $\pi$, which is $k$. By the same argument as the previous case, the cardinality of $L_a \cap [b+1,a+n-1]$ is $|S_{\pi^\dagger}(b,a+n)|$. Therefore, in this case,
\[ s = k - 1 - |S_{\pi^\dagger}(b,a+n)| \]
This completes the proof.
\end{proof}

This allows us to restate the definition of $\Fr(\mathsf{A})$ as follows.

\begin{coro}\label{coro: twistcons} 
Let $\Fr(\mathsf{A})$ be defined as in Theorem \ref{thm: detpifrieze}. Then
\[ \Fr(\mathsf{A})_{a,b} := \begin{cases}
\rt(\mathsf{A})_{\overline{a}} \cdot \mathsf{A}_{\overline{b}} & \text{if $b\leq a < b+n$, $\overline{b}\leq \overline{a}$, and $a<\pi(a)$}, \\
(-1)^{k-1}\rt(\mathsf{A})_{\overline{a}} \cdot \mathsf{A}_{\overline{b}} & \text{if $b\leq a < b+n$, $\overline{b}> \overline{a}$, and $a<\pi(a)$}, \\
1 & \text{if }\pi(a)=a=b, \\
(-1)^{k} & \text{if }\pi(a)=a=b+n, \\
0 & \text{otherwise}. 
\end{cases} \]
where $1\leq \overline{a},\overline{b} \leq n$ denote the respective residues of $a,b$ mod $n$.
\end{coro}

\begin{thmA}
\label{thm: equivcons}
If  $\pi$ is loop-free and $\mathsf{A}$ is a $\pi$-unimodular matrix, then Construction \ref{cons: piunwrap} produces $\Fr(\mathsf{A})$, the $\pi^\dagger$-frieze defined in Theorem \ref{thm: detpifrieze}.
\end{thmA}

\begin{proof}[Proof of Theorem \ref{thm: equivcons}]
If $\pi$ is loop-free, then the corollary specializes to 
\[ \Fr(\mathsf{A})_{a,b} :=\begin{cases}
\rt(\mathsf{A})_{\overline{a}} \cdot \mathsf{A}_{\overline{b}}  & \text{if $b\leq a < b+n$ and $\overline{b}\leq \overline{a}$}, \\
(-1)^{k-1} \rt(\mathsf{A})_{\overline{a}} \cdot \mathsf{A}_{\overline{b}} & \text{if $b\leq a < b+n$ and $\overline{b}> \overline{a}$}, \\
0 & \text{otherwise}.
\end{cases} \]
Since $(\rt(\mathsf{A})^\top \mathsf{A}) _{\overline{a},\overline{b}} = \rt(\mathsf{A})_{\overline{a}} \cdot \mathsf{A}_{\overline{b}} $, Construction \ref{cons: piunwrap} produces the above entries.
\end{proof}


\begin{rem}\label{rem: loopcons}
Corollary \ref{coro: twistcons} indicates how to adapt Construction \ref{cons: piunwrap} to the case when $\pi$ has loops. Specifically, if $\pi(a)=a$, then $(\rt(\mathsf{A})^\top\mathsf{A})_{a,a}=0$. In Step (3) of Construction \ref{cons: piunwrap}, this $0$ should be split into $1+(-1)$. The $1$ stays in place and the $(-1)$ moves with the above-diagonal entries to become a $(-1)^k$ in position $\Fr(\mathsf{A})_{a+n,a}$.
%
\end{rem}

\subsection{$\Fr(\mathsf{A})$ is a $\pi^\dagger$-frieze (Theorem \ref{thm: detpifrieze})}

We will prove that $\Fr(\mathsf{A})$ is a $\pi^\dagger$-frieze by first showing it is a $\pi^\dagger$-prefrieze, and then showing its kernel consists of superperiodic sequences. 

\begin{lemma}\label{lemma: Fprefrieze}
If $\mathsf{A}$ is $\pi$-unimodular, then $\Fr(\mathsf{A})$ is a $\pi^\dagger$-prefrieze.
\end{lemma}

\begin{proof}
Since $\Fr(\mathsf{A})$ is lower unitriangular, it suffices to show that $\Fr(\mathsf{A})$ satisfies Lemma \ref{lemma: prefrieze} for $\pi^\dagger$.
Given $a\in \mathbb{Z}$, we split into two cases.
\begin{itemize}
    \item If $a=\pi(a)$ (i.e.~$a$ is a loop of $\pi$), then $\pi^\dagger(a)=a+n$ and so
    $ \Fr(\mathsf{A})_{\pi^\dagger(a),a} := (-1)^{k} $.
    Then 
%
    \[ |S_{\pi^\dagger}(a,a+n)|= n - \text{(\# of balls in $\pi^\dagger$)} = \text{(\# of balls in $\pi$)} = k\]
    Therefore, 
    $ \Fr(\mathsf{A})_{\pi^\dagger(a),a} = (-1)^{|S_{\pi^\dagger}(a,\pi^\dagger(a))|} $.
    \item If $a<\pi(a)$, then $ \pi^\dagger(a) < a+n$. Then last inequality is equivalent to $\pi^{-1}(a)<a$, and so $\pi^{-1}(a)\in L_{\pi^{-1}(a)}$. Since $a= \pi(\pi^{-1}(a))$,
    \begin{align*}
    L_{\pi^{-1}(a)+1} 
    &= L_{\pi^{-1}(a)} \smallsetminus \{\pi^{-1}(a)\} \cup \{a\} \\
    \Fr(\mathsf{A})_{\pi^\dagger(a),a} 
    &:= (-1)^{|S_{\pi^\dagger}(a, \pi^\dagger(a))|} \det(\mathsf{A}_{L_{\pi^\dagger(a)}\smallsetminus \{\pi^\dagger(a)\}\cup \{a\}}) \\
    &= (-1)^{|S_{\pi^\dagger}(a, \pi^\dagger(a))|} \det(\mathsf{A}_{L_{\pi^{-1}(a)}\smallsetminus \{\pi^{-1}(a)\}\cup \{a\}}) \\
    &= (-1)^{|S_{\pi^\dagger}(a, \pi^\dagger(a))|} \det(\mathsf{A}_{L_{\pi^{-1}(a)+1}}) = (-1)^{|S_{\pi^\dagger}(a, \pi^\dagger(a))|}
    \end{align*}
    We have used that $\mathsf{A}_I$ only depends on $I$ modulo $n$, and that $\mathsf{A}$ is $\pi$-unimodular.
\end{itemize}
In all cases, we see that $\Fr(\mathsf{A})_{\pi^\dagger(a),a} = (-1)^{|S_{\pi^\dagger}(a, \pi^\dagger(a))|}$, satisfying the first case of Equation \eqref{eq: prefrieze}.

Next, consider $a,b$ such that $a-n< b<\pi^{\dagger-1}(a) $ or $\pi^\dagger(b)<a<b+n$. By Corollary \ref{coro: twistcons},
\[ \Fr(\mathsf{A})_{a,b} = \pm (\rt(\mathsf{A})_a \cdot \mathsf{A}_b)\]
The dot product on the right-hand-side always vanishes, e.g.\ by \cite[Lemma 6.4]{MS17}.\footnote{It is also not too hard to show directly that $\det(\mathsf{A}_{L_a \smallsetminus \{a\}\cup \{b\}})=0$ whenever $a-n< b<\pi^{\dagger-1}(a) $ or $\pi^\dagger(b)<a<b+n$.} Therefore, $\Fr(\mathsf{A})$ satisfies the conditions of Lemma \ref{lemma: prefrieze} for $\pi^\dagger$.
\end{proof}

Given a vector $\mathsf{v}\in \k^n$ and integers $(k,n)$, the \textbf{superperiodic extension} of $\mathsf{v}$ to be the infinite sequence $q_{(k,n)}(\mathsf{v})\in \k^\mathbb{Z}$ defined by
\[ q_{(k,n)}(\mathsf{v})_a := (-1)^{(k-1)\frac{a-\overline{a}}{n}}\mathsf{v}_{\overline{a}} \]
Note that $\mathsf{v}$ can be recovered from $q_{(k,n)}(\mathsf{v})$ by restricting to the terms indexed by $[n]\subset \mathbb{Z}$.

\begin{lemma}
\label{lemma: qpkernel}
%
%
Let $\mathsf{A}$ be a $\pi$-unimodular $k\times n$-matrix.
Then $\mathsf{Av=0}$ iff $\Fr(\mathsf{A})q_{(k,n)}(\mathsf{v})=0$.
%
\end{lemma}
\noindent I.e.~a vector $\mathsf{v}$ is in the kernel of $\mathsf{A}$ iff its superperiodic extension $q_{(k,n)}(\mathsf{v})$ is in the kernel of $\Fr(\mathsf{A})$.


\begin{proof}
Let $\mathsf{v}\in \ker(\mathsf{A})$, and
consider the $a$th entry in the product $\Fr(\mathsf{A}) q_{(k,n)}(\mathsf{v})$, which is
\[ \sum_{b\in \mathbb{Z}} \Fr(\mathsf{A})_{a,b} q_{(k,n)}(\mathsf{v})_b
=\sum_{b\in \mathbb{Z}}  \Fr(\mathsf{A})_{a,b}
(-1)^{(k-1)\frac{b-\overline{b}}{n}} 
\mathsf{v}_{\overline{b}}
\]
If $a=\pi(a)$, then the only non-zero terms occur when $b\in \{a,a-n\}$, in which case
\[ \sum_{b\in \mathbb{Z}} \Fr(\mathsf{A})_{a,b} q_{(k,n)}(\mathsf{v})_b
=
(-1)^{(k-1)\frac{a-\overline{a}}{n}} 
\mathsf{v}_{\overline{a}}
+
(-1)^{(k-1)\frac{a-n-\overline{a}}{n}} 
\mathsf{v}_{\overline{a}}
=(-1)^{(k-1) \frac{a-\overline{a}}{n}}
( \mathsf{v}_{\overline{a}} - \mathsf{v}_{\overline{a}}) = 0
\]
If $a<\pi(a)$, then the only non-zero terms occur when $b\in [a-n+1,a]$. Let $i$ be the unique integer with $a-n<in \leq a$. 
Then 
\begin{align*}
\sum_{b\in \mathbb{Z}} \Fr(\mathsf{A})_{a,b} q_{(k,n)}(\mathsf{v})_b
&=\sum_{b=a-n+1}^a  \Fr(\mathsf{A})_{a,b}
(-1)^{(k-1)\frac{b-\overline{b}}{n}} 
\mathsf{v}_{\overline{b}}\\
&=\sum_{b=a-n+1}^{in}  (\rt(\mathsf{A})_{\overline{a}}\cdot \mathsf{A}_{\overline{b}})
(-1)^{(k-1)i} 
\mathsf{v}_{\overline{b}}
+\sum_{b=in}^{a}  (-1)^{k-1}(\rt(\mathsf{A})_{\overline{a}}\cdot \mathsf{A}_{\overline{b}})
(-1)^{(k-1)(i-1)} 
\mathsf{v}_{\overline{b}} \\
&=(-1)^{(k-1)i} \sum_{b=a-n+1}^{an}  (\rt(\mathsf{A})_{\overline{a}}\cdot \mathsf{A}_{\overline{b}})
\mathsf{v}_{\overline{b}}
=(-1)^{(k-1)i} \sum_{\overline{b} =1}^{n}  (\rt(\mathsf{A})_{\overline{a}}\cdot \mathsf{A}_{\overline{b}})
\mathsf{v}_{\overline{b}} \\
&=  (-1)^{(k-1)i}  ( \rt(\mathsf{A})^\top \mathsf{A} \mathsf{v})_{\overline{a}} 
\end{align*}
Since $\mathsf{Av}=0$, this expression is $0$. Therefore, $\Fr(\mathsf{A}) q_{(k,n)}(\mathsf{v})$ is the infinite sequence of zeroes.

This establishes that the superperiodic extension $q_{(k,n)}(\ker(\mathsf{A}))$ of $\ker(\mathsf{A})$ is contained in $\ker(\Fr(\mathsf{A}))$; we show equality next. Recall that $\Fr(\mathsf{A})$ is a $\pi^\dagger$-prefrieze (by Lemma \ref{lemma: Fprefrieze}), and so the dimension of the kernel of $\Fr(\mathsf{A})$ equals the number of balls in $\pi^\dagger$ (by Lemma \ref{lemma: linrec}.4), which is $n-k$.  Since $\mathsf{A}$ is $k\times n$ with rank $k$, the dimension of $\mathrm{ker}(\mathsf{A})$ is $n-k$. Since $q_{(k,n)}:\k^n\rightarrow \k^\mathbb{Z}$ is an inclusion, $q_{(k,n)}(\ker(\mathsf{A}))$ has the same dimension as $\ker(\Fr(\mathsf{A}))$, and so the two spaces are equal.
\end{proof}

%
\vspace{0.5em}
\begin{thmA}
\label{thm: detpifrieze}
If $\mathsf{A}$ is a $\pi$-unimodular matrix, then $\Fr(\mathsf{A})$ is a $\pi^\dagger$-frieze.
\end{thmA}

\begin{proof}[Proof of Theorem \ref{thm: detpifrieze}]
By Lemma \ref{lemma: Fprefrieze}, $\Fr(\mathsf{A})$ is a $\pi^\dagger$-prefrieze, and by Lemma \ref{lemma: qpkernel}, it has superperiodic kernel. By Theorem \ref{thm: juggleqp}, $\Fr(\mathsf{A})$ is a $\pi^\dagger$-frieze.
%
%
\end{proof}

%


\subsection{Inverting $\Fr$ (Theorem \ref{thm: Frbijection})}

In this section, we show that every juggler's frieze can be constructed using $\Fr$. First, a useful formula relating minors of $\mathsf{A}$ to projections of $\ker(\mathsf{A})$.

\begin{lemma}\label{lemma: rankker}
Let $\mathsf{A}$ be a $k\times n$ matrix of rank $k$. Then for any subset $I\subset [n]$, 
\[ \mathrm{rank}(\mathsf{A}_I) = \dim(\mathrm{ker}(\mathsf{A})_{[n]\smallsetminus I}) + |I| -(n-k) \]
where $\mathrm{ker}(\mathsf{A})_{[n]\smallsetminus I}$ is the image of $\mathrm{ker}(\mathsf{A})$ under the restriction map $\k^n\rightarrow \k^{[n]\smallsetminus I}$.
\end{lemma}

\begin{proof}
Since the restriction map $\k^n\rightarrow \k^{[n]\smallsetminus I}$ is the quotient map by the subspace $\k^I\subset \k^n$, 
\[ \mathrm{ker}(\mathsf{A})_{[n] \smallsetminus I} \simeq \frac{\mathrm{ker}(\mathsf{A} ) } { \mathrm{ker}(\mathsf{A} ) \cap \k^I} \]
Since $\mathrm{ker}(\mathsf{A} ) \cap \k^I\simeq \mathrm{ker}(\mathsf{A}_I)$, applying the Rank-Nullity Theorem gives the following
\[
\dim(\mathrm{ker}(\mathsf{A})_{[n] \smallsetminus I}) 
= \dim( \mathrm{ker}(\mathsf{A} ) ) - \dim ( \mathrm{ker}(\mathsf{A}_I ) ) 
= (n - \overbrace{\mathrm{rank}(\mathsf{A}}^{=k}) ) - ( |I| - \mathrm{rank}(\mathsf{A}_I) ) \qedhere
\]
%
\end{proof}

\begin{lemma}
\label{lemma: Frinv}
Let $\pi$ be an $n$-periodic juggling function with $k$-many balls and let $\C$ be a $\pi^\dagger$-frieze. If $\mathsf{A}$ is a $k\times n$ matrix such that, for all vectors $\mathsf{v}\in\k^n$, 
\[ \mathsf{A} \mathsf{v} = 0 \Leftrightarrow \C q_{(k,n)}(\mathsf{v}) = 0 \]
then there is a $\pi$-unimodular matrix $\mathsf{A}'$ with the same kernel as $\mathsf{A}$ and $\Fr(\mathsf{A}')=\C$.
\end{lemma}

\begin{proof}
Let $[a,b]\subset \mathbb{Z}$ be any interval with $b-a<n$, and let $\overline{[a,b]} \subset [n]$ denote the residues mod $n$. By Lemma \ref{lemma: rankker},
\[ \mathrm{rank}(\mathsf{A}_{\overline{[a,b]}}) = \dim(\ker(\mathsf{A})_{[n] \smallsetminus \overline{[a,b]}}) + (b-a+1) - (n-k) \]
Since $[n] \smallsetminus \overline{[a,b]} = \overline{[b+1,a+n-1]}$, 
\[ \dim(\ker(\mathsf{A})_{[n] \smallsetminus \overline{[a,b]}}) 
= \dim(\ker(\mathsf{A})_{\overline{[b+1,a+n-1]}})
= \dim(\mathrm{ker}(\C)_{[b+1,a+n-1]})
\]
By \cite[Prop.~7.12]{MulRes1}, this is equal to the number of $\pi^\dagger$-balls in the interval $[b+1,a+n-1]$; i.e.
\begin{align*}
\dim(\mathrm{ker}(\C)_{[b+1,a+n-1]}) 
&= |\{ i \in\mathbb{Z} \mid b< i < a+n \text{ and }a+ n \leq \pi^{\dagger}(i)\}| \\
&= |\{ i \in\mathbb{Z} \mid b< i < a+n \text{ and }a \leq \pi^{-1}(i) \} | \\
&= (a+n-b -1) - |\{ i \in\mathbb{Z} \mid b< i < a+n \text{ and }a > \pi^{-1}(i) \} | \\
&= (a+n-b -1) - | L_a \cap [b+1,a+n-1] | \\
&= (a+n-b -1) - (| L_a |  - | L_a \cap [a,b]|)
\end{align*}
Since $|L_a|=k$, the previous formulas yield that $\mathrm{rank}(\mathsf{A}_{\overline{[a,b]}}) = | L_a \cap [a,b]| $; that is, $\mathsf{A}$ satisfies the second condition of being $\pi$-unimodular.


Define a $\mathbb{Z}\times \mathbb{Z}$-matrix $\Fr(\mathsf{A})$ via the formula (identical to \ref{eq: detformula} but without assuming $\pi$-unimodular)
\[ \Fr(\mathsf{A})_{a,b} := \begin{cases}
\rt(\mathsf{A})_{\overline{a}} \cdot \mathsf{A}_{\overline{b}} & \text{if $b\leq a < b+n$, $\overline{b}\leq \overline{a}$, and $a<\pi(a)$}, \\
(-1)^{k-1}\rt(\mathsf{A})_{\overline{a}} \cdot \mathsf{A}_{\overline{b}} & \text{if $b\leq a < b+n$, $\overline{b}> \overline{a}$, and $a<\pi(a)$}, \\
1 & \text{if }\pi(a)=a=b, \\
(-1)^{k} & \text{if }\pi(a)=a=b+n, \\
0 & \text{otherwise}.
\end{cases} \]
If $a> \pi^\dagger (b)$, then either $a\geq b+n$ or $a<b+n$. In the first case, $\Fr(\mathsf{A})_{a,b}$ is defined to be $0$. In the second case, $\pi^{-1}(b)< a-n < b$, 
and so $\overline{b}\in \overline{L_a}\smallsetminus \{\overline{a}\}$. Therefore, $\Fr(\mathsf{A})_{a,b} = \pm(\tau(\mathsf{A})_{\overline{a}}\cdot \mathsf{A}_{\overline{b}}) = 0$.
If $b<(\pi^\dagger)^{-1}(a)$, then  either $a\geq b+n$ or $a<b+n$. In the first case, $\Fr(\mathsf{A})_{a,b}$ is defined to be $0$. In the second case, $a < b+n< \pi(a)$,  and so $\tau(\mathsf{A})_{\overline{a}}\cdot \mathsf{A}_{\overline{b}}=0$ by \cite[Lemma~6.4]{MS17}, and $\Fr(\mathsf{A})_{a,b} =0$.

Therefore, $\Fr(\mathsf{A})$ is a \emph{reduced recurrence matrix} in the sense of \cite{MulRes1}. Since $\C$ is another reduced recurrence matrix with the same kernel, $\Fr(\mathsf{A})=\C$ by \cite[Theorem~6.5]{MulRes1}. Then for any $b\in \mathbb{Z}$,
\[
(-1)^{|S_\dagger(b,\pi^\dagger(b))|}  
= \C_{\pi^\dagger(b),b} 
= \Fr(\mathsf{A})_{\pi^\dagger(b),b}
\]
Assuming $\pi^\dagger(b)<b+n$, then $\pi(\pi^\dagger(b)) > \pi^\dagger(b)$ and $\overline{\pi^\dagger(b)} = \pi^{-1}(\overline{b})$ so 
\[ 
(-1)^{|S_\dagger(b,\pi^\dagger(b))|}  
=
\begin{cases}
\rt(\mathsf{A})_{\pi^\dagger(\overline{b})} \cdot \mathsf{A}_{\overline{b}} & \text{if $\overline{b}\leq \overline{a}$},\\
(-1)^{k-1}\rt(\mathsf{A})_{\pi^\dagger(\overline{b})} \cdot \mathsf{A}_{\overline{b}} & \text{if $\overline{b}> \overline{a}$}.
\end{cases}
\]
By an analogous computation to the proof of Lemma \ref{lemma: taudotproduct},
\[ 
\rt(\mathsf{A})_{\pi^\dagger(\overline{b})} \cdot \mathsf{A}_{\overline{b}}
=
(-1)^{|S_\dagger(b,\pi^\dagger(b))|}  
\begin{cases}
\frac{\det(\mathsf{A}_{L_{b-1}})}{\det(\mathsf{A}_{L_b})} & \text{if $\overline{b}\leq \overline{a}$},\\
(-1)^{k-1} \frac{\det(\mathsf{A}_{L_{b-1}})}{\det(\mathsf{A}_{L_b})} & \text{if $\overline{b}> \overline{a}$}.
\end{cases}
\]
Therefore, $\det(\mathsf{A}_{L_{b-1}})=\det(\mathsf{A}_{L_b})$. In the missing case when $\pi^\dagger(b) = b+n$, $L_{b-1}=L_{b}$ and so $\det(\mathsf{A}_{L_{b-1}})=\det(\mathsf{A}_{L_b})$. It follows that $\det(\mathsf{A}_{L_a})=\det(\mathsf{A}_{L_b})$ for all $a,b\in \mathbb{Z}$. Choose any $k\times k$ matrix $\mathsf{G}$ with $\det(\mathsf{G})= \det(\mathsf{A}_{L_a})^{-1}$ and set $\mathsf{A}':=\mathsf{GA}$. Then $\mathsf{A}'$ is $\pi$-unimodular and $\ker(\mathsf{A}')=\ker(\mathsf{A})$.
%
\end{proof}

\begin{thmA}
\label{thm: Frbijection}
For each juggling function $\pi$ with $k$-many balls,
the map $\Fr$ descends to a bijection
\[ \mathrm{SL}(k)\backslash \{ \text{$\pi$-unimodular matrices}\} \xrightarrow{\Fr} \{\text{$\pi^\dagger$-friezes}\} \]
\end{thmA}

\begin{proof}
Lemma \ref{lemma: Frinv} gives an explicit preimage of each $\pi^\dagger$-frieze, so $\Fr$ is surjective.


Next, we show that $\Fr$ is injective. Let $\mathsf{A}$ and $\mathsf{A}'$ be $\pi$-unimodular matrices with $\Fr(\mathsf{A}) = \Fr(\mathsf{A}')$.
Then $\mathrm{ker}(\Fr(\mathsf{A})) = \mathrm{ker}(\Fr(\mathsf{A}'))$, and so Lemma \ref{lemma: qpkernel} implies $\mathrm{ker}(\mathsf{A})=\mathrm{ker}(\mathsf{A}')$. Then there is an invertible $k\times k$-matrix $\mathsf{G}$ such that $\mathsf{GA}=\mathsf{A}'$.
Choosing any $a$,
\[ \det( \mathsf{A}'_{L_a}) = \det( (\mathsf{GA})_{L_a}) = \det(\mathsf{G}) \det( \mathsf{A}_{L_a}) \]
By $\pi$-unimodularity, $\det(\mathsf{G})=1$, and so $\mathsf{A}$ and $\mathsf{A}'$ are in the same $\mathrm{SL}(k)$-orbit.
\end{proof}

\vspace{0.5em}

\subsection{Duality and the twist (Theorem \ref{thm: dualities})}

In this section, we demonstrate the relation between the twist, the positive complement, and the dual of a frieze. 

\newpage

\begin{prop}
\label{prop: twistduality}
Let $\mathsf{A}$ be $\pi$-unimodular, and let $\mathsf{A}^\ddagger$ be a positive complement of $\mathsf{A}$. Then 
$\tau(\mathsf{A})$ is a positive complement to $\tau^{-1}(\mathsf{A}^\ddagger)$.
\end{prop}

\noindent Equivalently, we may write $\tau(\mathsf{A})^\ddagger \equiv \tau^{-1}(\mathsf{A}^\ddagger)$ (modulo left multiplication by $\mathrm{SL}(k)$).

\begin{proof}
\emph{This proof uses notation and terminology from \cite{MS17} that has not been introduced and won't be used elsewhere.}
%
%
Consider a reduced plabic graph $\Gamma$ with juggling function $\pi$, and assume there is an edge weighting $w\in (\k^\times)^{E}$ of $\Gamma$ whose image $\mathbb{D}(w)$ under the boundary measurement map is the rowspan of $\mathsf{A}$. By \cite[Theorem~7.1]{MS17},
\[ \overset{\bullet \rightarrow}{\mathbb{F}}(\tau(\mathsf{A})) = \overset{\bullet \rightarrow}{\mathbb{F}}\tau\mathbb{D}(w) = \overrightarrow{\mathbb{M}}(w) \]
where $\overset{\bullet \rightarrow}{\mathbb{F}}$ denote the source-labeled Pl\"ucker coordinates of $\Gamma$ and $\overrightarrow{\mathbb{M}}(w)$ is the face weighting of $\Gamma$ determined by the downstream wedges of the edges in $w$.

Let $\Gamma^{\dagger}$ denote the reduced plabic graph obtained by swapping the colors of the vertices of $\Gamma$. The edge weighting $w$ may be be identified with an edge weighting of $\Gamma^\dagger$, and the image $\mathbb{D}^\dagger(w)$ under the boundary measurement map of $\Gamma^\dagger$ is the rowspan of $\mathsf{A}^\ddagger$. By \cite[Theorem~7.1]{MS17},
\[ \overset{\bullet \leftarrow}{\mathbb{F}}^\dagger(\tau^{-1}(\mathsf{A}^\ddagger)) = \overset{\bullet \leftarrow}{\mathbb{F}}^\dagger\tau^{-1}\mathbb{D}^\dagger(w) = \overleftarrow{\mathbb{M}}^\dagger(w) \]
where $\overset{\bullet \leftarrow}{\mathbb{F}}^\dagger$ denote the target-labeled Pl\"ucker coordinates of $\Gamma^\dagger$ and $\overleftarrow{\mathbb{M}}^\dagger(w)$ is the face weighting of $\Gamma^\dagger$ determined by the upstream wedges of the edges in $w$. Since the upstream wedges in $\Gamma^\dagger$ are the downstream wedges in $\Gamma$, $\overleftarrow{\mathbb{M}}^\dagger(w) = \overrightarrow{\mathbb{M}}(w)$ and so
\[ \overset{\bullet \rightarrow}{\mathbb{F}}(\tau(\mathsf{A})) =  \overset{\bullet \leftarrow}{\mathbb{F}}^\dagger(\tau^{-1}(\mathsf{A}^\ddagger)) \] 
Since the target labellings of $\Gamma^\dagger$ are the complements of the source labellings of $\Gamma$, 
\[ \overset{\bullet \leftarrow}{\mathbb{F}}^\dagger(\tau^{-1}(\mathsf{A}^\ddagger)) = \overset{\bullet \rightarrow}{\mathbb{F}}(\tau^{-1}(\mathsf{A}^\ddagger)^\ddagger)  \] 
where $\tau^{-1}(\mathsf{A}^\ddagger)^\ddagger$ is a positive complement to $\tau^{-1}(\mathsf{A}^\ddagger)$. Since $\overset{\bullet \rightarrow}{\mathbb{F}}$ is bijective on the image of $\tau\circ \mathbb{D}$ \cite[Theorem~7.1]{MS17}, $\tau(\mathsf{A})$ is a positive complement of $\tau^{-1}(\mathsf{A}^\ddagger)$.


This verifies the proposition when the rowspan of $\mathsf{A}$ is in the image of the boundary measurement map of $\Gamma$. By \cite[Prop.~7.6]{MS17}, this set is dense in the entire positroid variety, and so the proposition holds in general by continuity.
\end{proof}

\begin{lemma}
\label{lemma: tilingentries}
If $\mathsf{A}$ is a $\pi$-unimodular $k\times n$-matrix, then for all $a,b\in \mathbb{Z}$, 
\[ \mathrm{Tiling}(\Fr(\mathsf{A}))_{a,b} = (-1)^{a+b+(k-1)\frac{a+b-\overline{a}-\overline{b}}{n}} \tau (\mathsf{A})_{\overline{a}} \cdot \mathsf{A}_{\overline{b}}  \]
\end{lemma}

\begin{proof}
Since $\Fr(\mathsf{A})$ is a $\pi^\dagger$-frieze with $(n-k)$-many balls, the $(a,b)$th entry of the tiling is defined as
\begin{equation}
\label{eq: frtilingsum}
 \mathrm{Tiling}(\Fr(\mathsf{A}))_{a,b}
= \sum_{i\in \mathbb{Z}}
(-1)^{a+b+i(k-1)} \Fr(\mathsf{A})_{a+in,b} 
\end{equation}
The latter sum can be evaluated using the cases in Corollary \ref{coro: twistcons}.

If $\pi^\dagger(a)<a+n$ (that is, $a$ is not a loop of $\pi$), then every term in the sum \eqref{eq: frtilingsum} must vanish except for the $i\in \mathbb{Z}$ such that $b\leq  a+in<b+n$.
Equivalently, $1\leq a-b+1+in \leq n$, and so 
$i = \frac{(a-b+1) - \overline{a-b+1}}{n} $.
Corollary \ref{coro: twistcons} gives the value of the remaining term as
\begin{align*}
\mathrm{Tiling}(\Fr(\mathsf{A}))_{a,b}
&= 
(-1)^{a+b+i(k-1)}\begin{cases}
\rt(\mathsf{A})_{\overline{a}} \cdot \mathsf{A}_{\overline{b}}  & \text{if $\overline{b}\leq \overline{a}$},\\
(-1)^{k-1} \rt(\mathsf{A})_{\overline{a}} \cdot \mathsf{A}_{\overline{b}} & \text{if $\overline{b}> \overline{a}$}.\\
\end{cases}
\end{align*}
If $\overline{b}\leq \overline{a}$, then $\overline{a-b+1}=\overline{a}-\overline{b}+1$, and so $i= \frac{(a-b)-(\overline{a}-\overline{b})}{n}$.
If $\overline{b}>\overline{a}$, $\overline{a-b+1}=(\overline{a}-\overline{b}+1)+n$, and so $i= \frac{(a-b)-(\overline{a}-\overline{b})}{n}-1$. Therefore, the lemma holds whenever $\pi^\dagger(a)<a+n$.

If $\pi^\dagger(a)=a+n$, then $\pi(a)=a$ and so $\mathsf{A}_{\overline{a}}=0$ and $\tau(\mathsf{A})_{\overline{a}})=0$.
If $b\not\equiv a$ mod $n$, then every term in the sum \eqref{eq: frtilingsum} vanishes.
If $b=a+in$ for some $i$, then the sum \eqref{eq: frtilingsum} has two terms, and 
\begin{align*}
\mathrm{Tiling}(\Fr(\mathsf{A}))_{a,b}
&= 
(-1)^{a+b+i(k-1)} \Fr(\mathsf{A})_{a+in,b} + (-1)^{a+b+(i+1)(k-1)} \Fr(\mathsf{A})_{a+(i+1)n,b}
\\ &=
(-1)^{a+b+i(k-1)} (1) + (-1)^{a+b+(i+1)(k-1)} (-1)^k = 0
\end{align*}
Therefore, the lemma holds whenever $\pi^\dagger(a)=a+n$.
%
%
%
\end{proof}

\begin{thmA}
\label{thm: dualities}
Let $\mathsf{A}$ be a $\pi$-unimodular matrix. Then
\[ \mathrm{F}(\mathsf{A})^\dagger = \mathrm{F}(\rt(\mathsf{A}) ^\pperp) = \mathrm{F}(\rt^{-1}(\mathsf{A} ^\pperp)) \]
\end{thmA}

\begin{proof}
%
%
Let $\mathsf{A}$ be a $\pi$-unimodular $k\times n$-matrix, and let $\mathsf{B}$ be the negation of the odd numbered columns of $\mathsf{A}$, so that the rows of $\mathsf{B}$ span the kernel of $\mathsf{A}^\ddagger$. By Prop.~\ref{prop: twistduality}, $\tau(\mathsf{A})$ is a positive complement of $\tau^{-1}(\mathsf{A}^\ddagger)$ and the rows of $\tau(\mathsf{B})$ span the kernel of $\tau^{-1}(\mathsf{A}^\ddagger)$. 

Given a $k\times n$-matrix $\mathsf{M}$, let $q_{(n-k,n)}(\mathsf{M})$ denote the $k\times \mathbb{Z}$-matrix whose rows are the $(n-k,n)$ superperiodic extensions of the rows of $\mathsf{M}$; that is,
\[ q_{(n-k,n)}(\mathsf{M})_{a,b} = (-1)^{(n-k-1)\frac{b-\overline{b}}{n}} \mathsf{M}_{a,\overline{b}} \]
Since $\tau^{-1}(\mathsf{A}^\ddagger)$ is $(n-k)\times n$,  by Lemma \ref{lemma: qpkernel}, the rows of $q_{(n-k,n)}(\tau(\mathsf{B}))$ span the kernel of $\Fr(\tau^{-1}(\mathsf{A}^\ddagger))$. 
By \cite[Theorem~9.6]{MulRes1}, the solution matrix of this frieze factors as
\[ \sol( \Fr(\tau(\mathsf{A}^\ddagger) ) ) = q_{(n-k,n)}(\tau(\mathsf{B}))^\top q_{(n-k,n)}(\mathsf{B}) \]
The $(a,b)$th entry of $ \sol( \Fr(\tau(\mathsf{A})^\ddagger) )$ is therefore
\begin{align*}
\left(q_{(n-k,n)}(\tau(\mathsf{B}))^\top q_{(n-k,n)}(\mathsf{B})\right)_{a,b}
&=
q_{(n-k,n)}(\tau(\mathsf{B}))_a\cdot q_{(n-k,n)}(\mathsf{B})_b \\
&= \left((-1)^{(n-k-1)\frac{a-\overline{a}}{n}}\tau(\mathsf{B})_{\overline{a}}\right)\cdot
\left((-1)^{(n-k-1)\frac{b-\overline{b}}{n}}\tau(\mathsf{B})_{\overline{b}}\right) \\
&= \left((-1)^{(n-k-1)\frac{a-\overline{a}}{n}+\overline{a}}\tau(\mathsf{A})_{\overline{a}} \right)\cdot
\left((-1)^{(n-k-1)\frac{b-\overline{b}}{n}+\overline{b}}\mathsf{A}_{\overline{b}} \right)
\\
&= (-1)^{(n-k-1)\frac{a+b-\overline{a}-\overline{b}}{n}+\overline{a}+\overline{b}} \left(\tau(\mathsf{A})_{\overline{a}} 
\cdot\mathsf{A}_{\overline{b}} \right)
\\
&= (-1)^{(k-1)\frac{a+b-\overline{a}-\overline{b}}{n}+{a}+{b}} \left(\tau(\mathsf{A})_{\overline{a}} 
\cdot\mathsf{A}_{\overline{b}} \right) 
\end{align*}
By Lemma \ref{lemma: tilingentries}, this coincides with the $(a,b)$th entry of $\mathrm{Tiling}(\Fr(\mathsf{A}))$. With Proposition \ref{prop: soltiling},
\[ \sol(\Fr(\tau^{-1}(\mathsf{A}^\ddagger))) = \mathrm{Tiling}(\Fr(\mathsf{A})) = \sol(\Fr(\mathsf{A})^\dagger) \]
Since the friezes $\Fr(\tau^{-1}(\mathsf{A}^\ddagger))$ and $\Fr(\mathsf{A})^\dagger$ have the same solution matrix, they have the same kernel by Lemma \ref{lemma: linrec}.4, and so they are equal by Lemma \ref{lemma: linrec}.1.
\end{proof}

\appendix

\section{Collected notational conventions}

We collect a number of notational conventions here for the reader's convenience.

\begin{itemize}
    \item \high{$\k$} denotes an arbitrary choice of field.
    \item For an integer $n>0$, \high{$[n]$} denotes the interval of consecutive integers from $1$ to $n$; that is,
    \[ [n] := \{ 1,2,3,...,n\} \]
    \item For $a,b\in \mathbb{Z}$, \high{$[a,b]$} denotes the interval of consecutive integers from $a$ to $b$; that is,
    \[ [a,b] := \{ a,a+1,a+2,...,b\} \]
    Similarly, \high{$[a,b)$}, \high{$(a,b]$}, and \high{$(a,b)$} denote consecutive integers save one or both of $a$ and $b$.
    \item When a choice of $n$ is clear from context, we use the following notation.
    \begin{itemize}
        \item For each $a\in\mathbb{Z}$, \high{$\overline{a}$} denotes the unique integer in $[n]$ congruent to $a$ mod $n$.
        \item For each $I\subset \mathbb{Z}$, \high{$\overline{I}$} denotes the set of integers in $[n]$ congruent to elements of $I$ mod $n$.
    \end{itemize}
    \item When $\mathsf{A}$ is a $k\times n$-matrix, we use the following notation.
    \begin{itemize}
        \item Given $a\in [n]$, \high{$\mathsf{A}_a$} denotes the $a$th column of $\mathsf{A}$.
        \item Given $a\in \mathbb{Z}$, \high{$\mathsf{A}_a$} denotes the $\overline{a}$th column of $\mathsf{A}$.
        \item Given $I\subset [n]$, \high{$\mathsf{A}_I$} denotes the $k\times |I|$-submatrix of $\mathsf{A}$ whose columns are indexed by $I$.
        \item Given $I\subset \mathbb{Z}$, \high{$\mathsf{A}_I$} denotes the $k\times |\overline{I}|$-submatrix of $\mathsf{A}$ whose columns are indexed by $\overline{I}$.
        
        \emph{Warning:} The columns of $\mathsf{A}_I$ are ordered according to the indices order in $\overline{I}$, not in $I$.
        
    \end{itemize}
\end{itemize}

\subsection*{Acknowledgements}

This paper owes its existence to conversations with two people.
\begin{itemize}
    \item David Speyer, with whom the second author discovered several curious properties of the twist; most notably the $\mathrm{GL}(k)$-invariance of the function $\mathsf{A}\mapsto \tau(\mathsf{A})^\top \mathsf{A}$.
    \item Karin Baur, whose talks on their work in \cite{BFGST18a,BFGST21} connected these curious properties to friezes and inspired this project.
\end{itemize}
We are also grateful to Emily Gunawan and Khrystyna Serhiyenko for helpful conversations.



\end{document}